\documentclass[11pt]{5191}
\usepackage{amssymb,latexsym}
\usepackage{amssymb,latexsym}
\usepackage[colorlinks]{hyperref}

 \newtheorem{thm}{Theorem}[section]
 \newtheorem{cor}[thm]{Corollary}
 \newtheorem{lem}[thm]{Lemma}
 
 \theoremstyle{definition}
 
 \theoremstyle{remark}
 \newtheorem{rem}[thm]{Remark}
 \newtheorem{ex}[thm]{Example}
 \numberwithin{equation}{section}
\begin{document}
\hyphenation{Leib-niz}

\font\eightrm=cmr8
\font\smallbf=ptmbo at 10pt
\font\smallit=ptmri at 10pt
\font\tensans=cmss10 %`
%\font\eightsans=cmss10 at 8pt %recover-file
\font\smallsmallit=ptmri at 8pt
\font\smallmedit=ptmri at 9pt
\font\teneuf=eufm10
\font\seveneuf=eufm7
\font\fiveeuf=eufm5
\font\tenmsa=msam10
\font\sevenmsa=msam7
\font\fivemsa=msam5
\font\tenmsb=msbm10
\font\sevenmsb=msbm7
\font\fivemsb=msbm5
\newfam\euffam
\textfont\euffam=\teneuf
\scriptfont\euffam=\seveneuf
\scriptscriptfont\euffam=\fiveeuf
\def\hexnumber@#1{\ifcase#1 0\or1\or2\or3\or4\or5\or6\or7\or8\or9\or
	A\or B\or C\or D\or E\or F\fi }
\newfam\msafam
\newfam\msbfam
\textfont\msafam=\tenmsa  \scriptfont\msafam=\sevenmsa
  \scriptscriptfont\msafam=\fivemsa
\textfont\msbfam=\tenmsb  \scriptfont\msbfam=\sevenmsb
  \scriptscriptfont\msbfam=\fivemsb
\edef\msb@{\hexnumber@\msbfam}
%\mathchardef\smallsetminus="2\msb@72
\mathchardef\varGamma="0100
%\mathchardef\Delta="0101
\mathchardef\varTheta="0102
\mathchardef\varLambda="0103
\mathchardef\varXi="0104
\mathchardef\varPi="0105
\mathchardef\varSigma="0106
\mathchardef\varUpsilon="0107
\mathchardef\varPhi="0108
%\mathchardef\varPsi="0109
\mathchardef\varOmega="010A
\def\checked{}
\def\td{\tilde}
\def\dvw{\mathrm{div}\hskip1.4ptw}
\def\dvv{\mathrm{div}\hskip1.4ptv}
\def\ns{\nabla\hskip-1.7pt_}
\def\nsu{\nabla_{\hskip-2.2ptu}}
\def\nsv{\nabla_{\hskip-2.2ptv}}
\def\nsc{\nabla_{\hskip-2.2ptc}}
\def\nsvu{\nsv\hh u}
\def\nsvw{\nsv\hn w}
\def\nsuv{\nsu\hh v}
\def\nsuw{\nsu\hh w}
\def\nsw{\nabla_{\hskip-2ptw}}
\def\nswu{\nsw u}
\def\nswv{\nsw v}
\def\nsww{\nsw w}
\def\nav{\nabla\hskip-.4ptv}
\def\naw{\nabla\nh w}
\def\bna{\hs\overline{\nh\nabla\nh}\hs}
\def\rc{\hbox{$\mathrm{R}\hskip-6.5pt^{^{^{_{\circ}}}}$}{}}
\def\rw{\hbox{$\mathrm{R}\hskip-6.5pt^{^{^{_{\wedge}}}}$}{}}
\def\lx{\mathtt{[}}
\def\rx{\mathtt{]}}
\def\bn{\mathbf{b}}
\def\cn{\mathbf{c}}
\def\jt{\text{\tt j}}
\def\jk{\jt^{\nh k}}
\def\jm{\jt^{\nh k-1}}
\def\ca{\nnh\circledast\nnh}
\def\e{E}
\def\bbF{\mathrm{I\!F}}
\def\bbH{\mathrm{I\!H}}
\def\bbK{\mathrm{I\!K}}
\def\bbR{\mathrm{I\!R}}%\def\bbR{\mathbf{R}}
\def\rn{{\bbR}^{\hskip-.6ptn}}
\newcommand{\bbC}{{\mathchoice {\setbox0=\hbox{$\displaystyle\mathrm{C}$}
\hbox{\hbox to0pt{\kern0.4\wd0\vrule height0.9\ht0\hss}\box0}} 
{\setbox0=\hbox{$\textstyle\mathrm{C}$}\hbox{\hbox 
to0pt{\kern0.4\wd0\vrule height0.9\ht0\hss}\box0}} 
{\setbox0=\hbox{$\scriptstyle\mathrm{C}$}\hbox{\hbox 
to0pt{\kern0.4\wd0\vrule height0.9\ht0\hss}\box0}} 
{\setbox0=\hbox{$\scriptscriptstyle\mathrm{C}$}\hbox{\hbox 
to0pt{\kern0.4\wd0\vrule height0.9\ht0\hss}\box0}}}}
\def\dimr{\dim_{\hskip-.4pt\bbR\hskip-1.2pt}^{\phantom j}}
\def\dimc{\dim_{\hskip-.4pt\bbC\hskip-1.2pt}^{\phantom j}}
\def\dimf{\dim_{\hskip-.4pt\bbF\hskip-1.2pt}^{\phantom j}}
\def\dimh{\dim_{\hskip.4pt\bbH\hskip-1.2pt}^{\phantom i}}
\def\dr{d_{\hskip.1pt\bbR}^{\phantom j}}
\def\df{d_{\hskip.1pt\bbF}^{\phantom j}}
\def\hyp{\hskip.5pt\vbox
{\hbox{\vrule width2.5ptheight0.5ptdepth0pt}\vskip2pt}\hskip.5pt}
\def\f{f}
\def\h{h}
\def\k{j}
\def\d{l}
\def\m{m}
\def\n{n}
\def\p{p}
\def\q{q}
\def\r{r}
\def\x{x}
\def\y{y}
\def\z{z}
\def\K{L}
\def\U{U}
\def\V{V}
\def\W{W}
\def\ct{\kappa}
\def\aj{\mathrm{A}}
\def\bj{\mathrm{B}}
\def\dj{\mathrm{D}}
\def\ej{\mathrm{E}}
\def\gj{\mathrm{G}}
\def\jj{\mathrm{J}}
\def\kj{\mathrm{K}}
\def\mj{\mathrm{M}}
\def\nj{\mathrm{N}}
\def\sj{\mathrm{S}}
\def\by{\beta}
\def\bya{\by_a\w}
\def\byab{\by_{a,b}\w}
\def\zy{\zeta}
\def\eya{\eta_a\w}
\def\cy{\theta}
\def\cya{\cy\hskip-1pt_a\w}
\def\cyab{\cy\hskip-1pt_{a,b}\w}
\def\cyc{\cy\hskip-1pt_c\w}
\def\cyac{\cy\hskip-1pt_{a,c}\w}
\def\iy{\iota}
\def\iyp{\iy\nh_{a,b}^+}
\def\iym{\iy\nh_{a,b}^-}
\def\ipm{\iy\nh_{a,b}^\pm}
\def\ly{\lambda}
\def\my{\mu}
\def\mya{\my_a\w}
\def\myab{\my_{a,b}\w}
\def\myc{\my_c\w}
\def\myac{\my_{a,c}\w}
\def\myad{\my_{a,d}\w}
\def\sy{\sigma}
\def\sya{\sy\hskip-1.8pt_a\w}
\def\ty{\tau}
\def\tya{\ty\hskip-1.8pt_a\w}
\def\tyb{\ty_{\nh b}\w}
\def\tyc{\ty\hskip-1.8pt_c\w}
\def\tyd{\ty\hskip-.8pt_d\w}
\def\tye{\ty\hskip-1.8pt_e\w}
\def\fy{\psi}
\def\xy{\xi}
\def\ap{\alpha}
\def\ch{\Phi}
\def\pr{\Pi}
\def\dz{\Delta}
\def\cu{\Omega}
\def\ea{\mathcal{A}}
\def\eb{\mathcal{B}}
\def\ec{\mathcal{C}}
\def\ee{\mathcal{E}}
\def\ef{\mathcal{F}}
\def\el{\mathcal{L}}
\def\em{\mathcal{M}}
\def\en{\mathcal{N}}
\def\ep{\mathcal{P}}
\def\eq{\mathcal{Q}}
\def\es{\mathcal{S}}
\def\et{\mathcal{T}}
\def\eu{\mathcal{U}}
\def\ev{\mathcal{V}}
\def\ew{\mathcal{W}}
\def\ey{\mathcal{Y}}
\def\ez{\mathcal{Z}}
\def\fz{\text{\smallbf F}}
\def\hz{\text{\smallbf H}}
\def\iz{\text{\smallbf I}}
\def\jz{\text{\smallbf J}}
\def\kz{\text{\smallbf K}}
\def\lz{\text{\smallbf L}}
\def\mz{\text{\smallbf M}}
\def\pz{\text{\smallbf P}}
\def\qz{\text{\smallbf Q}}
\def\w{^{\phantom i}}
\def\s{^{\phantom j}}
\def\hs{\hskip.7pt}
\def\hh{\hskip.4pt}
\def\hn{\hskip-.4pt}
\def\nh{\hskip-.7pt}
\def\nnh{\hskip-1.5pt}
\def\vg{\varGamma}
\def\t{{t}}
\def\gy{\gamma}
\def\sy{{\sigma}}
\def\vd{\delta}
\def\ve{\varepsilon}
\def\ro{\hbox{$R\hskip-4.5pt^{^{^{_{\circ}}}}$}{}}

\renewcommand{\theequation}{\thesection.\arabic{equation}}

\title[Indefinite Ein\-stein metrics]{Indefinite Ein\-stein metrics on simple 
Lie groups}
\author[A. Derdzinski]{Andrzej Derdzinski} %\large 
\address{Department of Mathematics, The Ohio State University, Columbus, 
\hbox{OH 43210,} USA}
\email{andrzej@math.ohio-state.edu}
\author[\'S.\ R.\ Gal]{\'Swiatos\l aw R.\ Gal} %\large 
\address{Mathematical Institute, \hskip-2ptWroc\l aw University, pl. 
\hskip-1ptGrunwaldzki 2/\nnh4, \hbox{50-384 \hskip-.7ptWroc\l aw,} Poland}
\email{Swiatoslaw.Gal@math.uni.wroc.pl}
\dedicatory{\smallit\hskip-22.5pt Dedicated to Professor Witold Roter on the 
occasion of his eightieth birthday}
\begin{abstract}The set $\mathcal{E}$ of Levi-Civita connections of 
left-invariant pseudo-Riemannian Einstein metrics on a given semisimple Lie 
group always includes D, the Levi-Civita connection  of 
the Killing form. For the groups SU($l,j$) (or SL($n,\bbR$), or SL($n,\bbC$) 
or, if $n$ is even, SL($n/2,\bbH$)), with $0\le j\le l$ and $\,j+l>2$ (or, 
$n>2$), we explicitly describe the connected component $\mathcal{C}$ of 
$\mathcal{E}$, containing D. It turns out that $\mathcal{C}$, a 
relatively-open subset of $\mathcal{E}$, is also an algebraic variety of real 
dimension $2lj$ (or, real/complex dimension $[n^2\!/2]$ or, respectively, 
real dimension $4[n^2\!/8]$), forming a union of 
$(j+1)(j+2)/2$ (or, $[n/2]+1$ or, respectively, $[n/4]+1$) orbits of the 
adjoint action. In the case of SU($n$) one has $2lj=0$, so that a 
pos\-i\-tive-def\-i\-nite multiple of the Killing form is isolated among 
suitably normalized left-in\-var\-i\-ant Riemannian Ein\-stein metrics on 
SU($n$).
\end{abstract}

\subjclass{53C30, 53C50, 22E99}

\keywords{Indefinite Ein\-stein metrics, left-in\-var\-i\-ant 
Ein\-stein metrics}

\maketitle

\setcounter{thm}{0}
\section{Introduction}\label{in}
\setcounter{equation}{0}
The Kil\-ling form $\,\by\,$ of any sem\-i\-sim\-ple real Lie group $\,\gj\,$ 
is a bi-in\-var\-i\-ant pseu\-\hbox{do\hskip1pt-}\hskip0ptRiem\-ann\-i\-an 
Ein\-stein metric on the underlying manifold of $\,\gj$. In the case of 
compact simple groups $\,\gj\,$ other than $\,\mathrm{SU}\hh(2)\,$ and 
$\,\mathrm{SO}\hh(3)$, D'Atri and Ziller \cite{d'atri-ziller} proved, over 
three decades ago, the existence on $\,\gj\,$ of at least one 
left-in\-var\-i\-ant Riemannian Ein\-stein metric which is not a multiple of 
$\,\by$. Recently, Gibbons, L\"u and Pope \cite{gibbons-lu-pope} found six 
more such essentially different examples -- two on $\,\mathrm{SO}\hh(5)$ and 
four on $\,\gj_2\w$. In addition, among left-in\-var\-i\-ant Ein\-stein 
metrics on $\,\gj=\mathrm{SU}\hh(3)$, they exhibited one which is {\smallit 
indefinite}.% See also \cite{chen-liang}.

This raises the more general question of
\begin{enumerate}
  \def\theenumi{{\rm\roman{enumi}}}
\item[{\rm($*$)}] {\smallit classifying left-in\-var\-i\-ant 
pseu\-\hbox{do\hh-}\hskip0ptRiem\-ann\-i\-an Ein\-stein metrics on simple Lie 
groups}.
\end{enumerate}
For noncompact groups, question ($*$) is in turn related to a conjecture of 
Alek\-se\-ev\-sky \cite{alekseevsky}, cf.\ \cite[p.\ 190]{besse}, according to 
which a noncompact homogeneous space $\,\gj/\kj\,$ may carry a non\-flat 
$\,\gj$-in\-var\-i\-ant Riemannian Ein\-stein metric only if $\,\kj\,$ is a 
maximal compact subgroup of $\,\gj$. If true, Alek\-se\-ev\-sky's conjecture, 
applied to a noncompact simple group $\,\gj\,$ and its trivial subgroup 
$\,\kj$, would imply that left-in\-var\-i\-ant Ein\-stein metrics on $\,\gj\,$ 
are all indefinite. 

The Ein\-stein condition is essential here, as opposed to mere negativity 
of the Ric\-ci curvature $\,\mathrm{Rc}\hs$: in fact, Leite and Dotti 
\cite{leite-dotti-miatello} constructed (non-Ein\-stein) left-in\-var\-i\-ant 
Riemannian metrics with $\,\mathrm{Rc}<0\,$ on $\,\mathrm{SL}(n,\bbR)$, 
$\,n\ge3$. Also, all noncompact Lie groups known to admit left-in\-var\-i\-ant 
Riemannian Ein\-stein metrics are solvable. Heber \cite{heber} and Lauret 
\cite{lauret} made substantial progress towards understanding the moduli and 
properties of such metrics on solvable Lie groups.

Given a Lie group $\,\gj\,$ with the Lie algebra $\,\mathfrak{g}$, let 
$\,\ee\,$ denote the set of all {\smallit Einstein connections in\/} 
$\,\mathfrak{g}$, by which we mean the Le\-\hbox{vi\hh-}\hskip0ptCi\-vi\-ta 
connections of left-in\-var\-i\-ant 
pseu\-\hbox{do\hskip1pt-}\hskip0ptRiem\-ann\-i\-an Ein\-stein metrics on 
$\,\gj$. One may view $\,\ee\,$ as a set of bi\-lin\-e\-ar operations 
$\,\mathfrak{g}\times\mathfrak{g}\to\mathfrak{g}$, that is, a subset of 
$\,[\mathfrak{g}\nh^*]^{\otimes2}\nnh\otimes\mathfrak{g}$. If $\,\mathfrak{g}\,$ 
is sem\-i\-sim\-ple, the Le\-\hbox{vi\hh-}\hskip0ptCi\-vi\-ta connection 
$\,\dj=[\hskip2pt,\hskip.6pt]/2\,$ of the Kil\-ling form $\,\by\,$ is an 
element of $\,\ee\nh$.

The present paper provides an initial step towards answering the question 
stated above in ($*$): for any simple Lie algebra $\,\mathfrak{g}$, we 
explicitly describe the connected component $\,\ec\,$ of $\,\ee\,$ containing 
$\,\dj$. As $\,\ec\,$ turns out to be a relatively open subset of $\,\ee\nh$, it 
also contains all Ein\-stein connections sufficiently close to $\,\dj$.

We restrict much of our discussion to $\,\mathfrak{g}\,$ that correspond to 
the groups
\begin{equation}\label{gsu}
\gj=\mathrm{SL}\hh(\n,\bbR),\,\gj=\mathrm{SL}\hh(\n,\bbC),\,\gj=
\mathrm{SU}\hh(\d,\k),\mathrm{\ or\ }(\n\,\mathrm{\ even})\,\hs\gj
=\mathrm{SL}\hh(\n/2,\bbH),
\end{equation}
with $\,\d\ge\k\ge0\,$ and $\,\n=\d+\k\ge3\,$ since, for all remaining simple 
Lie algebras, $\,\dj$ is isolated in $\,\ee$ and $\,\ec=\{\dj\}$. (See 
Remark~\ref{kroez}.) Our main result, Theorem~\ref{mnres}, realizes $\,\ec$, 
for $\,\gj\,$ in (\ref{gsu}), as the bijective image of a well-un\-der\-stood 
algebraic set in $\,\mathfrak{g}\,$ under a specific nonhomogeneous quadratic 
mapping $\,\mathfrak{g}\to[\mathfrak{g}\nh^*]^{\otimes2}\nnh\otimes\mathfrak{g}$. 
Theorem~\ref{mnres} also states that $\,\ec\,$ itself is an algebraic set 
of dimension $\,\df$ over $\,\bbF\nnh$, consisting of $\,s\,$ orbits of the 
adjoint action, where $\hs(\df,s)=([\n^2\nnh/2],\hs[\n/2]+1)\hs$ for 
$\,\gj=\mathrm{SL}\hh(\n,\bbF)\,$ and  
$\,(\dr,s)=([\n/2]+1,\hs(\k+1)(\k+2)/2)\,$ if $\,\gj=\mathrm{SU}\hh(\d,\k)$, 
while $\,(\dr,s)=(4[n^2\nnh/8],\hs[\n/4]+1)\,$ when 
$\,\gj=\mathrm{SL}\hh(\n/2,\bbH)$.

As a consequence (Theorem~\ref{isola}), on $\,\mathrm{SU}\hh(\n)\hh$, 
pos\-i\-tive-def\-i\-nite multiples of the Kil\-ling form are isolated among 
left-in\-var\-i\-ant Riemannian Ein\-stein metrics. This is a special case of 
a conjecture made by B\"ohm, Wang and Ziller \cite[p.\ 683]{bohm-wang-ziller}.

For the group $\,\gj=\mathrm{SL}\hh(\n,\bbC)\,$ our argument leads to the 
following conclusion (Theorem~\ref{slcre}): all left-in\-var\-i\-ant 
pseu\-\hbox{do\hskip1pt-}\hskip0ptRiem\-ann\-i\-an Ein\-stein metrics on 
$\,\gj\,$ close to multiples of the Kil\-ling form are real parts of 
hol\-o\-mor\-phic Ein\-stein metrics, cf.\ \cite[p.\ 210]{lebrun}, which also 
makes them {\smallit K\"ah\-ler-Nor\-den metrics\/} in the sense of 
\cite{olszak}.

\section{Outline of the main argument}\label{oa} 
\setcounter{equation}{0} 
Let $\,\gj\,$ be a real or complex sem\-i\-sim\-ple Lie group, with the 
associated sem\-i\-sim\-ple Lie algebra $\,\mathfrak{g}\,$ over 
$\,\bbF=\bbR\,$ or $\,\bbF=\bbC$. We study left-in\-var\-i\-ant connections 
on $\,\gj$, which are also hol\-o\-mor\-phic when $\,\bbF=\bbC$, as well as 
left-in\-var\-i\-ant metrics on $\,\gj$, assuming the latter to be 
pseu\-\hbox{do\hh-}\hskip0ptRiem\-ann\-i\-an if $\,\bbF=\bbR$, and 
hol\-o\-mor\-phic (in the sense of being $\,\bbC$-bi\-lin\-e\-ar, symmetric 
and nondegenerate) if $\,\bbF=\bbC$.

In this section we skip the term `left-in\-var\-i\-ant' and simply speak of 
{\smallit connections in\/} $\,\mathfrak{g}\,$ and {\smallit metrics in\/} 
$\,\mathfrak{g}$. An example of the latter is  the Kil\-ling form $\,\by$.

All vector spaces under consideration are 
fi\-\hbox{nite\hh-}\hskip0ptdi\-men\-sion\-al. The spaces
\[%begin{equation}\label{spc}
\et,\hskip12pt\ey,\hskip12pt\es\hh,\hskip12pt\mathrm{with}\hskip8pt
\es\,\subset\,\ey
\]%end{equation}
consist, respectively, of symmetric $\,\bbF$-bi\-lin\-e\-ar forms 
$\,\mathfrak{g}\times\mathfrak{g}\to\bbF\,$ (such as metrics in 
$\,\mathfrak{g}$), arbitrary connections in $\,\mathfrak{g}\,$ (treated as 
$\,\bbF$-bi\-lin\-e\-ar operations 
$\,\mathfrak{g}\times\mathfrak{g}\to\mathfrak{g}$) and, finally, the 
operations $\,\mathfrak{g}\times\mathfrak{g}\to\mathfrak{g}\,$ in $\,\ey\,$ 
which are symmetric. Thus, $\,\by\in\et$ for the Kil\-ling form $\,\by$, 
and $\,\dj\in\ey$, where $\,\dj=[\hskip2pt,\hskip.6pt]/2\,$ is the {\smallit 
standard connection}. Our discussion focuses on the af\-fine subspace 
$\,\dj+\es\,$ of $\,\ey\,$ formed by connections $\,\nabla\in\ey\,$ which are 
tor\-sion-free, so that their skew-sym\-met\-ric part is $\,\dj$. We refer to 
any such $\,\nabla\nh=\dj+\hs\sj$, with $\,\sj\in\es$, as a {\smallit 
weak\-\hbox{ly\hh-}\hskip0ptEin\-stein connection\/} in $\,\mathfrak{g}$, if the symmetric 
$\,2$-ten\-sor $\,\{\nabla\hskip-2.5pt\cdot\hskip-2.5pt\nabla\}\,$ defined 
in Section~\ref{ts}, is $\,\nabla\nnh$-par\-al\-lel, that is,
\begin{equation}\label{nnn}
\nabla\{\nabla\hskip-2.5pt\cdot\hskip-2.5pt\nabla\}\,=\,0\hs.
\end{equation}
Since $\,\mathfrak{g}\,$ is sem\-i\-sim\-ple, whenever $\,\nabla\,$ happens to 
be the Le\-\hbox{vi\hh-}\hskip0ptCi\-vi\-ta connection of a metric, its 
Ric\-ci tensor is 
$\,\rho\nnh^\nabla\nnh=-\{\nabla\hskip-2.5pt\cdot\hskip-2.5pt\nabla\}$. The 
set $\,\ew\,$ of all weak\-\hbox{ly\hh-}\hskip0ptEin\-stein connections in $\,\mathfrak{g}\,$ 
thus contains the set $\,\ee\,$ of {\smallit Einstein connections\/} in 
$\,\mathfrak{g}$, defined here to be the Le\-\hbox{vi\hh-}\hskip0ptCi\-vi\-ta 
connections $\,\nabla\,$ of (left-in\-var\-i\-ant) {\smallit Ein\-stein 
metrics\/} $\,\gy\,$ on $\,\gj$. By the latter we mean, as usual, all 
$\,\gy\,$ with $\,\rho\nnh^\nabla\nnh=\ct\hs\gy\,$ for some $\,\ct\in\bbF\nh$.

For instance, $\,\dj\,$ is an Einstein connection, being the 
Le\-\hbox{vi\hh-}\hskip0ptCi\-vi\-ta connection of the Ein\-stein metric 
$\,\by\,$ (the Kil\-ling form).

The converse ``weak\-\hbox{ly\hh-}\hskip0ptEin\-stein implies Ein\-stein'' of 
the implication ``Ein\-stein implies weak\-\hbox{ly\hh-}\hskip0ptEin\-stein'' 
($\ee\subset\ew$), generally false (Section~\ref{ne}), is true {\smallit 
generically\/}:
\begin{equation}\label{nde}
\mathrm{if\ }\,\nabla\in\ew\,\mathrm{\ and\ }\,
\{\nabla\hskip-2.5pt\cdot\hskip-2.5pt\nabla\}\,\mathrm{\ is\ 
nondegenerate,\ then\ }\,\nabla\in\ee\hh.
\end{equation}
In fact, $\,\nabla\,$ then is the Le\-\hbox{vi\hh-}\hskip0ptCi\-vi\-ta connection of the metric 
$\,\{\nabla\hskip-2.5pt\cdot\hskip-2.5pt\nabla\}$, with the Ric\-ci tensor 
$\,\rho\nnh^\nabla\nnh=-\{\nabla\hskip-2.5pt\cdot\hskip-2.5pt\nabla\}$.

Instead of Ein\-stein metrics, we choose to look for weak\-\hbox{ly\hh-}\hskip0ptEin\-stein 
connections in $\,\mathfrak{g}$. By (\ref{nde}), the latter question is only 
slightly more general than the former. It can, however, be phrased in much 
simpler algebraic terms: for $\,\hz:\es\to\es\,$ given by 
$\,\hz\hs(\sj)=4\hs(\dj+\hs\sj)\hh\{(\dj+\hs\sj)\nnh\cdot\nnh(\dj+\hs\sj)\}\,$ 
(the left-hand side in (\ref{nnn}) with $\,\nabla\nh=\dj+\hs\sj$),
\[%begin{equation}\label{ewn}
\ew\,=\,\{\nabla\in\dj
+\es:\nabla\{\nabla\hskip-2.5pt\cdot\hskip-2.5pt\nabla\}\,=\,0\}\,
=\,\dj\,+\,\hz\hs^{-\nh1}(0)\hs,
\]%end{equation}
while $\,\hz\,$ is a (nonhomogeneous) cubic polynomial mapping, due to 
the quadratic dependence of 
$\,\{\nabla\hskip-2.5pt\cdot\hskip-2.5pt\nabla\}\,$ on $\,\nabla\nnh$.

Enlarging the space of unknowns, we rewrite the cubic condition 
$\,\hz\hs(\sj)=0$ as a system of quadratic equations. Specifically 
(Remark~\ref{kssez}), there exists
\begin{equation}\label{bij}
\begin{array}{l}
\mathrm{a\ bijective\ correspondence\ between\ }\,\,\ew\hs=\,\dj\hs+\hs\hz\hs^{-\nh1}(0)\\
\mathrm{and\ the\ set\ of\ all\ }\,(\sj,\sy)\in\es\times\et\mathrm{\ 
with\ }\,\kz\hh(\sj,\sy)=(0,0)\hh,
\end{array}
\end{equation}
where $\,\kz:\es\times\et\to\es\times\et\,$ is a suitable nonhomogeneous 
quadratic mapping. 

Thus, since weak\-\hbox{ly\hh-}\hskip0ptEin\-stein connections (elements of $\,\ew$) are 
precisely those $\,\dj+\hs\sj\,$ for which $\,\sj\in\es\,$ and 
$\,\hz\hs(\sj)=0$,
\begin{equation}\label{slv}
\begin{array}{l}
\mathrm{the\ search\ for\ weak\-ly}\hyp\mathrm{Ein\-stein\ connections\ 
}\,\hs\dj\hs+\,\sj\hs\,\mathrm{\ is\ reduced}\\
\mathrm{to\ solving\ the\ equation\ }\,\kz\hh(\sj,\sy)=(0,0)\,\mathrm{\ for\ }\,
\sj\in\es\,\mathrm{\ and\ }\,\sy\in\et\nnh.
\end{array}
\end{equation}
From now on, $\,\mathfrak{g}\,$ is assumed to be one of the simple Lie algebras
\[%begin{equation}\label{sln}
\mathfrak{sl}\hh(\n,\bbR),\hskip5pt\mathfrak{sl}\hh(\n,\bbC),\hskip5pt
\mathfrak{su}\hh(\d,\k)\,\mathrm{\ with\ }\,\d\ge\k\ge0\,\mathrm{\ and\ 
}\,\n=\d+\k\ge3,\mathrm{\ or\ }\,\mathfrak{sl}\hh(\n/2,\bbH)\hh,
\]%end{equation}
the last for even $\,\n\,$ only. As a first step towards constructing 
$\,(\sj,\sy)\in\es\times\et\,$ with $\,\kz\hh(\sj,\sy)=(0,0)$, we set 
$\,(\sj,\sy)=(\dj\hh\ly,\hs\fy)$, where $\,\ly,\fy\,$ are certain special 
elements of $\,\et\,$ and $\,\dj\hh\ly\,$ is the $\,\dj$-co\-var\-i\-ant 
derivative of $\,\ly\,$ (treated as an element of $\,\es\,$ via 
$\,\by$-in\-dex-rais\-ing). More precisely, we define in Section~\ref{sl} some 
$\,\tya,\cya,\mya\in\et\nnh$, which depend on $\,a\in\mathfrak{g}$, then 
choose $\,\ly=\x\hh\tya+\n^2\y\hh\cya
+\z\hh\mya$ and $\,\fy=\p\hh\tya+\q\hh\cya+\r\hh\mya+\f\nh\by$ with 
(unknown) scalar parameters $\,\f\nh,\x,\y,\z,\p,\q,\r$. In other words,
\begin{equation}\label{iav}
\begin{array}{l}
(\dj\hh\ly,\fy)=\iz_a\w[\mathbf{v}]\hs\mathrm{\ for\ the\ linear\ operator\ 
}\hs\iz_a\w\hskip-2pt:\hskip-1pt\bbF^9\nnh
\to\hn\es\nh\times\nnh\et\hn\mathrm{\ given\ by}\\
\iz_a\w[\mathbf{v}]\,\,
=\,\,(\x\hh\dj\hh\tya\,\hs+\hs\,\n^2\y\hh\dj\hh\cya\,\hs+\hs\,\z\hh\dj\hh\mya,
\,\,\p\hh\tya\,\hs+\hs\,\q\hh\cya\,\hs+\hs\,\r\hh\mya\,\hs+\hs\,\f\nh\by)
\end{array}
\end{equation}
and any $\,\mathbf{v}=(\xy,\f\nh,\h,\x,\y,\z,\p,\q,\r)\in\bbF^9\nnh$. The 
variables $\,\xy\,$ and $\,\h$, not occurring in the formula, become useful later.
Note that, if $\,\mathbf{v}=(\xy,\f\nh,\h,\x,\y,\z,\p,\q,\r)$,
\begin{equation}\label{iaz}
\iz_b\w[\mathbf{v}]\,=\,\mathbf{0}\hskip8pt\mathrm{whenever}\hskip6ptb=0
\hskip6pt\mathrm{and}\hskip6pt\f\nh=0\hs.
\end{equation}
To find $\,a,\nh\mathbf{v}\,$ with $\,\kz\hh(\iz_a\w[\mathbf{v}])=(0,0)$, we 
prove in Section~\ref{sa} the {\smallit fundamental formula}
\begin{equation}\label{fin}
\kz\hh(\iz_a\w[\mathbf{v}])\,-\,\iz_a\w[\hh\jz\hh(\nh\mathbf{v}\nh)]\,
-\,\iz_b\w[\hh\mz\hh(\nh\mathbf{v}\nh)]\,\,\in\,\,\en
\end{equation}
for any $\,(a,\nh\mathbf{v})\in\mathfrak{g}\times\bbF^9\nnh$, where 
$\,\jz,\mz\,$ are specific polynomial mappings $\,\bbF^9\to\hs\bbF^9\nnh$. 
(Remark~\ref{polym} describes $\,\jz\,$ and $\,\mz$, while (\ref{fin}) 
appears as (\ref{med}) -- (\ref{zap}).)

What $\,\en\,$ and $\,\in\,$ in (\ref{fin}) stand for requires the following, 
more detailed explanation. First, we treat the left-hand side of (\ref{fin}) 
as a polynomial function of 
\begin{equation}\label{abg}
\mathrm{the\ variables\ }\,\,(a,b,\nh\mathbf{v})
=(a,b,\xy,\f\nh,\h,\x,\y,\z,\p,\q,\r)\,
\in\,\mathfrak{g}\times\mathfrak{g}\times\bbF^9
\end{equation}
with values in $\,\es\times\et$. This is achieved by replacing all 
occurrences of
\begin{equation}\label{aqe}
\begin{array}{l}
a\nh^2\mathrm{\ with\ }\,\ve^{-1}(\h a+b)+\ve^{-2}\xy\hs\mathrm{Id}\hh,
\mathrm{\ which\ in\ turn\ leads\ to\ replacements\ of}\\
\mathrm{tr}\hskip2pta\nh^2\hs\mathrm{\ by\ }\,\n\hs\ve^{-2}\xy\hh,\hs
\mathrm{\ and\ }\,a\nh^3\hs\mathrm{\ by\ }
\,\ve^{-1}ba+\ve^{-2}(\h^2\nh+\xy)\hh a+\ve^{-2}\h b
+\ve^{-3}\xy\h\hskip1.2pt\mathrm{Id}\hh.
\end{array}
\end{equation}
(Here $\,\ve\in\{\hn1,\mathrm{i}\hs\}\,$ is a parameter depending on 
$\,\mathfrak{g}$, cf.\ (\ref{lie}).) Although $\,\h\,$ still remains quite 
artificial, $\,\xy\,$ and $\,b\,$ have now become more tangible.

Secondly, $\,\in\,$ in (\ref{fin}) means that the function of (\ref{abg}) 
standing on the left-hand side of (\ref{fin}) is {\smallit itself an element 
of\/} $\,\en\,$ (rather than taking values in $\,\en$).

Finally, $\,\en\,$ is the space of $\,(\es\times\et)$-val\-ued {\smallit 
negligible functions\/} on 
$\,\mathfrak{g}\times\mathfrak{g}\times\bbF^9\nnh$. Negligible (polynomial) 
functions on $\,\mathfrak{g}\times\mathfrak{g}\times\bbF^9$, valued in various 
vector spaces, are defined in Section~\ref{np}; the proper class formed by 
them (in the sense of set theory) may be called an ideal, since one easily 
verifies that, applying a mul\-ti\-lin\-e\-ar mapping to several polynomial 
functions on $\,\mathfrak{g}\times\mathfrak{g}\times\bbF^9\nnh$, one of which 
is negligible, we always obtain another negligible function. In addition,
\begin{equation}\label{bez}
\mathrm{negligible\ functions\ vanish\ at\ all\ 
}\,(a,b,\nh\mathbf{v})\in\mathfrak{g}\times\mathfrak{g}\times\bbF^9
\mathrm{\ such\ that\ }\,b=0\hh.
\end{equation}
One specific vector $\,\mathbf{u}=(\xy,\f\nh,\h,\x,\y,\z,\p,\q,\r)\in\bbF^9\nnh$, 
given by formula (\ref{fhz}), has
\begin{equation}\label{jue}
\jz\hh(\mathbf{u})=\mathbf{0}\,\,\mathrm{\ and\ }\,\,\xy=\h=0\hh.
\end{equation}
In Section~\ref{mr} we use $\,\mathbf{u}\,$ to define a family $\,\ec\,$ of 
Ein\-stein connections by
\begin{equation}\label{edp}
\ec\,=\,\dj\,
+\,\pi_\es\w(\{\iz_a\w[\mathbf{u}]:a\in\mathfrak{g}\hs\,\mathrm{\ and\ 
}\,a^2\nh=0\})\hh,
\end{equation}
which appears as formula (\ref{cdl}); $\,\pi_\es\w:\es\times\et\to\es\,$ is 
the Cartesian projection.

That every $\,\nabla\nh=\dj+\hs\sj\in\ec\,$ is a weak\-\hbox{ly\hh-}\hskip0ptEin\-stein 
connection (or, equivalently, $\,\kz\hh(\sj,\sy)=(0,0)\,$ for the 
corresponding $\,(\sj,\sy)$, cf.\ (\ref{bij})) is immediate: by (\ref{aqe}), 
with $\,\xy=\h=0\,$ due to (\ref{jue}), the relation 
$\,a^2\nh=0\,$ in (\ref{edp}) gives $\,b=0$. Now (\ref{bez}) implies that the 
value of left-hand side of (\ref{fin}) at $\,(a,b,\mathbf{u})\,$ equals $\,0$, 
while $\,\iz_b\w[\hh\mz\hh(\nh\mathbf{v}\nh)]=0$ by (\ref{iaz}), since the 
$\,\f\nh$-com\-po\-nent $\,\mz\hh(\nh\mathbf{v}\nh)\,$ vanishes identically, 
and $\,\iz_a\w[\hh\jz\hh(\nh\mathbf{u}\nh)]=0\,$ due to (\ref{jue}) and 
linearity of $\,\iz_a\w$. Thus, $\,\kz\hh(\iz_a\w[\mathbf{v}])=0$.

Furthermore, evaluating $\,\{\nabla\hskip-2.5pt\cdot\hskip-2.5pt\nabla\}\,$ and 
using (\ref{nde}) we now conclude that every $\,\nabla\in\ec\,$ is in fact an 
Ein\-stein connection.

The most important further step consists in showing, in Section~\ref{ra}, that
\begin{equation}\label{thf}
\begin{array}{l}
\mathrm{the\ set\ (\ref{edp})\ of\ Ein\-stein\ connections\ contains\ all\  
real}\hyp\mathrm{an\-a\-lyt\-ic\ curves}\\
{}[\hs0,\vd)\ni t\,\mapsto\,\dj\hs+\hs\hs\sj(t)\,\,\hs\mathrm{\ of\ 
weak\-ly}\hyp\mathrm{Ein\-stein\ connections\ with\ }\hs\,\,\sj(0)=0\hh.
\end{array}
\end{equation}
Assertion (\ref{thf}) easily implies the crucial part of our main result 
(Theorem~\ref{mnres}):
\begin{equation}\label{sfc}
\begin{array}{l}
\mathrm{the\ family\ (\ref{edp})\ of\ Ein\-stein\ connections\ contains\ all\ 
weak\-ly}\hyp\mathrm{Ein}\hyp\\
\mathrm{stein\ connections\ sufficiently\ close\ to\ the\ standard\ 
connection\ }\,\,\dj\hh.
\end{array}
\end{equation}
For $\,\mathfrak{g}=\mathfrak{su}\hh(\n)$, (\ref{sfc}) states that {\smallit a 
pos\-i\-tive-def\-i\-nite multiple of the Kil\-ling form is isolated among 
suitably normalized left-in\-var\-i\-ant Riemannian Ein\-stein metrics on\/}  
$\,\mathrm{SU}\hh(\n)$.

The mechanism which reduces proving the inclusion in (\ref{sfc}) to establishing 
the analogous claim (\ref{thf}) for real-an\-a\-lyt\-ic\ curves is a version 
of Milnor's \hbox{curve\hs-}\hskip0ptse\-lec\-tion lemma 
\cite[p.\ 25]{milnor}, stated below as Corollary~\ref{equal}, and applicable 
in our case since both $\,\ec\,$ and $\,\ew\,$ are algebraic sets.

The rest of this section is devoted to summarizing the proof of (\ref{thf}) 
given in Section~\ref{ra}. The argument is based on the fundamental formula 
(\ref{fin}) along with the following fact (Lemma~\ref{neglg}(ii)): suppose 
that a vec\-tor-val\-ued polynomial function $\,\pz\,$ of (\ref{abg}) is 
negligible and (\ref{abg}) are themselves $\,C^\infty$ functions of a variable 
$\,t\in[\hs0,\vd)$, with $\,\vd\in(0,\infty)$, which turns $\,\pz\,$ into a 
function of $\,t$, and $\,k\ge1\,$ is an integer. In such a case, due to 
negligibility of $\,\pz\nh$,
\begin{equation}\label{azz}
\mathrm{if\ }\,a(0)=0\hh,\,\,\xy(0)=\h(0)=0\hh,\mathrm{\ and\ }\,\jm[b\hh]=0\hh,
\mathrm{\ then\ }\,\jk[\pz\hskip1.4pt]=0\hh.
\end{equation}
Here $\,\jk[\hs\ldots\hs]\,$ denotes the $\,k$-jet of $\,\ldots\,$ at $\,t=0$.

Fixing $\,t\mapsto\sj(t)\,$ in (\ref{thf}), we now realize 
$\,a,b,\xy,\f\nh,\h,\x,\y,\z,\p,\q,\r,\nh\mathbf{v}\,$ and $\,\sy$ as 
real-an\-a\-lyt\-ic functions of $\,t\in[\hs0,\vd)$. First, a natural 
surjective operator $\,\es\to\mathfrak{g}$ (see Remark~\ref{natrl}), applied 
to $\,\sj=\sj(t)$, yields $\,a=a(t)\in\mathfrak{g}$. To define $\,\sy(t)\,$ 
and $\,\xy(t)\,$ we use (\ref{bij}) for $\,\sj=\sj(t)\,$ and the relation 
$\,\mathrm{tr}\hskip2pta\nh^2\nh=\n\hs\ve^{-2}\xy\,$ in (\ref{aqe}).

Our polynomial mapping $\,\jz:\bbF^9\to\hs\bbF^9$ takes values in 
$\,\bbF^8\nnh$, as its first component is $\,0$, while 
$\,\mathrm{rank}\hskip2.7pt\mathrm{d}\hskip.3pt\jz\nh_{\mathbf{u}}\w=8\,$ at 
our point $\,\mathbf{u}\,$ with $\,\jz\hh(\mathbf{u})=\mathbf{0}$. 
Consequently (cf.\ Lemma~\ref{nnupl}(b)), a neighborhood of $\,\mathbf{u}\,$ 
in $\,\jz^{\hs-1}(0)\,$ forms the graph of an $\,\bbF^8\nh$-val\-ued 
$\,\bbF$-an\-a\-lyt\-ic function 
$\,\xy\,\mapsto\,(\f\nh,\h,\x,\y,\z,\p,\q,\r)\,$ of $\,\xy\,$ that varies 
near $\,0\,$ in $\,\bbF$, with $\,\mathrm{d}\h/\mathrm{d}\hh\xy\ne0\,$ at 
$\,\xy=0$. Since $\,\xy\,$ already is a function of $\,t\in[\hs0,\vd)$, so are 
now $\,\f\nh,\h,\x,\y,\z,\p,\q,\r\,$ (with smaller $\,\vd$), 
$\,\mathbf{v}=(\xy,\f\nh,\h,\x,\y,\z,\p,\q,\r)$, and $\,b$, cf.\ 
(\ref{aqe}).

Next, while proving Lemma~\ref{ractz}, we establish the following claim:
\begin{equation}\label{cla}
\mathrm{for\ all\ }\,t\in[\hs0,\vd)\,\mathrm{\ one\ has\ }
\,(\sj(t),\sy(t))=\iz_{a(t)}\w[\mathbf{v}(t)\hn]\,\mathrm{\ and\ }
\,b(t)=0\hh.
\end{equation}
This is done by using induction on $\,k\,$ to show equality between 
$\,k$-jets $\,\jk[\hs\ldots\hs]\,$ of both sides at $\,t=0$, for all 
$\,k\ge0$. In the induction step, assuming that $\,\jm[b\hh]=0$, we use 
(\ref{azz}), with $\,\pz\,$ equal to the left-hand side of (\ref{fin}) 
(for $\,a=a(t)\,$ and $\,\mathbf{v}=\mathbf{v}(t)$) to conclude that 
$\,\jk[\pz\hskip1pt]=0$. However, our choice of $\,\mathbf{v}(t)\,$ guarantees 
that $\,\jz\hh(\mathbf{v}(t))=\mathbf{0}$. Thus, by linearity of $\,\iz_a\w$, 
(\ref{fin}) with $\,\jk[\pz\hskip1pt]=0\,$ amounts to
\begin{equation}\label{vkj}
\mathrm{vanishing\ of\ the\ }\,k\hyp\mathrm{jet\ of\ 
}\,\kz\hh(\iz_a\w[\mathbf{v}])\,
-\,\iz_b\w[\hh\mz\hh(\nh\mathbf{v}\nh)]\,\mathrm{\ at\ }\,t=0\hh.
\end{equation}
(From now on we write $\,a,\nh\mathbf{v}\,$ for $\,a(t),\hs\mathbf{v}(t)$, etc.) 
Our quadratic mapping $\,\kz\,$ is the sum of homogeneous components $\,\lz\,$ 
and $\,\kz-\lz\,$ of degrees $\,1\,$ and $\,2$, cf.\ Remark~\ref{kssez}. As 
$\,(\sj,\sy)=0\,$ at $\,t=0$, the $\,k$-jet $\,\jk[(\kz-\lz)(\sj,\sy)]\,$ 
depends only on $\,\jm[(\sj,\sy)]=\jm[\hh\iz_a\w[\mathbf{v}]]$, the last 
equality being the remaining part of the in\-duc\-tive-step assumption in our 
ongoing proof of (\ref{cla}). Recalling that our hypothesis about 
$\,t\mapsto\sj(t)\,$ in (\ref{thf}) reads $\,\kz\hh(\sj,\sy)=(0,0)\,$ for all 
$\,t\,$ (see (\ref{slv})), we thus have 
$\,\jk[\hh\lz(\sj,\sy)]=\jk[(\lz-\kz)(\sj,\sy)]
=\jk[(\lz-\kz)(\iz_a\w[\mathbf{v}])]\,$ and so, by (\ref{vkj}),
\begin{equation}\label{jlk}
\jk[\hh\lz((\sj,\sy)-\iz_a\w[\mathbf{v}])]\,
=\,-\jk[\hh\iz_b\w[\hh\mz\hh(\nh\mathbf{v}\nh)]]\hh.
\end{equation}
Due to the fact that $\,\mz\,$ takes values in a specific 
\hbox{two\hh-}\hskip0ptdi\-men\-sion\-al vector subspace of $\,\bbF^9$ (cf.\ 
Remark~\ref{polym}), one has 
$\,\iz_b\w[\mz\hh(\nh\mathbf{v}\nh)]\in\mathrm{Ker}\,\lz\,$ for all 
$\,b\in\mathfrak{g}\,$ and $\,\mathbf{v}\in\bbF^9\nnh$. However (as a 
consequence of Remark~\ref{kssez} and (\ref{opl}.b)), the image of $\,\lz\,$ 
intersects $\,\mathrm{Ker}\,\lz\,$ trivially, that is, $\,\lz\,$ restricted to 
its own image must be injective. Thus, both sides in (\ref{jlk}) must vanish. 
Now, as $\,\mz\hh(\nh\mathbf{u})\ne\mathbf{0}$, the equality 
$\,\jk[\hh\iz_b\w[\mz\hh(\nh\mathbf{v}\nh)]]=(0,0)$ easily implies that 
$\,\jk[b\hh]=0$. Similarly, since $\,(\sj,\sy)-\iz_a\w[\mathbf{v}]\,$ lies in the 
image of $\,\lz\,$ (see Remark~\ref{ssini}), the injectivity property just 
mentioned, combined with vanishing of the left-hand side in (\ref{jlk}), 
yields $\,\jk[(\sj,\sy)]=\jk[\hh\iz_a\w[\mathbf{v}]]$, completing the 
inductive step and, consequently, proving (\ref{cla}).

The induction argument, used above to show that $\,b(t)=0\,$ in (\ref{cla}), 
would obviously remain valid if the quantities involved, instead of being 
real-an\-a\-lyt\-ic functions of $\,t$, were just formal power series in the 
variable $\,t$. The subsequent conclusion that $\,a\nh^2=0\,$ for $\,a=a(t)\,$ 
and every $\,t$, outlined in the next paragraph is, however, less 
straightforward in this regard, as it relies on certain rationality and 
non\-\hbox{zero\hh-}\hskip0ptde\-riv\-a\-tive properties of $\,\h\,$ and 
$\,\xy$.

Specifically, vanishing of $\,b=b(t)\,$ in (\ref{cla}) gives, by (\ref{aqe}), 
$\,a\nh^2=\ve^{-1}\h a+\ve^{-2}\xy\hs\mathrm{Id}\,$ for $\,a=a(t)\,$ which, 
unless $\,a=0$, easily implies that $\,\h^2$ is a rational multiple of 
$\,\xy$ (see  Lemma~\ref{ratnl}). Since $\,\h\,$ is an 
$\,\bbF$-an\-a\-lyt\-ic function of $\,\xy\,$ (cf.\ the lines preceding 
(\ref{cla})), with $\,\mathrm{d}\h/\mathrm{d}\hh\xy\ne0=\h\,$ at $\,\xy=0$, 
while $\,\xy\,$ is a real-an\-a\-lyt\-ic function of $\,t$, the rationality 
conclusion just stated shows that $\,\xy(t)=\h(t)=0$ identically in $\,t$. 
(See the end of Section~\ref{ra}.) Thus, $\,a=a(t)\,$ satisfies the condition 
$\,a\nh^2=0$, while, as $\,\mathbf{v}=(\xy,\f\nh,\h,\x,\y,\z,\p,\q,\r)\,$ is a 
function of $\,\xy\,$ (Lemma~\ref{nnupl}(b)), one also has 
$\,\mathbf{v}(t)=\mathbf{u}\,$ for all $\,t$. By (\ref{cla}), this completes 
the proof of (\ref{thf}).

\section{Linear-algebra preliminaries}\label{la}\checked
\setcounter{equation}{0}
In this section, vector spaces are always 
fi\-\hbox{nite\hh-}\hskip0ptdi\-men\-sion\-al. An {\smallit in\-ner-prod\-uct 
space\/} over $\,\bbF=\bbR\,$ or $\,\bbF=\bbC\,$ is a vector space over 
$\,\bbF\,$ with a fixed $\,\bbF\nh$-val\-ued nondegenerate symmetric 
$\,\bbF$-bi\-lin\-e\-ar form. When $\,\bbF=\bbR$, in\-ner-prod\-uct spaces are 
what one usually calls {\smallit 
pseu\-\hbox{do-}\nh Eu\-clid\-e\-an spaces}. Note that, except 
in Section~\ref{cw}, complex inner products are {\smallit not\/} assumed 
ses\-qui\-lin\-e\-ar.

For en\-do\-mor\-phisms $\hs\aj,\td\aj\hh$ of a vector space $\,\ev\,$ 
over $\,\bbF=\bbR\,$ or $\,\bbF=\bbC$, let
\begin{equation}\label{acb}
(\aj,\hs\td\aj)\,=\,\hh\mathrm{tr}\hskip2pt\aj\td\aj\hs.
\end{equation}
Clearly, $\,(\hskip2pt,\hskip.4pt)\,$ turns the en\-do\-mor\-phism space 
$\,\mathrm{End}\hskip2pt\ev$ into an  in\-ner-prod\-uct space. For 
en\-do\-mor\-phisms $\,\aj\,$ of a {\smallit complex\/} vector space, with 
$\,\mathrm{tr}^\bbF$ denoting the $\,\bbF\nh$-trace,
\begin{equation}\label{abc}
\mathrm{tr}^\bbR\nh\aj\,=\,2\,\mathrm{Re}\hskip2.7pt\mathrm{tr}^\bbC\nh\aj\hh.
\end{equation}
Treating a complex vector space $\,\ev\,$ as real and writing 
$\,(\hskip2pt,\hskip.4pt)^\bbF$ rather than $\,(\hskip2pt,\hskip.4pt)\,$ to 
indicate the choice of the scalar field, we have, in 
$\,\mathrm{End}_\bbC\w\nh\ev$, by (\ref{acb}) and (\ref{abc}),
\begin{equation}\label{rec}
(\hskip2pt,\hskip.4pt)^\bbR\,
=\,\,2\,\mathrm{Re}\hskip2.7pt(\hskip2pt,\hskip.4pt)^\bbC.
\end{equation}
Let us define the adjoint $\,\dz\nnh^*$ of an operator $\,\dz\,$ between 
in\-ner-prod\-uct spaces in the usual way, refer to an en\-do\-mor\-phism 
$\,\dz\,$ of such a space $\,\es\,$ as {\smallit self-ad\-joint\/} if 
$\,\dz\nnh^*\nh=\dz$, and call a sub\-space $\,\es'$ of $\,\es\,$ {\smallit 
nondegenerate\/} if so is the restriction of the inner product to 
$\,\es'\nnh$. For a self-ad\-joint en\-do\-mor\-phism $\,\dz\,$ of an 
in\-ner-prod\-uct space $\,\es$, with $\,[\hskip3pt]^\perp$ denotung the 
orthogonal complement,
\begin{equation}\label{imd}
\begin{array}{l}
\mathrm{a)}\hskip8pt\mathrm{the\ image\ }\,\dz(\es)\,\mathrm{\ coincides\ 
with\ }\,[\hs\mathrm{Ker}\,\dz]^\perp\nnh,\mathrm{\ so\ that}\\
\mathrm{b)}\hskip8pt\es\,
=\,[\hh\dz(\es)]\oplus\mathrm{Ker}\,\dz\hskip6pt\mathrm{\ if}\hskip5pt
\mathrm{Ker}\,\dz\hskip5pt\mathrm{is\ nondegenerate.}
\end{array}
\end{equation}
Given sub\-spaces $\,\es,\hs\es'$ of an in\-ner-prod\-uct space $\,\ev\nnh$,
\begin{equation}\label{sin}
\es\,\mathrm{\ is\ nondegenerate\ if\ }\,\ev=\es\oplus\es'\mathrm{\ and\ }
\,\es,\hs\es'\mathrm{\ are\ orthogonal\ to\ each\ other,}
\end{equation}
since vectors in $\,\es$, orthogonal to $\,\es$, must be orthogonal to all 
of $\,\ev\nnh$.

Let $\,\ch\,$ be an en\-do\-mor\-phism of an in\-ner-prod\-uct space 
$\,\es\nnh$. Then
\begin{equation}\label{dia}
\begin{array}{l}
\ch^*\nnh=\ch\,\mathrm{\ if\ }\,\ch\,\mathrm{\ is\ di\-ag\-o\-nal\-izable\ 
with\ mutually\ orthogonal\ eigen\-spaces;}\\
\ch\,\mathrm{\ has\ nondegenerate\ eigen\-spaces\ if\ }\,\ch^*\nnh
=\ch\,\mathrm{\ and\ }\,\ch\,\mathrm{\ is\ di\-ag\-o\-nal\-izable.}
\end{array}
\end{equation}
Cf.\ (\ref{sin}). Obviously, given an en\-do\-mor\-phism $\,\dz\,$ of a vector 
space $\,\es$,
\begin{equation}\label{iso}
\mathrm{\ if\ }\,\es=[\hh\dz(\es)]\oplus\mathrm{Ker}\,\dz,\mathrm{\ then\ 
}\,\dz:\dz(\es)\to\dz(\es)\,\mathrm{\ is\ an\ isomorphism.}
\end{equation}
Let there be given vector spaces $\,\es,\et\nnh,\ev\nh$, an integer $\,k\ge1$, 
and a bi\-lin\-e\-ar mapping $\,B:\es\times\et\nnh\to\ev\nh$, along with 
$\,C^\infty$ curves $\,\sj\,$ in $\,\es\,$ and $\,\sy\,$ in $\,\et\nnh$, 
both parametrized by $\,t\in[\hs0,\vd)$, where $\,\vd\in(0,\infty)$. If 
$\,\sj(0)=0\,$ and $\,\sy(0)=0$, then the Leib\-niz rule clearly implies that, 
whenever $\,k\ge1$,
\begin{equation}\label{jkz}
\jk[B(\sj,\sy)]\,\mathrm{\ depends\ only\ on\ }\,\jm[\hh\sj]\,\mathrm{\ and\ }
\jm[\hs\sy]\hh.
\end{equation}
Here and below we use the convention that, for such curves,
\begin{equation}\label{jet}
\jk[\hs\ldots\hs]\,\mathrm{\ denotes\ the\ }\,k\hyp\mathrm{jet\ of\ }
\,\ldots\,\mathrm{\ at\ }\,t=0\hh.
\end{equation}
\begin{rem}\label{trzro}Due to invariance of the trace functional, 
an en\-do\-mor\-phism $\,\td\aj\,$ of a vector space such that 
$\,\jj\td\aj\jj^{-1}\nnh=\omega\td\aj\,$ for some linear automorphism $\,\jj\,$
and some scalar $\,\omega\ne1\,$ is necessarily trace\-less. When 
$\,\omega=-1\,$ and $\,\jj\,$ is the multiplication by $\,\mathrm{i}\,$ in the 
underlying real space of a complex vector space, this shows that {\smallit 
the real trace of a $\,\bbC$-anti\-lin\-e\-ar\ en\-do\-mor\-phism always 
equals\/} $\,0$.
\end{rem}
\begin{rem}\label{quate}Denoting by 
$\,1,\nh\mathrm{i},\mathrm{j},\nh\mathrm{k}\,$ 
the standard $\,\bbR$-basis of the qua\-ter\-ni\-on field $\,\bbH$, we 
identify the real sub\-space of $\,\bbH\,$ spanned by $\,1\,$ and 
$\,\mathrm{i}\,$ with $\,\bbC$, which turns any given vector space $\,\ev\,$ 
over $\,\bbH\,$ into a complex vector space. Then 
$\,\mathrm{tr}^\bbC\nh\aj\in\bbR\,$ for all 
$\,\bbH$-lin\-e\-ar en\-do\-mor\-phisms $\,\aj\,$ of $\,\ev\nh$, as one sees 
applying Remark~\ref{trzro} and (\ref{abc}) to $\,\td\aj=\mathrm{i}\aj$, with 
$\,\omega=-1\,$ and $\,\jj\,$ acting via the $\,\bbH$-multiplication by 
$\,\mathrm{j}$.
\end{rem}

\section{Milnor's curve\hs-se\-lec\-tion lemma}\label{mc}\checked
\setcounter{equation}{0}
Throughout this section, all vector spaces are 
fi\-\hbox{nite\hh-}\hskip0ptdi\-men\-sion\-al and real.

A subset of a vector space $\,\es\,$ is called {\it algebraic\/} if it equals 
$\,F^{-1}(0)\,$ for some polynomial mapping $\,F:\es\to\ev\,$ into a vector 
space $\,\ev\nh$. By a {\it sem\-i-al\-ge\-bra\-ic set\/} in 
$\,\es\,$ one means the intersection of an algebraic set with 
$\,\bigcap_{\hs j=1}^{\hs k}f\nnh_j^{\,\hs-1}((0,\infty))$, where $\,k\ge1\,$ 
and $\,f\nnh_1\w,\dots,f\nnh_k\w$ are polynomial functions $\,\es\to\bbR$. The 
intersection of two sem\-i-al\-ge\-bra\-ic sets in $\,\es\,$ is 
sem\-i-al\-ge\-bra\-ic, while complements of algebraic subsets of $\,\es\,$ 
constitute finite unions of sem\-i-al\-ge\-bra\-ic sets. Thus, whenever 
$\,\ez\subset\es\,$ and $\,\el\subset\es\,$ are algebraic, one easily sees that
\begin{equation}\label{emz}
\ez\nh\smallsetminus\el\,\mathrm{\ is\ a\ finite\ union\ of\ 
sem\-i}\hyp\mathrm{al\-ge\-bra\-ic\ sets\ in\ }\,\es\hh.
\end{equation}
The following result of Milnor \cite[p.\ 25]{milnor}, known as the {\it 
curve\hh-se\-lec\-tion lemma}, generalizes the earlier versions due to 
Bruhat and Cartan \cite[Theorem 1]{bruhat-cartan}, and Wallace 
\cite[Lemma 18.3]{wallace}. Further details can be found in 
\cite[p.\ 402]{moreira-ruas}.
\begin{thm}\label{cslem}{\smallit 
Suppose that\/ $\,0\hs$ lies in the closure of a sem\-i-al\-ge\-bra\-ic 
subset\/ $\hs\mathcal{A}\,$ of a vector space\/ $\,\es$. Then there exists a 
real-an\-a\-lyt\-ic curve\/ $\,[\hs0,\vd)\ni t\mapsto \sj(t)\in\es$, with\/ 
$\,\vd\in(0,\infty)$, such that\/ $\,\sj(0)=0\,$ and\/ 
$\,\sj(t)\in\mathcal{A}\,$ for all\/ $\,t\in(0,\vd)$.
}
\end{thm}
\begin{proof}See \cite[p.\ 25]{milnor}.
\end{proof}
We will need this immediate consequence of Theorem~\ref{cslem}:
\begin{cor}\label{equal}{\smallit 
Let\/ $\,\ez\,$ and\/ $\,\el\,$ be algebraic sets in a vector space\/ $\,\es$. 
If\/ $\,\el\,$ contains every real-an\-a\-lyt\-ic curve\/ 
$\,[\hs0,\vd)\ni t\mapsto\sj(t)\in\es\,$ with\/ $\,\sj(0)=0$, lying entirely 
in\/ $\,\ez$, then\/ $\,\el$ also contains all points of\/ $\hs\ez\hs$ 
sufficiently close to\/ $\,0$.
}
\end{cor}
\begin{proof}Otherwise, an obvious contradiction would arise from 
Theorem~\ref{cslem} applied to $\,\mathcal{A}\,$ which is one of the 
sem\-i-al\-ge\-bra\-ic sets mentioned in (\ref{emz}).
\end{proof}

\section{The curvature and Ric\-ci tensors}\label{cr}\checked
\setcounter{equation}{0}
We denote by $\,R=R^\nabla$ the curvature tensor of a tor\-sion-free 
$\,C^\infty$ connection $\,\nabla$ on a manifold, using the sign convention
\begin{equation}\label{cur}
R(v,w)\hh u=\nsw\nsvu-\nsv\nswu+\nabla_{[v,w]}u\hs.\hskip7pt\mathrm{for\ }
\,C^\infty\mathrm{\ vector\ fields}\hskip4ptu,v,w\hs.
\end{equation}
The Ric\-ci tensor $\,\rho=\rho\nnh^\nabla$ has the components 
$\,R_{jk}\w=R_{jsk}\w{}^s\nnh$. Repeated indices, here and below, are to be 
{\smallit summed over}.

Any tor\-sion-free $\,C^\infty$ connection $\,\nabla\,$
satisfies the Boch\-ner identity
\begin{equation}\label{bch}
\rho\hh(\,\cdot\,,w)\,=\,\hs\mathrm{div}\hs\naw\,
-\,\hs\mathrm{d}\hs(\dvw)\hskip15pt\mathrm{for\ every\ }\,C^\infty\mathrm{\ 
vector\ field}\hskip6ptw\hs,
\end{equation}
cf.\ \cite[pp.\ 57, 74, 166]{besse}, 
Here $\,\dvw=\mathrm{tr}\hskip2.5pt\naw$, with $\,\naw\,$ treated as the 
en\-do\-mor\-phism of the tangent bundle acting on vector fields $\,v\,$ by 
$\,v\mapsto\nsvw$.

The coordinate form of (\ref{bch}), 
$\,R_{jk}\w w^{\hs k}\nh=w^{\hs k}{}_{,\hs jk}\w-w^{\hs k}{}_{,kj}\w$, arises 
by contraction in $\,l=k\,$ from the Ric\-ci identity $\,w^{\hs l}{}_{,jk}\w
-w^{\hs l}{}_{,\hs kj}\w=R_{jks}\w{}^lw^s\nnh$, which in turn is the 
coordinate form of (\ref{cur}). Evaluated on a $\,C^\infty$ vector field 
$\,v$, (\ref{bch}) gives
\begin{equation}\label{rww}
\rho\hh(v,w)\hs
=\hs\mathrm{div}\hs\lx\nsvw\hs-\hs(\dvw)\hh v\rx\hs+\hs(\dvv)\hskip2pt\dvw\hs
-\hs(\nav,\nnh\naw)\hs,
\end{equation}
via differentiation by parts. For the meaning of $\,(\nav,\naw)$, see 
(\ref{acb}).

By an {\smallit Ein\-stein metric\/} on a manifold we mean any $\,C^\infty$ 
pseu\-\hbox{do\hskip1pt-}\hskip0ptRiem\-ann\-i\-an metric $\,\gy\,$ with 
$\,\rho=\ct\hs\gy\,$ for the Ric\-ci tensor $\,\rho\,$ of $\,\gy\,$ (that is, 
the Ric\-ci tensor of the Le\-\hbox{vi\hh-}\hskip0ptCi\-vi\-ta connection of $\,\gy$) and some 
$\,\ct\in\bbR\,$ (the {\smallit Ein\-stein constant\/} of $\,\gy$).

One calls a tor\-sion-free $\,C^\infty$ connection $\,\nabla\,$ on an 
$\,\m$-di\-men\-sion\-al manifold {\smallit equi\-af\-fine\/} if the 
connection induced by $\,\nabla\,$ in the $\,\m\hh$th exterior power of the 
tangent bundle is flat, that is, if the manifold locally admits 
$\,\nabla\nh$-par\-al\-lel volume forms.

Equi\-af\-fin\-i\-ty of $\,\nabla\,$ is equivalent to symmetry of its Ric\-ci 
tensor $\,\rho\nnh^\nabla\nnh$. This well-known fact 
\cite[III 4,5 and (5.8), pp.\ 144--145]{schouten} is not used in our argument.

\begin{rem}\label{recom}The whole discussion in this section remains valid for 
hol\-o\-mor\-phic connections and hol\-o\-mor\-phic metrics on complex 
manifolds \cite[pp.\ 210\hs--211]{lebrun}.
\end{rem}
\begin{rem}\label{lccon}Recall that the Le\-\hbox{vi\hh-}\hskip0ptCi\-vi\-ta connection of 
a $\,C^\infty$ or hol\-o\-mor\-phic  metric is the unique tor\-sion-free 
connection making the metric parallel \cite[p.\ 211]{lebrun}.
\end{rem}

\section{Left-in\-var\-i\-ant connections on Lie groups}\label{li}\checked
\setcounter{equation}{0}
Let $\,\gj\,$ be a real or complex Lie group. Its Lie algebra 
$\,\mathfrak{g}$, over $\,\bbF=\bbR\,$ or $\,\bbF=\bbC$, is always identified 
with the space of left-in\-var\-i\-ant vector fields on $\,\gj$. In addition, 
left-in\-var\-i\-ant connections on $\,\gj\,$ are assumed hol\-o\-mor\-phic 
when $\,\bbF=\bbC$. Similarly, {\smallit left-in\-var\-i\-ant metrics\/} on 
$\,\gj\,$ are by definition pseu\-\hbox{do\hh-}\hskip0ptRiem\-ann\-i\-an if 
$\,\bbF=\bbR$, and hol\-o\-mor\-phic (in the sense of being 
$\,\bbC$-bi\-lin\-e\-ar, symmetric and nondegenerate) if $\,\bbF=\bbC$. Cf.\ 
also Remark~\ref{recom}.

Left-in\-var\-i\-ant connections $\,\nabla\,$ on $\,\gj$, and 
left-in\-var\-i\-ant symmetric $\,2$-ten\-sor fields $\,\tau\,$ on $\,\gj\,$ 
(including metrics) are always treated as elements of 
$\,\ey=[\mathfrak{g}\nh^*]^{\otimes2}\nnh\otimes\mathfrak{g}$ (or, 
$\,\et=[\mathfrak{g}\nh^*]^{\odot2}$), that is, $\,\bbF\nh$-bi\-lin\-e\-ar 
operations $\,\nabla:\mathfrak{g}\times\mathfrak{g}\to\mathfrak{g}\,$ (or, 
respectively, symmetric $\,\bbF\nh$-bi\-lin\-e\-ar forms 
$\,\tau:\mathfrak{g}\times\mathfrak{g}\to\bbF$). For $\,v,w\in\mathfrak{g}$, 
we then use the traditional symbol $\,\nsvw$, rather than $\,\nabla(v,w)$, and 
denote by $\,\nsv$ (or, by $\,\naw$) the $\,\bbF$-lin\-e\-ar 
en\-do\-mor\-phism of $\,\mathfrak{g}\,$ sending $\,w\,$ (or, respectively, 
$\,v$) to $\,\nsvw$.

A left-in\-var\-i\-ant tor\-sion-free connection $\,\nabla\hs$ on $\,\gj\,$ 
will be called {\smallit uni\-mod\-u\-lar\/} if some/\nh any 
left-in\-var\-i\-ant real/\nh com\-plex volume form on $\,\gj\,$ is 
$\,\nabla\nh$-par\-al\-lel.

Uni\-mod\-u\-lar\-i\-ty of $\,\nabla\hs$ clearly implies its 
equi\-af\-fin\-i\-ty (see the end of Section~\ref{cr}).
\begin{lem}\label{unimo}{\smallit 
Let\/ $\,\nabla\,$ be a left-in\-var\-i\-ant tor\-sion-free connection on a 
real$/\hn$com\-plex Lie group\/ $\,\gj\,$ with the Lie algebra\/ 
$\,\mathfrak{g}\,$ over\/ $\,\bbF=\bbR\,$ or\/ $\,\bbF=\bbC$.
\begin{enumerate}
  \def\theenumi{{\rm\roman{enumi}}}
\item[{\rm(a)}] $\nabla\,$ is uni\-mod\-u\-lar if and only if\/ 
$\,\mathrm{tr}\hskip2pt\nsv=0\,$ for all\/ $\,v\in\mathfrak{g}$.
\item[{\rm(b)}] $\nabla\,$ is uni\-mod\-u\-lar whenever there exists a 
nondegenerate left-in\-var\-i\-ant\/ $\,\nabla\nnh$-par\-al\-lel 
twice-co\-var\-i\-ant symmetric tensor field on\/ $\,\gj$, that is, whenever\/ 
$\,\nabla\,$ is the Le\-\hbox{vi\hh-}\hskip0ptCi\-vi\-ta connection of some left-in\-var\-i\-ant 
metric.
\item[{\rm(c)}] If\/ $\,\gj\,$ is uni\-mod\-u\-lar, then 
uni\-mod\-u\-lar\-i\-ty of\/ $\,\nabla\,$ is equivalent to trace\-less\-ness 
of\/ $\,\naw$ for every\/ $\,w\in\mathfrak{g}\hh$, that is, to the condition\/ 
$\,\dvw=0\,$ for all\/ $\,w\in\mathfrak{g}\hh$. 
\item[{\rm(d)}] If both\/ $\,\gj\,$ and\/ $\,\nabla\,$ are uni\-mod\-u\-lar, 
then the Ric\-ci tensor\/ $\,\rho=\rho\nnh^\nabla$ is given by 
$\,\rho\hh(v,w)=-\hs(\nav,\naw)\,$ for\/ $\,v,w\in\mathfrak{g}\hh$, cf.\ 
{\rm(\ref{acb})}. In particular, $\,\rho\,$ is symmetric.
\item[{\rm(e)}] If\/ $\,\gj\,$ is uni\-mod\-u\-lar and the tensor\/ $\,\rho\,$ 
given by $\,\rho\hh(v,w)=-\hs(\nav,\naw)\,$ is nondegenerate and\/ 
$\,\nabla\nnh$-par\-al\-lel, then\/ $\,\nabla\,$ is uni\-mod\-u\-lar and\/ 
$\,\rho\,$ is its Ric\-ci tensor.
\end{enumerate}
}
\end{lem}
\begin{proof}Let $\,\m=\dimf\mathfrak{g}$. For the connection $\,\bna\,$ 
induced by $\,\nabla\,$ in the $\,\m\hh$th exterior power, over $\,\bbF\nnh$, 
of the tangent bundle of $\,\gj$, and any basis $\,e_1\w,\dots,e_\m\w$ of 
$\,\mathfrak{g}$, one clearly has 
$\,\bna_{\hskip-2.2ptv}(e_1\w\wedge\hs\ldots\hs\wedge e_\m\w)
=(\mathrm{tr}\hskip2pt\nsv)\hskip2pte_1\w\wedge\hs\ldots\hs\wedge e_\m\w$, 
which proves (a). To obtain (b), note that, for an orthonormal basis 
$\,e_1\w,\dots,e_\m\w$, the $\,\m$-vec\-tor 
$\,e_1\w\wedge\hs\ldots\hs\wedge e_\m\w$ is parallel due to its being uniquely 
determined, up to a sign, by the metric. Assertion (c) follows in turn from 
(a), since uni\-mod\-u\-lar\-i\-ty of $\,\gj\,$ amounts to trace\-less\-ness 
of $\,\mathrm{Ad}\,v\,$ for all $\,v,w\in\mathfrak{g}$, and 
$\,\mathrm{Ad}\,v=\nsv-\nav\,$ as $\,\nabla\,$ is tor\-sion-free, while (d) is 
obvious from (\ref{rww}), along with (c) applied to the left-in\-var\-i\-ant 
vector fields $\,w\,$ and $\,\nsvw\hs-\hs(\dvw)\hh v$. Finally, (e) is a 
consequence of (b), for the metric $\,\rho$, and (d).
\end{proof}
\begin{rem}\label{bijct}By assigning, to any left-in\-var\-i\-ant Ein\-stein 
metric with Einstein constant $\,1\,$ on a Lie group $\,\gj$, its 
Le\-\hbox{vi\hh-}\hskip0ptCi\-vi\-ta connection, we obtain a bi\-jec\-tive mapping between the 
set of such metrics and the set of all left-in\-var\-i\-ant tor\-sion-free 
connections $\,\nabla\,$ on $\,\gj\,$ having symmetric, nondegenerate, 
$\hs\nabla\nnh$-par\-al\-lel Ric\-ci tensors. The inverse mapping sends 
$\,\nabla\,$ to its Ric\-ci tensor.

This is a trivial consequence of Remark~\ref{lccon}.
\end{rem}
\begin{rem}\label{ssuni}We will use the fact that every sem\-i\-sim\-ple 
Lie group is uni\-mod\-u\-lar.
\end{rem}

\section{Some operations involving connections and $\,2$-ten\-sors}\label{so}
\setcounter{equation}{0}\checked
In this section, $\,\mathfrak{g}\,$ is a fixed sem\-i\-sim\-ple Lie algebra 
over $\,\bbF=\bbR\,$ or $\,\bbF=\bbC$.

As in the second paragraph of Section~\ref{li}, we identify elements 
$\,\nabla\,$ of the space 
$\,\ey=[\mathfrak{g}\nh^*]^{\otimes2}\nnh\otimes\mathfrak{g}\,$ with 
left-in\-var\-i\-ant connections on a Lie group $\,\gj\,$ having 
the Lie algebra $\,\mathfrak{g}$, writing $\,\nsvw,\hs\nsv$ and $\,\naw\,$ 
instead of $\,\nabla(v,w)$, $\,\nabla(v,\,\cdot\,)$ and 
$\,\nabla(\,\cdot\,,w)$, for $\,v,w\in\mathfrak{g}$, so that 
$\,\nsv,\naw:\mathfrak{g}\to\mathfrak{g}$. We will repeatedly refer to the 
vector spaces
\begin{equation}\label{sbs}
\et\,=\,[\mathfrak{g}\nh^*]^{\odot2},\hskip22pt
\es\,=\,[\mathfrak{g}\nh^*]^{\odot2}\nnh\otimes\mathfrak{g}\,
\subset\,\ey\,=\,[\mathfrak{g}\nh^*]^{\otimes2}\nnh\otimes\mathfrak{g}\hs.
\end{equation}
Note that $\,\et\subset[\mathfrak{g}\nh^*]^{\otimes2}$ and 
$\,\es\subset[\mathfrak{g}\nh^*]^{\otimes2}\nnh\otimes\mathfrak{g}\,$ consist of 
all $\,\bbF$-val\-ued (or, respectively, $\,\mathfrak{g}$-val\-ued) symmetric 
$\,\bbF$-bi\-lin\-e\-ar mappings defined on $\,\mathfrak{g}\times\mathfrak{g}$.

Elements of $\,\et\hs$ may, again, be thought of as twice-co\-var\-i\-ant 
symmetric tensor fields on $\,\gj$, invariant under left translations. The 
symbols $\,\by\,$ and $\,\langle\,,\rangle\,$ both stand for the 
$\,\bbF$-bi\-lin\-e\-ar Kil\-ling form of $\,\mathfrak{g}$. Thus, 
$\,\by\in\et$ and, with 
$\,(\hskip1.3pt,\hskip-.1pt)\,$ as in (\ref{acb}),
\begin{equation}\label{bvw}
\by(v,w)\,=\,\langle v,w\rangle\,
=\,(\mathrm{Ad}\,v,\hs\mathrm{Ad}\,w)\hskip13pt\mathrm{whenever}\hskip7pt
v,w\in\mathfrak{g}\hs.
\end{equation}
The $\,\bbF$-bi\-lin\-e\-ar inner product $\,\by=\langle\,,\rangle\,$ in 
$\,\mathfrak{g}\,$ leads to an isomorphic identification between 
bi\-lin\-e\-ar forms $\,\sigma\,$ on $\,\mathfrak{g}\,$ and 
$\,\bbF$-lin\-e\-ar en\-do\-mor\-phisms $\,\Sigma\,$ of $\,\mathfrak{g}$, with
\begin{equation}\label{sgv}
\sy(v,w)\,=\,\langle\Sigma v,w\rangle\hskip13pt\mathrm{for\ any}\hskip7pt
v,w\in\mathfrak{g}\hs.
\end{equation}
Left-in\-var\-i\-ant connections on $\,\gj\,$ which are {\smallit 
tor\-sion-free\/} form an af\-fine subspace of 
$\,[\mathfrak{g}\nh^*]^{\otimes2}\nnh\otimes\mathfrak{g}$, namely, the coset 
$\,\dj+\es$, where $\,\dj\,$ is the {\smallit standard connection}, with
\begin{equation}\label{tdv}
2\hs\dj_v\w w\,=\,[v,\hs w]\hs,\hskip12pt\mathrm{so\ that}\hskip8pt
\dj_v\w\hs=\hs-\hh\dj\hh v\hs=\hs\mathrm{Ad}\,v/2\hs.
\end{equation}
We define the {\smallit contraction operator\/} 
$\,\cn:\ey
=[\mathfrak{g}\nh^*]^{\otimes2}\nnh\otimes\mathfrak{g}\to\mathfrak{g}\nh^*$ by 
$\,(\cn\nabla)w=\mathrm{tr}\hskip2pt\naw$. In view of Lemma~\ref{unimo}(c) and 
Remark~\ref{ssuni}, for any $\,\nabla\nh=\dj+\hs\sj$, where $\,\sj\in\es$,
\begin{equation}\label{uni}
\dj+\hs\sj\hskip12pt\mathrm{is\ uni\-mod\-u\-lar\ if\ and\ only\ if}\hskip12pt
\cn\hh\sj=0\hs.
\end{equation}
Tensor-calculus arguments, when needed, will use components relative to a 
fixed basis $\,e_1\w,\dots,e_\m\w$ of $\,\mathfrak{g}$, with 
$\,\m=\dimf\mathfrak{g}$. Vectors $\,v\in\mathfrak{g}$, bi\-lin\-e\-ar forms 
$\,\sigma\,$ on $\,\mathfrak{g}$, connections 
$\,\nabla\in\ey=[\mathfrak{g}\nh^*]^{\otimes2}\nnh\otimes\mathfrak{g}$, and 
symmetric operations $\,\sj\in\es\,$ are represented by the components 
$\,v^{\hs j}\nnh$, $\,\sy_{\!jk}\w$, 
$\,\vg_{\hskip-2.2ptjk}^{\hs l}$ and 
$\,S_{\nnh jk}^{\hs l}=S_{\nh kj}^{\hs l}$ characterized by $\,v=v^{\hs j}e_j\w$, 
$\,\sy(v,w)=\sy_{\!jk}\w v^{\hs j}w^{\hs k}\nnh$, 
$\,[\nsvw]^l\nh=\vg_{\hskip-2.2ptjk}^{\hs l}v^{\hs j}w^{\hs k}$ and 
$\,[\hh\sj_vw]^l\nh=S_{\nnh jk}^{\hs l}v^{\hs j}w^{\hs k}\nnh$. Note the 
traditional symbol $\,\vg_{\hskip-2.2ptjk}^{\hs l}$, rather than 
$\,\nabla_{\hskip-2.2ptjk}^{\hs l}$. For the components $\,D_{\!jk}^{\hs l}$ 
of $\,\dj\,$ we have $\,2D_{\!jk}^{\hs l}=C_{\!jk}\w{}^l\nnh$, where 
$\,C_{\!jk}\w{}^l$ are the structure constants of $\,\mathfrak{g}$, with 
$\,[v,w]^l\nh=C_{\!jk}\w{}^lv^{\hs j}w^{\hs k}$ or, equivalently, 
$\,[e_j\w,e_k\w]=C_{\!jk}\w{}^le_l\w$. 
The reciprocal tensor of the Kil\-ling form $\,\by\,$ has the components 
$\,\by^{\hs jk}\nnh$, forming the inverse matrix of $\,[\by_{jk}\w]$. We will 
raise and lower indices with the aid of $\,\by^{\hs jk}$ and $\,\by_{jk}\w$, 
so that $\,C^{\hs j}{}_k\w{}^l\nh=\by^{\hs js}C_{\!sk}\w{}^l\nh$, 
$\,C_{\!jkl}\w=C_{\!jk}\w{}^s\by_{sl}\w$, etc. If an en\-do\-mor\-phism 
$\,\Sigma\,$ of $\,\mathfrak{g}\,$ corresponds to $\,\sy\,$ via (\ref{sgv}), 
its components will be written as $\,\sy_{\!j}^{\hh k}$, and then 
$\,[\Sigma v]^k\nh=\sy_{\!j}^{\hh k}v^{\hs j}\nnh$, while 
$\,\sy_{\!j}^{\hh k}=\sy_{\!js}\w\by^{\hs sk}$. Thus, for the contraction 
operator $\,\cn\,$ appearing in (\ref{uni}),
\begin{equation}\label{cne}
[\cn\nabla]_{j}\w\hh\,=\,\,\vg_{\hskip-2.2ptkj}^{\hs k}\hs.
\end{equation}
Since the Kil\-ling form $\,\by\,$ is bi-in\-var\-i\-ant (that is, 
$\,\mathrm{Ad}$-in\-var\-i\-ant), one has
\begin{equation}\label{aws}
(\mathrm{Ad}\,v,\Sigma)=0\,\mathrm{\ whenever\ }\,v\in\mathfrak{g}\,\mathrm{\ 
and\ }\,\Sigma:\mathfrak{g}\to\mathfrak{g}\,\mathrm{\ is\ 
}\,\by\hyp\mathrm{self}\hyp\mathrm{ad\-joint,} 
\end{equation}
due to skew-ad\-joint\-ness of $\,\mathrm{Ad}\,v$, with 
$\,(\hskip1.3pt,\hskip-.1pt)\,$ as in (\ref{acb}). For the same reason,
\begin{equation}\label{skw}
C_{\!jkl}\hskip5pt\mathrm{is\ totally\ skew}\hyp\mathrm{sym\-met\-ric\ in}
\hskip6ptj,k,l\hh,
\end{equation}
which amounts to skew-sym\-me\-try of 
$\,\mathfrak{g}\times\mathfrak{g}\times\mathfrak{g}
\ni(v,w,u)\mapsto\by([v,w],u)\in\bbR$.

\section{Further operations}\label{fo}\checked
\setcounter{equation}{0}
We continue using the same assumptions and notations as in Section~\ref{so}.

The Kil\-ling form $\,\by=\langle\,,\rangle$, being an 
$\,\bbF$-bi\-lin\-e\-ar inner product in $\,\mathfrak{g}$, naturally induces 
$\,\bbF$-bi\-lin\-e\-ar inner products, also denoted by $\,\langle\,,\rangle$, 
in the spaces $\,[\mathfrak{g}\nh^*]^{\otimes2}$ and 
$\,\ey=[\mathfrak{g}\nh^*]^{\otimes2}\nnh\otimes\mathfrak{g}$. (Their 
non\-de\-gen\-er\-a\-cy is obvious, since a $\,\by$-or\-tho\-nor\-mal basis of 
$\,\mathfrak{g}\,$ gives rise to $\,\langle\,,\rangle$-or\-tho\-nor\-mal bases 
in the other two spaces.) Explicitly,
\begin{equation}\label{inp}
\mathrm{i)}\hskip8pt\langle\sy,\ty\rangle
=\by^{\hs jp}\by^{\hs kq}\sy_{\!jk}\w\ty_{\!pq}\w\hs,\hskip15pt
\mathrm{ii)}\hskip8pt\langle\hs\nabla\nnh,\td\nabla\rangle
=\by^{\hs jp}\by^{\hs kq}\by_{lr}\w
\vg_{\hskip-2.2ptjk}^{\hs l}\td\vg_{\hskip-2.2ptpq}^{\hs r}\hs.
\end{equation}
for any $\,\sy,\ty\in[\mathfrak{g}\nh^*]^{\otimes2}$ and 
$\,\nabla\nh,\td\nabla
\in\ey=[\mathfrak{g}\nh^*]^{\otimes2}\nnh\otimes\mathfrak{g}$. Furthermore,
\begin{equation}\label{pse}
(\es\nh,\langle\,,\rangle)\hskip7pt\mathrm{and}\hskip7pt
(\et\nh,\langle\,,\rangle)\hskip7pt\mathrm{are\ 
inner}\hyp\mathrm{product\ spaces.}
\end{equation}
In other words, $\,\es\subset\ey\,$ and 
$\,\et\subset[\mathfrak{g}\nh^*]^{\otimes2}$ appearing in (\ref{sbs}) are 
nondegenerate sub\-spaces, which one easily sees applying (\ref{sin}) to our 
$\,\es\,$ and $\,\es'\nh=[\mathfrak{g}\nh^*]^{\wedge2}\nnh\otimes\mathfrak{g}$, 
with $\,\ev=\ey$, or, respectively, to $\,\et\,$ instead of $\,\es$, with 
$\,\es'\nh=[\mathfrak{g}\nh^*]^{\wedge2}$ and 
$\,\ev=[\mathfrak{g}\nh^*]^{\otimes2}\nnh$.

For $\,\nabla\in\ey=[\mathfrak{g}\nh^*]^{\otimes2}\nnh\otimes\mathfrak{g}$, the 
$\,\nabla\nnh${\smallit-grad\-i\-ent\/} of any $\,\sy\in\et\,$ is
\begin{equation}\label{sns}
\sj=\nabla\nh\sy\in\es,\mathrm{\ characterized\ by\ 
}\,\langle\hs\sj_v\w w,u\rangle=-\hs\sy(\nsuv,w)-\sy(v,\nsuw)\hh.
\end{equation}
(Note that $\,\nabla\nh\sy\,$ differs only slightly from the 
$\,\nabla\nnh$-co\-var\-i\-ant derivative of $\,\sy$, from which it arises 
via in\-dex-rais\-ing applied to the direction of differentiation.) Given 
such $\,\nabla\,$ and $\,\sy$, we also define
\begin{equation}\label{nes}
\td\nabla\nh=\sy\nabla\in\es,\mathrm{\ by\ declaring\ 
}\,\td\nabla_{\hskip-2.2ptv}\hh w\,\mathrm{\ to\ be\ }\,\Sigma(\nsvw),
\mathrm{\ for\ }\,\Sigma\,\mathrm{\ with\ (\ref{sgv}).}
\end{equation}
Equivalently, 
$\,\langle\td\nabla_{\hskip-2.2ptv}\hh w,u\rangle=\sy(\nsvw,u)\,$ and, in 
terms of components,
\begin{equation}\label{sjk}
\mathrm{a)}\hskip4ptS_{\nnh jk}^{\hs l}
=-\hh\by^{\hs lp}(\vg_{\hskip-2.2ptpj}^{\hs r}\sy\nh_{rk}\w
+\vg_{\hskip-2.2ptpk}^{\hs r}\sy_{\!jr}\w)\hskip5pt\mathrm{if}\hskip5pt\sj
=\nabla\nh\sy\hs,\hskip9pt\mathrm{b)}\hskip4pt\td\vg_{\hskip-2.2ptjk}^{\hs l}=
\sy_{\nh r}^{\hh l}\vg_{\hskip-2.2ptjk}^{\hs r}\hskip5pt\mathrm{if}\hskip5pt
\td\nabla\nh=\sy\nabla.
\end{equation}
For $\,\dj\,$ as in (\ref{tdv}), and any $\,\sy\in\et\nnh$, (\ref{sjk}.a) with 
$\,\sj=\dj\hh\sy\,$ and $\,\vg_{\hskip-2.2ptjk}^{\hs l}=C_{\!jk}\w{}^l\nnh/2$ 
gives, by (\ref{skw}), $\,2\hh S_{\nh rk}^{\hs r}
=-\hh C^{\hs l}{}_l\w{}^r\sy\nh_{rk}\w-C^{\hh j}{}_k\w{}^r\sy_{\!jr}\w=0$. 
Thus, from (\ref{cne}), 
\begin{equation}\label{cnd}
\cn\hh(\dj\hh\sy)\,=\,0\hskip13pt\mathrm{whenever}\hskip7pt\sy\in\et.
\end{equation}
In view of the line following (\ref{skw}), the three definitions: (\ref{sns}) 
of $\,\nabla\nh\sy\,$ (applied to $\,\nabla\nh=\dj$), (\ref{tdv}) of $\,\dj$, 
and 
(\ref{nes}) of $\,\sy\nabla\,$ (applied to $\,\sy=\ty\,$ and 
$\,\nabla\nh=\dj\hh\sy$), imply that for any $\,\sy,\ty\in\et\,$ and 
$\,\Sigma,\mathrm{T}\,$ related to them as in (\ref{sgv}), since $\,\Sigma\,$ 
is $\,\by$-self-ad\-joint,
\begin{equation}\label{dsv}
\mathrm{a)}\hskip5pt[\dj\hh\sy]_v\w v=[\Sigma v,v]\hh,\hskip15pt
\mathrm{b)}\hskip5pt[\hh\ty(\dj\hh\sy)]_v\w v
=\mathrm{T}[\Sigma v,v]\hs.
\end{equation}
The {\smallit normalized} $\,\by${\smallit-cur\-va\-ture operator\/} 
$\,\cu:\et\to\et\,$ is given by
\begin{equation}\label{oms}
[\hs\cu\hs\sy](v,w)\,=\,2\hs(\mathrm{Ad}\,v,\hs(\mathrm{Ad}\,w)\hh\Sigma)
\hskip13pt\mathrm{for}\hskip7ptv,w\in\mathfrak{g}\hs,
\end{equation}
with $\,(\hskip1.3pt,\hskip-.1pt)\,$ as in (\ref{acb}), any $\,\sy\in\et\nnh$, 
and $\,\Sigma\,$ corresponding to $\,\sy\,$ via (\ref{sgv}). We have
\begin{equation}\label{osj}
\mathrm{i)}\hskip5pt\cu\hh\by=2\hh\by\hs,\hskip10pt
\mathrm{ii)}\hskip5pt[\hs\cu\hs\sy]_{jk}\w
=2\hh C_{\!jq}\w{}^rC_{\nh ks}\w{}^q\sy\nh_{r}^{\hh s}\hs,\hskip10pt
\mathrm{iii)}\hskip5pt\mathrm{tr}\hskip2pt\cu
=-\dimf\mathfrak{g}\hh.
\end{equation}
In fact, (\ref{osj}.i) and (\ref{osj}.ii) are obvious from (\ref{bvw}) and 
(\ref{oms}). To verify (\ref{osj}.iii), 
we extend $\,\cu:\et\to\et\,$ to 
$\,\td\cu:[\mathfrak{g}\nh^*]^{\otimes2}\nh\to[\mathfrak{g}\nh^*]^{\otimes2}$ 
defined by the same formula (\ref{osj}.ii), so that 
$\,[\hs\td\cu\hs\sy]_{jk}\w=Z_{\nnh jk}^{\hh rs}\sy\nh_{rs}\s$ with 
$\,Z_{\nnh jk}^{\hh rs}=2\hh C_{\!jq}\w{}^rC_{\nh k}\w{}^{sq}$. Thus, 
$\,\mathrm{tr}\hskip2pt\td\cu=Z_{rs}^{\hh rs}=0\,$ (cf.\ (\ref{skw})). 
At the same time, $\,Z_{\nnh jk}^{\hh rs}=Z_{kj}^{sr}$, so that $\,\td\cu\,$ 
leaves the sub\-space $\,[\mathfrak{g}\nh^*]^{\wedge2}$ invariant, and 
acts on $\,\sy\in[\mathfrak{g}\nh^*]^{\wedge2}$ via 
$\,Z_{\nnh jk}^{\hh rs}\sy\nh_{rs}\s=(Z_{\nnh jk}^{\hh rs}-Z_{kj}^{\hh rs})\sy\nh_{rs}\s/2\,$ 
which, by the Ja\-cobi identity, gives 
$\,Z_{\nnh jk}^{\hh rs}\sy\nh_{rs}\s=H_{\nnh jk}^{\hh rs}\sy\nh_{rs}\s$ for 
$\,H_{\nnh jk}^{\hh rs}=2\hh C_{\nh kjq}\w{}C^{sqr}$. Due to skew-sym\-me\-try 
of $\,H_{\nnh jk}^{\hh rs}$ in both $\,j,k\,$ and $\,r,s\,$ (see (\ref{skw})), 
the trace of the restriction of $\,\td\cu\,$ to 
$\,[\mathfrak{g}\nh^*]^{\wedge2}$ is obtained as the sum of 
$\,H_{rs}^{\hh rs}$ over the pairs $\,r,s\,$ with $\,r<s$. In terms of 
the ordinary summing convention, this last trace equals 
$\,H_{rs}^{\hh rs}\nh/2=C_{\!srq}\w{}C^{sqr}=\by_{\nh s}^{\hh s}
=\dimf\mathfrak{g}$, 
as $\,\by_{jk}\w=C_{\!jq}\w{}^pC_{\nh kp}\w{}^q$ by (\ref{bvw}). Since 
$\,\mathrm{tr}\hskip2pt\td\cu=0$, (\ref{osj}.iii) follows.

The name used for $\,\cu\,$ reflects the fact that, if one treats the 
Kil\-ling form $\,\beta\,$ as a left-in\-var\-i\-ant 
pseu\-\hbox{do\hh-}\hskip0ptRiem\-ann\-i\-an Ein\-stein metric on a Lie group 
with the Lie algebra $\,\mathfrak{g}$, then $\,\cu\,$ is, up to a factor, the 
curvature operator of $\,\beta$, acting on symmetric $\,2$-ten\-sors 
\cite[Remark~1.4]{derdzinski-gal}.

\section{Two symmetric pairings}\label{ts}\checked
\setcounter{equation}{0}
We continue the discussion of Sections~\ref{so} --~\ref{fo}. The symbol 
$\,\{\hskip2pt\cdot\hskip2pt\}\,$ denotes two 
$\,\et\nnh$-val\-ued symmetric bi\-lin\-e\-ar mappings defined on 
$\,\ey\,$ (and, respectively, on $\,\et\hh$), characterized by the 
corresponding $\,\et\nnh$-val\-ued homogeneous quadratic functions:
\begin{equation}\label{ndn}
\{\nabla\hskip-2.5pt\cdot\hskip-2.5pt\nabla\}(v,w)=(\nav,\nnh\naw)\hs,
\hskip9pt\{\sy\hskip-2.2pt\cdot\nnh\sy\}(v,w)
=((\mathrm{Ad}\,v)\Sigma,(\mathrm{Ad}\,w)\Sigma)
\end{equation}
for $\,\nabla\in\ey,\,\sy\in\et$ and $\,v,w\in\mathfrak{g}$, with 
$\,\Sigma,\,(\hskip1.3pt,\hskip-.1pt)\,$ as in (\ref{sgv}) and (\ref{acb}). 
Equivalently,
\begin{equation}\label{ghg}
2\hh\{\nabla\hskip-2.5pt\cdot\hskip-2.5pt\td\nabla\}_{\!jk}\w\nh
=\nh\vg_{\hskip-2.2ptrj}^{\hs s}\td\vg_{\hskip-2.2ptsk}^{\hs r}\nh
+\nh\td\vg_{\hskip-2.2ptrj}^{\hs s}\vg_{\hskip-2.2ptsk}^{\hs r}\hs,\hskip8pt
2\hh\{\sy\hskip-2.2pt\cdot\nnh\ty\}_{\!jk}\w\nh
=\nh C_{\!jr}\w{}^p\sy_{\!q}^{\hh r}C_{\nh ks}\w{}^q\ty_{\nh p}^{\hh s}\nh
+\nh C_{\!jr}\w{}^p\ty_{\nh q}^{\hh r}C_{\nh ks}\w{}^q\sy_{\!p}^{\hh s}
\end{equation}
whenever $\,\nabla\nnh,\td\nabla
\in\ey=[\mathfrak{g}\nh^*]^{\otimes2}\nnh\otimes\mathfrak{g}\,$ and 
$\,\sy,\ty\in\et\nnh$. By (\ref{oms}), for $\,\sy\in\et\nnh$,
\begin{equation}\label{tsb}
2\hh\{\sy\hskip-2.2pt\cdot\nnh\by\}\,=\,\cu\hs\sy\hs.
\end{equation}
Being symmetric, $\,\{\sy\hskip-2.2pt\cdot\nnh\ty\}\,$ is uniquely 
characterized by the requirement that
\begin{equation}\label{stw}
\{\sy\hskip-2.2pt\cdot\nnh\ty\}(w,w)
=((\mathrm{Ad}\,w)\Sigma,(\mathrm{Ad}\,w)\mathrm{T})\hskip12pt\mathrm{for\ 
all}\hskip7ptw\in\mathfrak{g}\hh,
\end{equation}
with $\,\Sigma,\mathrm{T}\,$ as in (\ref{dsv}.b). Lemma~\ref{unimo}(d) and 
(\ref{ndn}) imply that a left-in\-var\-i\-ant tor\-sion-free uni\-mod\-u\-lar 
connection $\,\nabla\nh=\dj+\hs\sj$, with $\,\sj\in\es$, has the Ric\-ci tensor
\begin{equation}\label{rne}
\rho\nnh^\nabla\,=\,\,-\{\nabla\hskip-2.5pt\cdot\hskip-2.5pt\nabla\}\hs,
\end{equation}
cf.\ Remark~\ref{ssuni}. Note that, due to (\ref{bvw}), (\ref{tdv}), 
(\ref{rne}), (\ref{skw}) and (\ref{dsv}.a),
\begin{equation}\label{fdd}
\mathrm{a)}\hskip4pt4\hh\{\dj\nnh\cdot\hskip-2.1pt\dj\}=\by\hs,\hskip5pt
\mathrm{that\ is,}\hskip4pt\by_{jk}\w=C_{\nh pj}\w{}^qC_{\!qk}\w{}^p,\hskip7pt
\mathrm{b)}\hskip4pt4\hh\rho^\dj\hh=\,-\hs\by\hs,\hskip8pt\mathrm{c)}
\hskip4pt\dj\by\,=\,0\hs. 
\end{equation}
Finally, for the inner products in (\ref{pse}), any $\,\sj\in\es$, and any 
$\,\sy\in\et\nnh$,
\begin{equation}\label{dss}
\mathrm{a)}\hskip8pt
\langle\hs\sj,\dj\hh\sy\rangle\,
=\,2\hh\langle\{\dj\nnh\cdot\hskip-1.3pt\sj\},\sy\rangle\hs,\hskip14pt
\mathrm{b)}\hskip8pt
\{\dj\nnh\cdot\nnh(\sj\by)\}\,=\,\{\dj\nnh\cdot\hskip-1.3pt\sj\}\hs.
\end{equation}
In fact, since $\,S_{\nnh jk}^{\hs l}=S_{\nh kj}^{\hs l}$, 
$\,\sy_{\!jk}\w=\sy_{\!kj}\w$, $\,2D_{\!jk}^{\hs l}=C_{\!jk}\w{}^l\nnh$, 
(\ref{inp}.ii), (\ref{sjk}.a) and (\ref{skw}) give 
$\,\langle\hs\sj,\dj\hh\sy\rangle
=-\hh\by^{\hs jp}\by^{\hs kq}S_{\nnh jk}^{\hs l}C_{\nh lp}\w{}^r\sy\nh_{rq}\w
=-S_{\nnh jk}^{\hs l}C_{\nh l}\w{}^{\hs jr}\sy\nh_{r}^{\hs k}
=S_{\nnh jk}^{\hs l}C_{\nh lr}\w{}^{\hs j}\sy^{\hs kr}
=2\hh\{\dj\nnh\cdot\hskip-1.3pt\sj\}_{\nh kr}\w\sy^{\hs kr}
=2\hh\langle\{\dj\nnh\cdot\hskip-1.3pt\sj\},\sy\rangle$. 
For similar reasons, one has $\hs4\hh\{\dj\nnh\cdot\nnh(\sj\by)\}_{\!jk}\w
\nnh=\nnh-\hh C^{\hs p}{}_j\w{}^{\hs q}(S_{\nh pq}^{\hs r}\by_{rk}\w
\nh+\nh S_{\nh pk}^{\hs r}\by_{qr}\w)\nh
-C^{\hs p}{}_k\w{}^{\hs q}(S_{\nh pq}^{\hs r}\by_{rj}\w
+S_{\nh pj}^{\hs r}\by_{qr}\w)$. As a consequence of (\ref{skw}), the first 
and third of the four resulting terms vanish, and the other two add up to 
$\,C_{\nh rj}\w{}^{\hs p}\nh S_{\nh pk}^{\hs r}\nnh
+C_{\nh rk}\w{}^{\hs p}S_{\nh pj}^{\hs r}\nnh
=\nh4\hh\{\dj\nnh\cdot\hskip-1.3pt\sj\}_{\!jk}\w$.

\section{The Chri\-stof\-fel isomorphism}\label{ci}\checked
\setcounter{equation}{0}
With the assumptions and notations as in the last three sections, let us 
introduce the {\smallit Chri\-stof\-fel isomorphism\/} $\,\ch:\es\to\es\,$ by 
requiring that, whenever $\,v,w,u\in\mathfrak{g}$,
\begin{equation}\label{chr}
2\langle u,\td\sj_vw\rangle\,=\,\langle v,\sj_wu\rangle\,
+\,\langle w,\sj_vu\rangle\,-\,\langle u,\sj_vw\rangle\hskip10pt
\mathrm{for}\hskip7pt\td\sj=\ch\hs\sj\hs.
\end{equation}
We give two more descriptions of $\ch$, one based on applying, to 
$\,\nabla\nh=\sj\,$ and $\,\sy=\by$, the definition (\ref{sns}) of the 
$\,\nabla\nnh$-grad\-i\-ent, the other using components and (\ref{sjk}.a):
\begin{equation}\label{chc}
\mathrm{i)}\hskip8pt2\hs\ch\hs\sj=-\hh\sj\by-\sj\hs,\hskip12pt
\mathrm{ii)}\hskip8pt2(\ch\hs\sj)_{\nnh jk}^{\hs l}
=\by^{\hs lp}(S_{\nh pj}^{\hs r}\by_{rk}\w
+S_{\nh pk}^{\hs r}\by_{rj}\w)-S_{\nnh jk}^{\hs l}\hs.
\end{equation}
For yet another description of $\ch$, see Remark~\ref{dscch}. If 
$\,\sj\in\es\,$ and $\,\sy\in\et\nnh$, then
\begin{equation}\label{chs}
\mathrm{a)}\hskip8pt\ch(\dj\hh\sy)\,=\,-\hs\dj\hh\sy\hh,\hskip18pt
\mathrm{b)}\hskip8pt\ch(\sj\sy)\,=\,-\hs\sy\sj\hh,
\end{equation}
as one sees combining the line following (\ref{skw}) and (\ref{chr}) with 
(\ref{dsv}.a) or, respectively, applying the definitions of 
$\,\nabla\nh\sy\,$ and $\,\sy\nabla\,$ in (\ref{sns}) -- (\ref{nes}) to 
$\,\nabla\nh=\sj$. Also,
\begin{equation}\label{fsb}
\mathrm{a)}\hskip8pt\ch(\sj\by)\,=\,-\hh\sj\hh,\hskip18pt
\mathrm{b)}\hskip8pt\{\dj\nnh\cdot\nnh(\ch\hs\sj)\}\,
=\,-\hh\{\dj\nnh\cdot\hskip-1.3pt\sj\}\hh.
\end{equation}
In fact, (\ref{fsb}.a) is a special case of (\ref{chs}.b), for $\,\sy=\by$, 
while (\ref{chc}.i) and (\ref{dss}.b) give 
$\,2\hh\{\dj\nnh\cdot\nnh(\ch\hs\sj)\}
=-\hh\{\dj\nnh\cdot\nnh(\sj\by)\}-\{\dj\nnh\cdot\hskip-1.3pt\sj\}
=-2\hh\{\dj\nnh\cdot\hskip-1.3pt\sj\}$, proving (\ref{fsb}.b).

Note that $\,\ch\,$ is actually an isomorphism, as a consequence of 
(\ref{fsb}.a).

In the next lemma, $\,\dj\hh\ly$, or $\,\dj\hh\chi\,$ or, respectively, 
$\,(\hh\dj\hh\ly)\hh\fy\,$ is given by (\ref{sns}) with $\,\nabla\nh=\dj\,$ 
and $\,\sy=\ly$, or $\,\nabla\nh=\dj\,$ and $\,\sy=\chi\,$ or, respectively, 
$\,\nabla\nh=\dj\hh\ly\,$ and $\,\sy=\fy$. Similarly, 
$\,\fy\hh(\hh\dj\hh\ly)\,$ is defined as in (\ref{nes}) with $\,\sy=\fy\,$ and 
$\,\nabla\nh=\dj\hh\ly\hh$.
\begin{lem}\label{psidl}{\smallit 
If\/ $\,\fy,\ly,\chi\in\et\nnh$, \ for\/ $\,\et\,$ in\/ {\rm(\ref{sbs})}, 
and\/ $\,\fy\hh(\hh\dj\hh\ly)=\dj\hh\chi\hs$, then\/ 
$\,(\hh\dj\hh\ly)\hh\fy=\dj\hh\chi$.
}
\end{lem}
\begin{proof}From (\ref{chs}) with $\,\sj=\dj\hh\ly\,$ we get 
$\,\ch\hh[(\hh\dj\hh\ly)\hh\fy-\dj\hh\chi\hs]=\dj\hh\chi-\fy\hh(\hh\dj\hh\ly)$.
\end{proof}
The operation $\,\et\nh\times\et\ni(\sy,\ty)\mapsto\sy\ca\ty\in\et\,$ is 
defined by requiring it to correspond, under the identification in 
(\ref{sgv}), to 
$\,(\Sigma,\td\Sigma)\mapsto(\Sigma\td\Sigma+\td\Sigma\Sigma)/2\,$ 
(one-half of the an\-ti\-com\-mu\-ta\-tor of $\,\by$-self-ad\-joint 
linear en\-do\-mor\-phisms of $\,\mathfrak{g}$). Thus, $\,\by\ca\sy=\sy$. 
Any nondegenerate $\,\gy\in\et\,$ has an inverse $\,\gy^{-1}\in\et\,$ with 
$\,\gy\hh\ca\gy^{-1}\nh=\by$.

Every $\,\sy\in\et\,$ satisfies the following relations, in which, 
once again, $\,\dj,\hh\by$ and $\,\cu\,$ are the standard connection, 
Kil\-ling form, and normalized $\,\by$-cur\-va\-ture operator, with 
(\ref{tdv}), (\ref{bvw}) and (\ref{oms}), $\,\{\hskip3pt\cdot\hskip3pt\}\,$ 
stands for the two operations given by (\ref{ndn}) or (\ref{ghg}), while 
$\,\ca\,$ is defined as in the last paragraph:
\begin{equation}\label{edd}
\begin{array}{l}
\mathrm{a)}\hskip4pt(\dj\hh\sy)\by\,=\,\dj\hh\sy\hh,\\
\mathrm{b)}\hskip4pt8\hh\{\dj\nnh\cdot\nnh(\dj\hh\sy)\}\,
=\,(\hh\cu\,-\,2\hskip1.4pt\mathrm{Id}\hh)\hs\sy\hh,\\
\mathrm{c)}\hskip4pt4\hh\{(\dj\hh\sy)\nnh\cdot\nnh(\dj\hh\sy)\}\hs
=\{\sy\hskip-2.2pt\cdot\nnh\sy\}\hs
+[(\hh\mathrm{Id}-\cu\hh)\hs\sy]\ca\hs\sy\hh,\\
%-\hs(\cu\hs\sy)\hs\sy\hs-\hs\sy\hs(\cu\hs\sy)\hh,\\
\mathrm{d)}\hskip4pt4\hh\{(\dj+\dj\hh\sy)\nnh\cdot\nnh(\dj+\dj\hh\sy)\}
=\by-\sy+\{\sy\hskip-2.2pt\cdot\nnh\sy\}
+[(\cu-\mathrm{Id})\hs\sy]\ca\hs(\by-\sy)\hs.
\hskip-12pt
\end{array}
\end{equation}
In fact, $\,(\dj\hh\sy)\by=\dj\hh\sy\,$ since 
$\,\ch\hs[(\dj\hh\sy)\by\hh]=\ch\hh(\dj\hh\sy)\,$ by (\ref{fsb}.a) (for 
$\,\sj=\dj\hh\sy$) and (\ref{chs}.a). Next, 
$\,-8\hh\{\dj\nnh\cdot\nnh(\dj\hh\sy)\}_{\!jk}\w=
C_{\nh pj}\w{}^q(C^{\hh p}{}_{\!qr}\w\sy_{\!k}^{\hh r}
+C^{\hh p}{}_{\!kr}\w\sy_{\!q}^{\hh r})
+C_{\nh pk}\w{}^q(C^{\hh p}{}_{\!qr}\w\sy_{\!j}^{\hh r}
+C^{\hh p}{}_{\!jr}\w\sy_{\!q}^{\hh r})$ due to (\ref{sjk}.a) with 
$\,\vg_{\hskip-2.2ptjk}^{\hs l}=D_{\!jk}^{\hs l}=C_{\!jk}\w{}^l\nnh/\hh2\,$ 
and (\ref{ghg}). In view of (\ref{fdd}.a) and (\ref{skw}), the first and third 
of the resulting four terms are both equal to 
$\,\by_{jr}\w\sy_{\!k}^{\hh r}=\by_{kr}\w\sy_{\!j}^{\hh r}=\sy_{\!jk}\w$, and 
their sum is $\,2\hs\sy_{\!jk}\w$. The other two terms are also equal and, 
by (\ref{osj}.ii), add up to $\,-\hs[\hs\cu\hs\sy]_{jk}\w$, which proves 
(\ref{edd}.b). Also, 
$\,8\hh\{(\dj\hh\sy)\nnh\cdot\nnh(\dj\hh\sy)\}_{\!jk}\w$ equals
\[
2(C^{\hs q}{}_p\w{}^r\sy_{\nh rj}\w+C^{\hs q}{}_j\w{}^r\sy_{\nh pr}\w)
(C^{\hs p}{}_q\w{}^s\sy_{\nh sk}\w+C^{\hs p}{}_k\w{}^s\sy_{\nh qs}\w)\hh,
\]
that is, 
$\,2\hh C^{\hs q}{}_p\w{}^rC^{\hs p}{}_q\w{}^s\sy_{\nh rj}\w\sy_{\nh sk}\w
-2\hh C_{\nh rp}\w{}^qC_{\nh ks}\w{}^p\sy_{\!q}^s\sy_{\!j}^r
-2\hh C_{\nh sq}\w{}^pC_{\!jr}\w{}^q\sy_{\!p}^r\sy_{\!k}^s
+2\hh C_{\!jr}\w{}^q\sy_{\!p}^rC_{\nh ks}\w{}^p\sy_{\!q}^s
=2\hs\sy_{\!j}^{\hh s}\sy_{\nh sk}\w
-\hs[\hs\cu\hs\sy]_{rk}\w\sy_{\!j}^r
-\hs[\hs\cu\hs\sy]_{sj}\w\sy_{\!k}^s
+2\hh\{\sy\hskip-2.2pt\cdot\nnh\sy\}_{\!jk}\w$, cf.\ (\ref{ghg}), 
(\ref{fdd}.a), (\ref{tdv}), (\ref{skw}), (\ref{sjk}.a) and 
(\ref{osj}.ii). This yields (\ref{edd}.c). Finally, 
(\ref{fdd}.a) and (\ref{edd}.b) -- (\ref{edd}.c) imply (\ref{edd}.d). 

Next, for $\,\sy,\ty\in\et\nnh$, applying (\ref{dss}.a) to $\,\sj=\dj\hh\ty\,$ 
we obtain, from (\ref{chs}.b), 
\begin{equation}\label{dsd}
4\hh\langle\hs\dj\hh\ty,\dj\hh\sy\rangle\,
=\,\langle\hs\cu\hs\ty-2\ty,\sy\rangle\hs.
\end{equation}
\begin{rem}\label{dscch}Another description of $\,\ch$, obvious from 
(\ref{chr}), reads: $\,\ch\,$ is di\-ag\-o\-nal\-izable, has the eigenvalues 
$\,1/2\,$ and $\,-1$, with the respective eigen\-spaces consisting of those 
$\,\sj\in\es\,$ for which $\,\langle u,\sj_vw\rangle\,$ is totally symmetric 
in $\,u,v,w\in\mathfrak{g}$ or, respectively, $\,\langle u,\sj_vw\rangle\,$ 
summed cyclically over $\,u,v,w\in\mathfrak{g}\,$ yields $\,0$.
\end{rem}
\begin{rem}\label{chfor}We identify nondegenerate elements of $\,\et\hs$ with 
left-in\-var\-i\-ant pseu\-\hbox{do\hskip1pt-}\hskip0ptRiem\-ann\-i\-an 
metrics $\,\gy\,$ on $\,\gj$. According to Remark~\ref{lccon}, the 
Le\-\hbox{vi\hh-}\hskip0ptCi\-vi\-ta connection of any such $\,\gy\,$ is the unique 
$\,\nabla\nh=\dj+\hs\sj$, with 
$\,\sj\in[\mathfrak{g}\nh^*]^{\odot2}\nnh\otimes\mathfrak{g}$, such that 
$\,\nabla\gy=0$. Writing the last condition as $\,\dj\gy+\sj\gy=0$, and 
using (\ref{chs}), we get $\,\dj\gy=-\hh\gy\sj$, that is, the classical 
{\smallit Chri\-stof\-fel formula\/} $\,\nabla\nh=\dj+\hs\sj$, for 
$\,\sj=-\hh\gy^{-1}(\dj\gy)$.
\end{rem}

\section{Einstein and weak\-\hbox{ly\hh-}\hskip0ptEin\-stein 
connections}\label{ec}\checked
\setcounter{equation}{0}
We adopt the assumptions and notations of Sections~\ref{so} --~\ref{ts}, 
so that $\,\mathfrak{g}\,$ is our fixed sem\-i\-sim\-ple Lie algebra over 
$\,\bbF=\bbR\,$ or $\,\bbF=\bbC$, associated with a real/com\-plex Lie group 
$\,\gj$, while $\,[\mathfrak{g}\nh^*]^{\otimes2}\nnh\otimes\mathfrak{g}\,$ is 
identified with the space of left-in\-var\-i\-ant connections on $\,\gj\,$ 
(assumed hol\-o\-mor\-phic if $\,\bbF=\bbC$), and $\,\dj+\es\,$ is its 
af\-fine subspace formed by all 
$\,\nabla\in[\mathfrak{g}\nh^*]^{\otimes2}\nnh\otimes\mathfrak{g}\,$ which are 
tor\-sion-free. Now, for $\,\ee\nh,\,\eu\nh,\ew\,$ defined below,
\begin{equation}\label{inc}
\ee\,\subset\,\,\eu\,\subset\,\ew\,\subset\,\dj+\es\,
\subset\,[\mathfrak{g}\nh^*]^{\otimes2}\nnh\otimes\mathfrak{g}\hh.
\end{equation}
Here $\,\ee\,$ denotes the set of {\smallit Einstein connections\/} in 
$\,\mathfrak{g}$, that is, of the 
Le\-\hbox{vi\hh-}\hskip0ptCi\-vi\-ta connections of left-in\-var\-i\-ant 
Ein\-stein metrics on $\,\gj\,$ (again, assumed hol\-o\-mor\-phic when 
$\,\bbF=\bbC$). The subset $\,\ew\,$ of $\,\dj+\es\,$ consists, in turn, of 
{\smallit weak\-\hbox{ly\hh-}\hskip0ptEin\-stein connections\/} in 
$\,\mathfrak{g}$, by which we mean all $\,\nabla\in\dj+\es\,$ such that the 
left-in\-var\-i\-ant symmetric $\,2$-ten\-sor field 
$\hs\{\nabla\hskip-2.5pt\cdot\hskip-2.5pt\nabla\}\hs$ on $\hs\gj\hs$ is 
$\hs\nabla\nnh$-par\-al\-lel. Finally, $\,\eu\subset\dj+\es\,$ is formed by 
those $\,\nabla\,$ which are uni\-mod\-u\-lar and have a 
$\,\nabla\nnh$-par\-al\-lel Ric\-ci tensor.

The inclusions $\,\ee\subset\,\eu\subset\ew\,$ follow from Lemma~\ref{unimo}(b) 
and (\ref{rne}). Next, we define a nonlinear mapping $\,\hz:\es\to\es\,$ and a 
linear en\-do\-mor\-phism $\,\dz\,$ of $\,\es\,$ by
\begin{equation}\label{ths}
\mathrm{i)}\hskip5pt\hz\hs(\sj)
=4\hs(\dj+\hs\sj)\hh\{(\dj+\hs\sj)\nnh\cdot\nnh(\dj+\hs\sj)\}\hh,\hskip10pt
\mathrm{ii)}\hskip5pt\dz\sj=8\hs\dj\hh\{\dj\nnh\cdot\hskip-1.3pt\sj\}
+\sj\by\hh,
\end{equation}
with $\,\es\,$ as in (\ref{sbs}). According to (\ref{fdd}.a), 
$\,4\hh\{\dj\nnh\cdot\hskip-2.1pt\dj\}=\by$, and hence
\begin{equation}\label{fds}
\mathrm{a)}\hskip5pt4\hh\{(\dj+\hs\sj)\nnh\cdot\nnh(\dj+\hs\sj)\}=\by+\sy\hh,
\hskip7pt\mathrm{where}\hskip5pt\mathrm{b)}\hskip5pt
\sy=8\hh\{\dj\nnh\cdot\hskip-1.3pt\sj\}+4\hh\{\sj\nnh\cdot\nnh\sj\nh\}\hh.
\end{equation}
The operations used in (\ref{ths}.i) and (\ref{fds}.a) are the 
$\,\nabla\nnh$-grad\-i\-ent, cf.\ (\ref{sns}), and the first pairing in 
(\ref{ndn}); thus, setting $\,\nabla\nh=\dj+\hs\sj$, we have
\begin{equation}\label{hsn}
\hz\hs(\sj)\,=\,4\hh\nabla\{\nabla\hskip-2.5pt\cdot\hskip-2.5pt\nabla\}\hh,
\end{equation}
while for the Ric\-ci tensor $\,\rho\nnh^\nabla$ of $\,\nabla\nh$, one 
obtains, from (\ref{uni}), (\ref{rne}) and (\ref{fds}),
\begin{equation}\label{hen}
\hz\hs(\sj)\,=\,-\hh4\nabla\nnh\rho\nnh^\nabla\hskip10pt\mathrm{and}\hskip8pt
\rho\nnh^\nabla\nh=-\hh(\by+\sy)/4\hskip8pt\mathrm{if}\hskip6pt
\cn\hh\sj=0\hskip8pt\mathrm{and}\hskip6pt\nabla\nh=\dj+\hs\sj\hh.
\end{equation}
By (\ref{ths}) -- (\ref{fds}), $\,\hz\hs(\sj)=\dj\hh\sy+\sj\sy+\sj\by$, so 
that, as $\,\dj\by=0\,$ in (\ref{fdd}.c),
\begin{equation}\label{hse}
\hz\hs(\sj)=\dz\sj+4\hh\dj\hs\{\sj\nnh\cdot\nnh\sj\nh\}
+\sj\sy\hh,\hskip11pt\mathrm{with}\hskip7pt
\sy=8\hh\{\dj\nnh\cdot\hskip-1.3pt\sj\}+4\hh\{\sj\nnh\cdot\nnh\sj\nh\}\hh.
\end{equation}
Consequently, $\,\dz\,$ equals the differential of $\,\hz\,$ at $\,0$, that is
\begin{equation}\label{dhz}
\mathrm{d}\hh\hz_{\hh0}\w\,
=\,\dz\hh,\hskip24pt\mathrm{while}\hskip9pt\hz\hs(0)\hs=\hs0\hh,
\end{equation}
since (\ref{hse}) expresses $\,\hz\hs(\sj)\,$ as $\,\dz\sj\,$ plus some 
terms quadratic and cubic in $\,\sj$.

The sets $\,\,\eu\,$ and $\,\ew\,$ appearing in (\ref{inc}) can now be described 
as follows:
\begin{equation}\label{eud}
\begin{array}{l}
\mathrm{\phantom ii)}\hskip7pt\eu\,=\,\dj\,+\,(\ez\cap\mathrm{Ker}\,\cn)\,
=\,\hs\ew\hh\cap\mathrm{Ker}\,\cn\hh,\hskip13pt\mathrm{and}\\
\mathrm{ii)}\hskip7pt\ew\,=\,\dj\,+\,\ez\hh,\hskip20pt\mathrm{where}\hskip8pt
\mathrm{iii)}\hskip7pt\ez\,=\,\hz\hs^{-\nh1}(0)\hh,
\end{array}
\end{equation}
as (\ref{hsn}) yields (\ref{eud}.ii), while (\ref{eud}.ii), (\ref{uni}), 
Remark~\ref{ssuni} and (\ref{rne}) give (\ref{eud}.i).
\begin{lem}\label{equiv}{\smallit 
With the assumptions and notations as in Sections\/~{\rm\ref{so}} 
--~{\rm\ref{ts}}, for any\/ 
$\,\sj\in\es$ and\/ $\,\sy\in\et\nnh$, the following two conditions are 
equivalent\/{\rm:}
\begin{enumerate}
  \def\theenumi{{\rm\roman{enumi}}}
\item[{\rm(a)}] $\dz\sj+4\hh\dj\hs\{\sj\nnh\cdot\nnh\sj\nh\}+\sj\sy=0$, \ 
while\/ $\,\cn\hh\sj=0\,$ and\/ $\,\sy=8\hh\{\dj\nnh\cdot\hskip-1.3pt\sj\}
+4\hh\{\sj\nnh\cdot\nnh\sj\nh\}$,
\item[{\rm(b)}] $\nabla\nh=\dj+\hs\sj\,$ is a uni\-mod\-u\-lar tor\-sion-free 
connection in\/ $\,\mathfrak{g}\,$ and its Ric\-ci tensor, given by\/ 
$\,\rho=-\hh(\by+\sy)/4$, is\/ $\,\nabla\nnh$-par\-al\-lel\/{\rm;} 
consequently, $\,\nabla\nh\in\eu$.
\end{enumerate}
}
\end{lem}
\begin{proof}First, (a) implies (b), since (\ref{uni}) with $\,\cn\hh\sj=0\,$ 
gives uni\-mod\-u\-lar\-i\-ty of $\,\nabla\nnh$, and (\ref{hen}) -- 
(\ref{hse}) show that $\,\rho=-\hh(\by+\sy)/4\,$ is 
$\,\nabla\nnh$-par\-al\-lel. Conversely, assuming (b), we obtain (a) from 
(\ref{uni}), (\ref{rne}), (\ref{fds}) and (\ref{hen}) -- (\ref{hse}).
\end{proof}
\begin{lem}\label{sfadj}{\smallit 
For the in\-ner-prod\-uct spaces\/ {\rm(\ref{pse})}, and\/ $\,\ch,\dz\,$ as 
in\/ {\rm(\ref{chr})}, {\rm(\ref{ths}.ii)},
\begin{enumerate}
  \def\theenumi{{\rm\roman{enumi}}}
\item[{\rm(i)}] $\ch,\dz\,$ as well as\/ $\,\sj\mapsto\sj\by\,$ and\/ 
$\,\sj\mapsto\dj\hh\{\dj\nnh\cdot\hskip-1.3pt\sj\}\hs$ are self-ad\-joint 
en\-do\-mor\-phisms of\/ $\,\es$,
\item[{\rm(ii)}] $\cu:\et\to\et\,$ defined by\/ {\rm(\ref{oms})} is 
self-ad\-joint.
\end{enumerate}
}
\end{lem}
\begin{proof}First, $\,\ch\,$ is self-ad\-joint due to the first line of 
(\ref{dia}) and Remark~\ref{dscch}; so is the en\-do\-mor\-phism 
$\,\sj\mapsto\sj\by\,$ equal, by (\ref{chc}.i), to $\,-2\hs\ch-\mathrm{Id}$. 
The same follows for the 
operator $\,\sj\mapsto\dj\hh\{\dj\nnh\cdot\hskip-1.3pt\sj\}\,$ obtained, up to 
a factor, as the composite of 
$\,\sj\mapsto2\hh\{\dj\nnh\cdot\hskip-1.3pt\sj\}\,$ with its adjoint 
$\,\sy\mapsto\dj\hh\sy\,$ (see (\ref{dss}.a)). By (\ref{ths}.ii), the last two 
conclusions give (i) for $\,\dz$. Finally, (\ref{inp}.i), (\ref{osj}.ii) and 
(\ref{skw}) imply symmetry of the expression 
$\,\,\langle\hs\cu\hs\sy,\ty\rangle
=2\hh\ty^{\hs jk}C_{\!jqp}\w C_{\nh kr}\w{}^q\sy^{\hs pr}$ in $\,\sy,\ty\in\et\nnh$, 
proving (ii).
\end{proof}

\begin{rem}\label{kssez}The quadratic mapping 
$\,\kz:\es\times\et\to\es\times\et\,$ in Section~\ref{oa} is given by 
$\,\kz\hh(\sj,\sy)
=(\dz\sj+4\hh\dj\hs\{\sj\nnh\cdot\nnh\sj\nh\}+\sj\sy,\,
\sy-8\hh\{\dj\nnh\cdot\hskip-1.3pt\sj\}-4\hh\{\sj\nnh\cdot\nnh\sj\nh\})$, and 
$\,\lz(\sj,\sy)=(\dz\sj,\hs\sy)$. The correspondence (\ref{bij}) sends 
$\,\dj+\hs\sj\,$ with $\,\hz\hs(\sj)=0\,$ to $\,(\sj,\sy)$, for $\,\sy\,$ as 
in (\ref{hse}).
\end{rem}

\section{The nondegeneracy condition}\label{nc}\checked
\setcounter{equation}{0}
We continue the discussion of Sections~\ref{so} --~\ref{ec}. By 
(\ref{ths}.ii), (\ref{edd}.a) and (\ref{edd}.b),
\begin{equation}\label{dio}
\dz(\dj\hh\sy)\,=\,\dj\hh[(\cu-\mathrm{Id})\hs\sy]
\hskip13pt\mathrm{whenever}\hskip6pt\sy\in\et.
\end{equation}
Using the Chri\-stof\-fel isomorphism $\,\ch$, given by (\ref{chr}), we obtain
\begin{equation}\label{krd}
\mathrm{Ker}\,\dz\,=\,\{\dj\hh\ty:\ty\in\mathrm{Ker}\,(\cu-\mathrm{Id})\}\hs.
\end{equation}
In fact, applying $\,\ch$, we see that, by (\ref{ths}.ii), (\ref{chs}.a) with 
$\,\sy=\{\dj\cdot\sj\}$, and (\ref{fsb}.a), $\,\dz\sj=0\,$ if 
and only if $\,\sj=\dj\hh\ty\,$ for 
$\,\ty=-8\hs\{\dj\nnh\cdot\hskip-1.3pt\sj\}$, which is in turn equivalent to 
$\,\sj=\dj\hh\ty\,$ for $\,\ty=-8\hh\{\dj\nnh\cdot\nnh(\hh\dj\hh\ty)\}$, and 
hence, in view of (\ref{edd}.b), amounts to $\,\sj=\dj\hh\ty$ for some 
$\,\ty\in\mathrm{Ker}\,(\cu-\mathrm{Id})$. From (\ref{krd}) we further 
conclude that
\begin{equation}\label{ttd}
\begin{array}{l}
\ty\mapsto\dj\hh\ty\,\,\mathrm{\ is\ a\ linear\ isomorphism\ 
}\,\,\hs\mathrm{Ker}\,(\cu-\mathrm{Id})\hs\to\hs\mathrm{Ker}\,\dz,\hs\mathrm{\ with}\\
\mathrm{the\ inverse\ isomorphism\ }
\,\mathrm{Ker}\,\dz\ni\sj\mapsto-8\hs\{\dj\nnh\cdot\hskip-1.3pt\sj\}
\in\mathrm{Ker}\,(\cu-\mathrm{Id})\hs,
\end{array}
\end{equation}
since, by (\ref{krd}), $\,\ty\mapsto\dj\hh\ty\,$ maps 
$\,\mathrm{Ker}\,(\cu-\mathrm{Id})\,$ onto $\,\mathrm{Ker}\,\dz$, while 
(\ref{edd}.b) implies its injectivity and describes its inverse.

From now on we will assume the {\smallit nondegeneracy condition\/}: for 
$\,\cu\,$ with (\ref{oms}),
\begin{equation}\label{knd}
\mathrm{Ker}\,(\cu-\mathrm{Id})\hskip9pt\mathrm{is\ a\ nondegenerate\ 
sub\-space\ of}\hskip7pt\et
\end{equation}
in the sense of Section~\ref{la}. This is no restriction of generality: in 
Lemma~\ref{noeig}(f) we prove (\ref{knd}) for the Lie algebras of the groups 
(\ref{gsu}), while the remaining simple Lie algebras have 
$\,\mathrm{Ker}\,(\cu-\mathrm{Id})=\{0\}\,$ (see Remark~\ref{kroez}), which 
again yields (\ref{knd}).
\begin{lem}\label{drsum}{\smallit 
Condition\/ {\rm(\ref{knd})} implies that
\begin{equation}\label{opl}
\begin{array}{l}
\mathrm{a)}\hskip6pt\et=[(\cu-\mathrm{Id})(\et)]
\oplus\mathrm{Ker}\,(\cu-\mathrm{Id})\hs,\hskip18pt\mathrm{b)}\hskip6pt
\es=[\hh\dz(\es)]\oplus\mathrm{Ker}\,\dz\hs,\\
\mathrm{c)}\hskip6pt\dz:\dz(\es)\to\dz(\es)\,\,\mathrm{\ is\ an\ isomorphism.}
\end{array}
\end{equation}
}
\end{lem}
\begin{proof}Formula (\ref{dsd}) for 
$\,\sy,\ty\in\mathrm{Ker}\,(\cu-\mathrm{Id})\,$ becomes 
$\,4\hh\langle\hs\dj\hh\ty,\dj\hh\sy\rangle=-\hh\langle\ty,\sy\rangle$, so 
that, by (\ref{ttd}) and (\ref{knd}), $\,\mathrm{Ker}\,\dz\,$ is a 
nondegenerate sub\-space of $\,\es$. In view of Lemma~\ref{sfadj}, 
$\,\dz:\es\to\es\,$ and $\,\cu-\mathrm{Id}:\et\to\et\,$ are self-ad\-joint. 
Non\-de\-gen\-er\-a\-cy of their kernels and (\ref{imd}.b) thus yield 
(\ref{opl}.a\hs-\hh b). Now (\ref{iso}) gives (\ref{opl}.c).
\end{proof}
We will use the symbol $\,\pr\,$ for both projection operators
\begin{equation}\label{prj}
\mathrm{a)}\hskip6pt\pr:\et\to\mathrm{Ker}\,(\cu-\mathrm{Id})\hs,\hskip14pt
\mathrm{b)}\hskip6pt\pr:\es\to\mathrm{Ker}\,\dz
\end{equation}
arising, under the assumption (\ref{knd}), from the decompositions in 
(\ref{opl}.a\hs-\hh b).
\begin{rem}\label{dpepd}For $\,\ly\in\et\,$ we have 
$\,\dj\hh(\pr\hh\ly)=\pr\hh(\dj\hh\ly)$, the first (or, second) $\,\pr\,$ 
being as in (\ref{prj}.a) (or, (\ref{prj}.b)). In fact, using (\ref{opl}.a) to 
write $\,\ly=\ty+(\cu-\mathrm{Id})\hh\sy$, with 
$\,\ty\in\mathrm{Ker}\,(\cu-\mathrm{Id})\,$ and $\,\sy\in\et\nnh$, we obtain, 
from (\ref{dio}), $\,\dj\hh\ly=\dj\hh\ty+\dz(\dj\hh\sy)$, where 
$\,\dj\hh\ty=\pr\hh(\dj\hh\ly)\in\mathrm{Ker}\,\dz\,$ by (\ref{krd}). Thus, 
$\,\pr\hh(\dj\hh\ly)=\dj\hh\ty=\dj\hh(\pr\hh\ly)$.
\end{rem}
\begin{rem}\label{kroez}For simple Lie algebras other than the ones 
corresponding to (\ref{gsu}) with $\,\n=\d+\k\ge3\,$ we have 
$\,\mathrm{Ker}\,(\cu-\mathrm{Id})=0$. (See 
\cite[Remark~5.4]{derdzinski-gal}.) Then, by (\ref{krd}) and (\ref{dhz}), 
$\,0\,$ is isolated in $\,\hz\hs^{-\nh1}(0)$. In other words, according to 
(\ref{eud}.ii) and (\ref{inc}), for those Lie algebras, $\,\dj\,$ is an 
isolated point in both $\,\ew\,$ and $\,\ee$.
\end{rem}

\section{The underlying real Lie algebras}\label{ur}
\setcounter{equation}{0}
Denoting by $\,\mathfrak{g}_\bbR\s$ the underlying real Lie algebra of a given 
sem\-i\-sim\-ple complex Lie algebra $\,\mathfrak{g}\,$ with the 
Kil\-ling form $\,\by$, we see that, by (\ref{bvw}) and (\ref{rec}),
\begin{equation}\label{tre}
\mathrm{the\ Kil\-ling\ form\ of\ }\,\,\mathfrak{g}_\bbR\s\,\mathrm{\ 
equals\ }\,\,2\,\mathrm{Re}\,\by\hh.
\end{equation}
Thus, $\,\mathfrak{g}_\bbR\s$ is sem\-i\-sim\-ple as well. The following lemma 
and remarks use
\begin{equation}\label{tsy}
\begin{array}{l}
\mathrm{the\ spaces\ }\,\et\hskip-3pt_\bbR\s\hh,\,\es\nh_\bbR\s\hh,
\,\ey\nnh_\bbR\s\hh,\mathrm{\ the\ linear\ en\-do\-mor\-phism\ }
\,\dz_\bbR\s:\es\nh_\bbR\s\to\es\nh_\bbR\s\hh,\\
\{\hskip2pt\cdot\hskip2pt\}_\bbR\s:\ey\nnh_\bbR\s\nh\times\ey\nnh_\bbR\s\nh
\to\et\hskip-3pt_\bbR\s\hh,\mathrm{\  and\ the\ projection\ 
}\,\pr_\bbR\s:\es\nh_\bbR\s\nh\to\mathrm{Ker}\,\dz_\bbR\s\hh,
\end{array}
\end{equation}
the $\,\mathfrak{g}_\bbR\s$-counterparts of 
$\,\et\nnh,\es,\ey,\dz,\{\hskip2pt\cdot\hskip2pt\},\pr\,$ in (\ref{sbs}), 
(\ref{ths}.ii), (\ref{ndn}), (\ref{prj}.b) for $\,\mathfrak{g}$.
\begin{lem}\label{undrl}{\smallit 
For the real space $\,\et\hh'$ of all Her\-mit\-i\-an ses\-qui\-lin\-e\-ar forms $\,\mathfrak{g}\times\mathfrak{g}\to\bbC$,
\begin{enumerate}
  \def\theenumi{{\rm\roman{enumi}}}
\item[{\rm(a)}] $\et\hskip-3pt_\bbR\s\,=\,\hs[\hs\mathrm{Re}\,\et\hh]\,
\oplus\,\hs[\hs\mathrm{Re}\,\et\hh'\hh]$, \skip1.5ptwhere the summands are\/ 
$\hs\bbR$-iso\-mor\-phic images of\/ $\,\et$ and\/ $\,\et\hh'\nh$ under the 
operator\/ $\,\sy\mapsto\hs\mathrm{Re}\,\sy$,
\item[{\rm(b)}] $\mathrm{Re}\,\et\hh'\nh\,\subset\,\mathrm{Ker}\,\cu_\bbR\s$,
\item[{\rm(c)}] $\cu_\bbR\s(\mathrm{Re}\,\sy)\,=\,\mathrm{Re}\,(\cu\hs\sy)\,$ 
for all\/ $\,\sy\in\et\nnh$, so that\/ $\,\cu_\bbR\s$ leaves the summand\/ 
$\,\mathrm{Re}\,\et\,$ invariant and\/ 
$\,\,\cu_\bbR\s:\mathrm{Re}\,\et\to\mathrm{Re}\,\et\,$ corresponds, under the 
isomorphic identification\/ $\,\et\to\mathrm{Re}\,\et\,$ in\/ {\rm(a)}, 
to\/ $\,\cu:\et\to\et\nnh$,
\item[{\rm(d)}] $\mathrm{Ker}\,(\cu_\bbR\s\nnh-\mathrm{Id})\,$ is contained in 
the summand\/ $\,\mathrm{Re}\,\et\nnh$, cf.\ {\rm(a)}, and arises as the image 
of\/ $\,\mathrm{Ker}\,(\cu-\mathrm{Id})\subset\,\et\,$ under the isomorphism\/ 
$\,\et\nh\ni\sy\mapsto\hs\mathrm{Re}\,\sy\in\hs\mathrm{Re}\,\et\nnh$.
\end{enumerate}
}
\end{lem}
\begin{proof}If a $\,\bbC$-lin\-e\-ar en\-do\-mor\-phism $\,\Sigma\,$ of 
$\,\mathfrak{g}\,$ is related to $\,\sy\in\et\,$ via (\ref{sgv}), taking the 
real parts of both sides of (\ref{sgv}), with $\,\langle\,,\rangle=\by$, we 
see that, by (\ref{tre}), the same relation holds, in $\,\mathfrak{g}_\bbR\s$, 
between $\,\Sigma\,$ and $\,2\,\mathrm{Re}\,\sy\in\mathrm{Re}\,\et\nnh$. Now 
(\ref{oms}), its real-part version, and (\ref{rec}) yield (c). Next, given 
$\,\sy\in\et\hh'\nnh$, let 
$\,\Sigma:\mathfrak{g}_\bbR\s\to\mathfrak{g}_\bbR\s$ be associated with 
$\,\mathrm{Re}\,\sy\,$ as in (\ref{sgv}) for $\,\mathfrak{g}_\bbR\s$ and 
$\,2\,\mathrm{Re}\,\by\,$ instead of  $\,\mathfrak{g}\,$ and $\,\by$. The 
$\,\bbR$-lin\-e\-ar en\-do\-mor\-phism 
$\,\mathfrak{g}_\bbR\s\to\mathfrak{g}_\bbR\s$ of multiplication by 
$\,\mathrm{i}\,$ is $\,(\mathrm{Re}\,\by)$-self-ad\-joint, so that 
$\,0=\mathrm{Re}\,[\hh\sy(\mathrm{i}v,w)+\sy(v,\mathrm{i}w)]
=2\,\mathrm{Re}\,\by(\Sigma\mathrm{i}v+\mathrm{i}\Sigma v,w)\,$ for all 
$\,v,w\in\mathfrak{g}$, and $\,\Sigma\,$ is $\,\bbC$-anti\-lin\-e\-ar. 
As (\ref{oms}) and (\ref{rec}) imply that 
$\,[\hs\cu_\bbR\s(\mathrm{Re}\,\sy)](v,w)$ equals $\hs4\,\mathrm{Re}\hskip2.7pt
\mathrm{tr}^\bbC[(\mathrm{Ad}\,v)(\mathrm{Ad}\,w)\hh\Sigma]$, the final 
clause of Remark~\ref{trzro} gives $\,\cu_\bbR\s(\mathrm{Re}\,\sy)=0$, proving 
(b). Finally, since any $\,\sy\in\et\,$ or $\,\sy\in\et\hh'$ is uniquely 
determined by its real part, (a) follows, while (d) is a trivial 
consequence of (a), (b) and (c).
\end{proof}
\begin{rem}\label{nrese}One has 
$\,2\hh\nabla\nh(\mathrm{Re}\,\sy\nnh)=\nabla\nh\sy\,$ if 
$\,\nabla:\mathfrak{g}\times\mathfrak{g}\to\mathfrak{g}\,$ is a 
(com\-plex-bi\-lin\-e\-ar) connection in a sem\-i\-sim\-ple complex Lie 
algebra $\,\mathfrak{g}\,$ and $\,\sy\in\et\nnh$, with $\,\nabla\nh\sy\,$ as 
in (\ref{sns}), where $\,\nabla\nh(\mathrm{Re}\,\sy\nnh)\,$ is analogously 
defined for $\,\mathfrak{g}_\bbR\s$, so that 
$\,\mathrm{Re}\,\sy\in\et\hskip-3pt_\bbR\s$ (notation of Lemma~\ref{undrl}) 
and $\,\nabla\,$ is treated as a (real-bi\-lin\-e\-ar) connection in 
$\,\mathfrak{g}_\bbR\s$.

To see this, apply $\,2\,\mathrm{Re}\,$ to both sides in (\ref{sns}) 
and use (\ref{tre}) with $\,\by=\langle\,,\rangle$.
\end{rem}
\begin{rem}\label{rehol}Given $\,\mathfrak{g}\,$ and 
$\,\mathfrak{g}_\bbR\s$ as above, let us use the notation of (\ref{tsy}).
\begin{enumerate}
  \def\theenumi{{\rm\roman{enumi}}}
\item[{\rm(i)}] The real part of a (hol\-o\-mor\-phic) metric $\,\gy\,$ in 
$\,\mathfrak{g}\,$ is a (pseu\-\hbox{do\hskip1pt-}\hskip0ptRiem\-ann\-i\-an) 
metric in $\,\mathfrak{g}_\bbR\s$, and both metrics have the same 
Le\-\hbox{vi\hh-}\hskip0ptCi\-vi\-ta connection.
\item[{\rm(ii)}] If $\,\rho\,$ is the Ric\-ci tensor of the Le\-\hbox{vi\hh-}\hskip0ptCi\-vi\-ta 
connection of a metric in $\,\mathfrak{g}$, then $\,\nabla\,$ viewed as a 
connection in $\,\mathfrak{g}_\bbR\s$ has Ric\-ci tensor 
$\,2\,\mathrm{Re}\,\rho$.
\item[{\rm(iii)}] The real part of every Ein\-stein metric in 
$\,\mathfrak{g}\,$ is an Ein\-stein metric in $\,\mathfrak{g}_\bbR\s$.
\item[{\rm(iv)}] Every Ein\-stein connection in $\,\mathfrak{g}\,$ is also 
an Ein\-stein connection in $\,\mathfrak{g}_\bbR\s$.
\item[{\rm(v)}] $\{\nabla\hskip-2.5pt\cdot\hskip-2.5pt\td\nabla\}_\bbR\s\hs
=\,2\,\mathrm{Re}\,\{\nabla\hskip-2.5pt\cdot\hskip-2.5pt\td\nabla\}\,$ 
whenever $\,\nabla\nnh,\td\nabla\in\ey\hs\subset\ey\nnh_\bbR\s$.
\item[{\rm(vi)}] $\mathrm{Ker}\,\dz_\bbR\s=\hs\mathrm{Ker}\,\dz$, so that any 
element 
$\,\sj:\mathfrak{g}_\bbR\s\times\mathfrak{g}_\bbR\s\to\mathfrak{g}_\bbR\s\,$ 
of $\,\mathrm{Ker}\,\dz_\bbR\s$, originally assumed real-bi\-lin\-e\-ar, 
must actually be com\-plex-bi\-lin\-e\-ar.
\item[{\rm(vii)}] $\,\dz:\es\to\es\,$ is the restriction of 
$\,\dz_\bbR\s:\es\nh_\bbR\s\to\es\nh_\bbR\s$ to $\,\es\subset\es\nh_\bbR\s$.
\item[{\rm(viii)}] If $\,\sj\in\es\nh_\bbR\s$ and $\,\dz_\bbR\s\sj\in\es$, 
then $\,\sj\in\es$.
\end{enumerate}
In fact, (i) is obvious from Remarks~\ref{lccon} and~\ref{nrese} (the latter 
for $\,\sy=\gy$); (ii) from Lemma~\ref{unimo}(d) combined with (\ref{rec}) 
(where the former can be applied in view of Remark~\ref{ssuni} and 
Lemma~\ref{unimo}(b)); (iii) from (i) -- (ii); (iv) from (i) and (iii); (v) 
from (\ref{ndn}) and (\ref{rec}).
Next, elements of $\,\mathrm{Ker}\,\dz_\bbR\s$ coincide, by (\ref{krd}) 
for $\,\mathfrak{g}_\bbR\s$, with the $\,\dj\hh$-grad\-i\-ents $\,\dj\hh\ty$, 
in $\,\mathfrak{g}_\bbR\s$, of all 
$\,\ty\in\mathrm{Ker}\,(\cu_\bbR\s\nnh-\mathrm{Id})$. According to 
Lemma~\ref{undrl}(d), such $\,\ty\,$ are precisely the same as 
all $\,\mathrm{Re}\,\sy\,$ for $\,\sy\in\mathrm{Ker}\,(\cu-\mathrm{Id})$. 
As Remark~\ref{nrese} with $\,\nabla\nh=\dj\,$ gives 
$\,2\hh\dj\hh\ty=2\hh\dj(\mathrm{Re}\,\sy\nnh)=\dj\hh\sy$, (vi) follows from 
(\ref{krd}). Furthermore, (vii) is an immediate consequence of (\ref{ths}.ii), 
(v), (\ref{tre}) and Remark~\ref{nrese}.

Finally, let $\,\sj\in\es\nh_\bbR\s$ and $\,\dz_\bbR\s\sj\in\es$. By 
(\ref{opl}.b), $\,\dz_\bbR\s\sj-\dz\td\sj\in\mathrm{Ker}\,\dz\,$ for 
some $\,\td\sj\in\es$. From (vi) and (vii) one thus has 
$\,\dz_\bbR\s(\sj-\td\sj)\in\mathrm{Ker}\,\dz_\bbR\s$. Applying 
(\ref{opl}.b\hs-\hs c) to $\,\mathfrak{g}_\bbR\s$ rather than 
$\,\mathfrak{g}$, we now conclude that $\,\dz_\bbR\s(\sj-\td\sj)=0$, and so 
$\,\sj-\td\sj\in\mathrm{Ker}\,\dz_\bbR\s=\hs\mathrm{Ker}\,\dz\subset\es\,$ 
(cf.\ (vi)). Hence $\,\sj\in\es$. 
\end{rem}

\section{Complexifications}\label{cx}
\setcounter{equation}{0}
All vector spaces (and Lie algebras) are assumed here to be 
fi\-\hbox{nite\hh-}\hskip0ptdi\-men\-sion\-al.

By a {\smallit real form\/} of a complex vector space (or, a complex Lie 
algebra) $\,\mathfrak{g}\nh^\bbC$ we mean any real sub\-space (or, real Lie 
sub\-al\-ge\-bra) $\,\mathfrak{g}\subset\mathfrak{g}\nh^\bbC$ such that, as a 
real vector space, $\,\mathfrak{g}\nh^\bbC\nh=\mathfrak{g}\oplus\mathrm{i}\mathfrak{g}$. 
This is, in an obvious sense, equivalent to requiring that 
$\,\mathfrak{g}\nh^\bbC$ be the vec\-tor-space (or, Lie\hs-al\-ge\-bra) 
com\-plexifica\-tion of $\,\mathfrak{g}$. Whenever $\,\d+\k=\n$,
\begin{equation}\label{rfr}
\mathfrak{sl}\hh(\n,\bbR),\hskip6pt\mathfrak{su}\hh(\d,\k)\,
\,\,\mathrm{\ and\ }\,\,\mathfrak{sl}\hh(\n/2,\bbH)\,
\,\mathrm{\ are\ real\ forms\ of\ }\,\,\mathfrak{sl}\hh(\n,\bbC)\hh,
\end{equation}
the last one for even $\,\n\,$ only. Here, by definition, 
$\,\mathfrak{sl}\hh(k,\bbH)\,$ consists of those $\,\bbH$-lin\-e\-ar 
en\-do\-mor\-phisms of $\,\bbH^k$ which are also $\,\bbC$-trace\-less in the 
sense of Remark~\ref{quate}. Thus, with 
$\,\mathfrak{g}=\mathfrak{sl}\hh(k,\bbH)$, and with 
$\,\mathfrak{g}\nh^\bbC\approx\,\mathfrak{sl}\hh(2k,\bbC)\,$ denoting the Lie 
algebra of all trace\-less $\,\bbC$-lin\-e\-ar en\-do\-mor\-phisms of 
$\,\bbH^k\nnh$, the $\,(\pm1)$-eigen\-space decomposition of 
$\,\mathfrak{g}\nh^\bbC$ under the involution $\,a\mapsto-\jj a\jj$, for 
$\,\jj\,$ as in Remark~\ref{quate}, reads 
$\,\mathfrak{g}\nh^\bbC\nh=\mathfrak{g}\oplus\mathrm{i}\mathfrak{g}$. (The 
components $\,a,b\in\mathfrak{g}\,$ of any 
$\,a+\mathrm{i}b\in\mathfrak{g}\nh^\bbC\,$ are indeed $\,\bbC$-trace\-less, 
since so is $\,a+\mathrm{i}b$, while the $\,\bbC$-traces of $\,a\,$ and 
$\,b\,$ are real by Remark~\ref{quate}.)

The four Lie algebras appearing in (\ref{rfr}) are subsets of 
$\,\mathfrak{gl}\hh(\n,\bbC)$. Any finite product $\,w\,$ of their elements is 
therefore a complex $\,\n\times\n\,$ matrix, and we denote by 
$\,\mathrm{tr}\hskip2ptw\,$ its matrix trace. Hence, for 
$\,a,b,w\in\mathfrak{gl}\hh(\n,\bbC)$, one has the inner product $\,(a,b)\,$ 
as in (\ref{acb}), and the trace\-less part $\,(w)_0\w$, with
\begin{equation}\label{ntt}
\mathrm{i)}\hskip8pt(a,b)\,=\,\mathrm{tr}\hskip2ptab\hs,\hskip28pt
\mathrm{ii)}\hskip8pt(w)_0\w=w
-(\mathrm{tr}\hskip2ptw)\hs\mathrm{Id}\hh/\hn\n\hs.
\end{equation}
Given a real form $\,\mathfrak{g}\,$ of an $\,m$\hh-\hskip0ptdi\-men\-sion\-al 
complex vector space $\,\mathfrak{g}\nh^\bbC\nnh$,
\begin{equation}\label{bss}
\mathrm{any\ }\,\bbR\hyp\mathrm{basis\ }\,e_1\w,\dots,e_\m\w\mathrm{\ of\ 
}\,\mathfrak{g}\,\mathrm{\ is\ at\ the\ same\ time\ a\ }\,\bbC\hyp\mathrm{basis\ 
of\ }\,\mathfrak{g}\nh^\bbC\nh,
\end{equation}
leading, as in the lines following (\ref{uni}), to the corresponding components of 
\begin{equation}\label{map}
\mathrm{any\ }\,\bbC\hyp\mathrm{(bi)\-lin\-e\-ar\ mapping\ from\ }
\mathfrak{g}\nh^\bbC\mathrm{\ valued\ in\ }\,\bbC\,\mathrm{\ or\ 
}\,\mathfrak{g}\nh^\bbC\nh.
\end{equation}
\begin{rem}\label{unext}
A mapping (\ref{map}) is the unique $\,\bbC$-(bi)\-lin\-e\-ar extension 
of an $\,\bbR$-val\-ued or a $\,\mathfrak{g}$-val\-ued mapping from 
$\,\mathfrak{g}\,$ if and only if the former has the same components as 
the latter relative to some/\nh any pair of bases as in (\ref{bss}). We will 
denote both mappings by the same symbol.
\end{rem}
\begin{ex}\label{trklf} For an $\,\bbR$-lin\-e\-ar en\-do\-mor\-phism 
$\,\Sigma\,$ of $\,\mathfrak{g}$, its $\,\bbC$-lin\-e\-ar extension $\,\Sigma$ 
to $\,\mathfrak{g}\nh^\bbC\nnh$, and a pair of bases of type (\ref{bss}),
\begin{equation}\label{trc}
\mathrm{the\ }\,\bbC\hyp\mathrm{trace\ of\ 
}\,\Sigma:\mathfrak{g}\nh^\bbC\nh\to\mathfrak{g}\nh^\bbC\mathrm{\ equals\ the\ 
}\,\bbR\hyp\mathrm{trace\ of\ }\,\Sigma:\mathfrak{g}\to\mathfrak{g}\hh,
\end{equation}
as both traces coincide with $\,\Sigma_k^k$, where 
$\,\Sigma e_j\w=\Sigma_j^ke_k\w$. In the case of Lie algebras,
\begin{equation}\label{kfc}
\begin{array}{l}
\mathrm{the\ Kil\-ling\ form\ of\ }\,\mathfrak{g}\nh^\bbC\mathrm{\ is\ 
}\by:\mathfrak{g}\nh^\bbC\nh\times\mathfrak{g}\nh^\bbC\to\bbC,\mathrm{\ 
the\ }\,\bbC\hyp\mathrm{bi\-lin\-e\-ar}\\
\mathrm{extension\ to\ }\,\mathfrak{g}\nh^\bbC\mathrm{\ of\ the\ Kil\-ling\ 
form\ }\,\by:\mathfrak{g}\times\mathfrak{g}\to\bbR
\,\mathrm{\ of\ }\,\mathfrak{g}\hh,
\end{array}
\end{equation}
since, due to (\ref{bvw}), (\ref{acb}), and (\ref{trc}) applied to 
$\,\Sigma=(\mathrm{Ad}\,v)\hs\mathrm{Ad}\,w$, the latter Kil\-ling form arises 
as the restriction of the former from $\,\mathfrak{g}\nh^\bbC$ to 
$\,\mathfrak{g}$. In addition, obviously,
\begin{equation}\label{lbc}
\begin{array}{l}
\mathrm{the\ Lie\ bracket\ }
\,[\hskip2pt,\hskip.6pt]:\mathfrak{g}\nh^\bbC\nh\times\mathfrak{g}\nh^\bbC
\to\mathfrak{g}\nh^\bbC\mathrm{\ of\ }\,\mathfrak{g}\nh^\bbC\mathrm{\  is\ the\ 
unique}\\
\bbC\hyp\mathrm{bi\-lin\-e\-ar\ extension\ of\ the\ Lie\ bracket\ 
}\,[\hskip2pt,\hskip.6pt]:\mathfrak{g}\times\mathfrak{g}\to\mathfrak{g}\hh.
\end{array}
\end{equation}
We could also have derived (\ref{kfc}) from (\ref{lbc}) via 
Remark~\ref{unext}, using the component formula 
$\,\by_{jk}\w=C_{\nh pj}\w{}^qC_{\!qk}\w{}^p$ in (\ref{fdd}.a) valid for both 
$\,\mathfrak{g}\,$ and $\,\mathfrak{g}\nh^\bbC\nnh$.
\end{ex}
Let $\,\mathfrak{g}\,$ be a real form of a complex Lie algebra 
$\,\mathfrak{g}\nh^\bbC\nnh$. We denote by $\,\et\hs^\bbC$ the second complex 
symmetric power of the complex dual space of $\,\mathfrak{g}\nh^\bbC\nnh$, that 
is, the analog for $\,\mathfrak{g}\nh^\bbC$ of the space $\,\et\,$ associated 
with $\,\mathfrak{g}\,$ as in (\ref{sbs}). Our notation is consistent with the 
fact that $\,\et\hs^\bbC$ may be treated as the complexification of 
$\,\et\nh$. In other words, $\,\et\,$ can be naturally identified with a real 
form of $\,\et\hs^\bbC\nnh$, so that
\begin{equation}\label{tce}
\et\subset\et\hs^\bbC\,\,\mathrm{\ and\ }\,\,\et\hs^\bbC
=\,\hs\et\oplus\,i\hs\et.
\end{equation}
Specifically, any $\,\ty\in\et\nh$, which is a symmetric bi\-lin\-e\-ar 
form $\,\tau:\mathfrak{g}\times\mathfrak{g}\to\bbR$, is identified with its 
$\,\bbC$-bi\-lin\-e\-ar extension 
$\,\ty:\mathfrak{g}\nh^\bbC\nh\times\mathfrak{g}\nh^\bbC\to\bbC$. Thus, 
$\,\et\subset\et\hs^\bbC$ consists of all symmetric $\,\bbC$-bi\-lin\-e\-ar 
forms $\,\ty:\mathfrak{g}\nh^\bbC\nh\times\mathfrak{g}\nh^\bbC\to\bbC\,$ for which 
$\,\ty(v,w)\,$ is real whenever $\,v,w\in\mathfrak{g}\,$ (or, equivalently, 
$\,\ty_{\nh jk}\w\in\bbR\,$ in a basis of type (\ref{bss})).
\begin{lem}\label{omgtc}{\smallit 
For a real Lie algebra\/ $\,\mathfrak{g}\,$ and its complexification\/ 
$\,\mathfrak{g}\nh^\bbC\nh=\mathfrak{g}\oplus\mathrm{i}\mathfrak{g}$,
\begin{enumerate}
  \def\theenumi{{\rm\roman{enumi}}}
\item[{\rm(i)}] $\mathfrak{g}\,$ is sem\-i\-sim\-ple if and only if so is\/ 
$\,\mathfrak{g}\nh^\bbC\nnh$.
\end{enumerate}
Let\/ $\,\mathfrak{g}\,$ now be sem\-i\-sim\-ple. Then
\begin{enumerate}
  \def\theenumi{{\rm\roman{enumi}}}
\item[{\rm(ii)}] the analog for\/ $\,\mathfrak{g}\nh^\bbC$ of\/ 
$\,\cu:\et\to\et\nh$, given in\/ $\,\mathfrak{g}\,$ by\/ {\rm(\ref{oms})}, is 
the unique\/ $\,\bbC$-lin\-e\-ar extension\/ 
$\,\cu:\et\hs^\bbC\nh\to\et\hs^\bbC$ of\/ $\,\cu\,$ to\/ $\,\et\hs^\bbC\nh$, 
cf.\/ {\rm(\ref{tce})},
\item[{\rm(iii)}] the operations\/ 
$\,\{\hskip2pt\cdot\hskip2pt\},\,\ca:\et\hs^\bbC\nh\times\et\hs^\bbC\nh
\to\et\hs^\bbC\nh$ defined for\/ $\,\mathfrak{g}\nh^\bbC$ as in\/ 
{\rm(\ref{ndn})} and the lines following Lemma\/~{\rm\ref{psidl}} are the\/ 
$\,\bbC$-lin\-e\-ar extensions of their counterparts in\/ $\,\mathfrak{g}$.
\end{enumerate}
}
\end{lem}
\begin{proof}By (\ref{kfc}) and Remark~\ref{unext}, both Kil\-ling forms have 
the same components $\,\by_{jk}\w$ in bases of type (\ref{bss}), which yields 
(i). Assertions (ii) -- (iii) are in turn immediate from Remark~\ref{unext}, 
applied, instead of a basis $\,e_1\w,\dots,e_\m\w$ of $\,\mathfrak{g}$, to 
the basis of $\,\et=[\mathfrak{g}\nh^*]^{\odot2}$ formed by suitable symmetric 
products of the basis of $\,\mathfrak{g}\nh^*$ dual to $\,e_1\w,\dots,e_\m\w$. 
The components in question are given by the equality 
$\,[\hs\cu\hs\sy]_{jk}\w=Z_{\nnh jk}^{\hh rs}\sy\nh_{rs}\s$, two lines after 
(\ref{osj}), along with (\ref{ghg}) and 
$\,2\hh(\sy\ca\ty)_{\nh jk}\w=\sy_{\!js}\w\ty_{\nh k}^s
+\ty_{\nh js}\w\sy_{\!k}^s$.
\end{proof}

\section{Traces and Lie subalgebras}\label{tl}\checked
\setcounter{equation}{0}
Let $\,\bbF=\bbR\,$ or $\,\bbF=\bbC$. For the linear en\-do\-mor\-phism 
$\,v\mapsto uvw\,$ of $\,\mathfrak{gl}\hh(\n,\bbF)$, 
\begin{equation}\label{tra}
\mathrm{the\ }\,\mathfrak{gl}\hh(\n,\bbF)\hyp\mathrm{trace\ of\ }
\,\,v\mapsto uvw\,\,\mathrm{\ equals\ }
\,(\mathrm{tr}\hskip2ptu)\hs\mathrm{tr}\hskip2ptw\hh,
\end{equation}
with any fixed $\,u,w\in\mathfrak{gl}\hh(\n,\bbF)$. In fact, 
$\,\,(\mathrm{tr}\hskip2ptu)\hs\mathrm{tr}\hskip2ptw\,$ is the sum of the 
diagonal terms of the coefficient matrix $\,u_p^sw_q^r$ in the component 
expression $\,(uvw)_q^s=u_p^sv_r^pw_q^r$.
\begin{rem}\label{trtpr}Suppose that $\,\ev\,$ is a vector space with 
$\,\dimf\nh\ev<\infty$.
\begin{enumerate}
  \def\theenumi{{\rm\roman{enumi}}}
\item[{\rm(i)}] If $\,\aj(\ev)\subset\td\ev\,$ for a linear en\-do\-mor\-phism 
$\,\aj:\ev\to\ev\,$ and a sub\-space $\,\td\ev\subset\ev\nh$, then the 
$\,\ev$-trace of $\,\aj\,$ equals the $\,\td\ev$-trace of its restriction to 
$\,\td\ev\nh$.
\item[{\rm(ii)}] Given a linear functional $\,\alpha\in\ev^*\nnh$, and a 
vector $\,w\in\ev\nh$, one clearly has 
$\,\mathrm{tr}\hskip2pt(\alpha\otimes w)=\alpha(w)$, where 
$\,\alpha\otimes w\,$ acts on $\,v\in\ev\,$ by $\,v\mapsto\alpha(v)\hh w$.
\end{enumerate}
\end{rem}
\begin{rem}\label{trsln}In the next two sections traces of linear 
en\-do\-mor\-phisms $\,\aj\,$ of $\,\mathfrak{sl}\hh(\n,\bbF)$, where 
$\,\bbF=\bbR\,$ or $\,\bbF=\bbC$, will be evaluated as follows:
\begin{enumerate}
  \def\theenumi{{\rm\alph{enumi}}}
\item[{\rm(a)}] write $\,\aj\,$ as an en\-do\-mor\-phism of 
$\,\mathfrak{gl}\hh(\n,\bbF)$, valued in $\,\mathfrak{sl}\hh(\n,\bbF)$, 
\item[{\rm(b)}] find the $\,\mathfrak{gl}\hh(\n,\bbF)$-trace of the latter, 
using either (\ref{tra}), or Remark~\ref{trtpr}(ii),
\item[{\rm(c)}] note that, by Remark~\ref{trtpr}(i), this is also the 
$\,\mathfrak{sl}\hh(\n,\bbF)$-trace of $\,\aj$.
\end{enumerate}
\end{rem}
\begin{lem}\label{kfsln}{\smallit 
If\/ $\,\mathfrak{g}\,$ is one of the four Lie algebras in\/ {\rm(\ref{rfr})}, 
$\,a,b\in\mathfrak{g}\hh$, while\/ $\,\langle\,,\rangle\,$ and\/ 
$\,(\hskip2pt,\hskip.4pt)\,$ are the Kil\-ling form and the inner product, 
with\/ {\rm(\ref{bvw})} and\/ {\rm(\ref{ntt})}, then\/ 
$\,\hs\langle a,b\rangle=2n\hh(a,b)$.
}
\end{lem}
\begin{proof}Let $\,a,b\in\mathfrak{gl}\hh(\n,\bbF)$, for $\,\bbF=\bbR\,$ or 
$\,\bbF=\bbC$. As $\,(\mathrm{Ad}\,a)\hs\mathrm{Ad}\,b\,$ sends any 
$\,v\in\mathfrak{gl}\hh(\n,\bbF)\,$ to $\,v\mapsto abv-avb-bva+vba$, its 
$\,\mathfrak{gl}\hh(\n,\bbF)$-trace equals, from (\ref{tra}), 
$\,2n\hh(a,b)-2(\mathrm{tr}\hskip2pta)\hs\mathrm{tr}\hskip2ptb$. By 
Remark~\ref{trtpr}(i), if $\,a,b\in\mathfrak{sl}\hh(\n,\bbF)$, this is also 
the $\,\mathfrak{sl}\hh(\n,\bbF)$-trace of 
$\,(\mathrm{Ad}\,a)\hs\mathrm{Ad}\,b$, which proves our claim in the case 
where $\,\mathfrak{g}=\mathfrak{sl}\hh(\n,\bbF)$. 

Our assertion for $\,\mathfrak{g}=\mathfrak{su}\hh(\d,\k)\,$ and 
$\,\mathfrak{g}=\mathfrak{sl}\hh(\n/2,\bbH)\,$ now follows from (\ref{kfc}) 
applied to either choice of $\,\mathfrak{g}\,$ and to 
$\,\mathfrak{g}\nh^\bbC\nh=\mathfrak{sl}\hh(\n,\bbC)$, cf.\ (\ref{rfr}).
\end{proof}

\section{The special linear and pseudo\hs-unitary Lie algebras}\label{sl}
\setcounter{equation}{0}\checked
In this and the following sections, $\,(\mathfrak{g},\bbF\nnh,\ve)\,$ will 
always be one of the triples
\begin{equation}\label{lie}
\begin{array}{l}
(\mathfrak{sl}\hh(\n,\bbR),\bbR,1)\hs,\hskip3pt
(\mathfrak{sl}\hh(\n,\bbC),\bbC,1)\hs,\hskip3pt
(\mathfrak{sl}\hh(\n,\bbC),\bbC,\mathrm{i})\hs,\hskip3pt\mathrm{with\ 
}\,\n\ge2\hh,\\
(\mathfrak{su}\hh(\d,\k),\bbR,\mathrm{i})\,\,\mathrm{\ and,\ for\ even\ 
}\,\n\ge2\,\mathrm{\ only,}\,\,\,(\mathfrak{sl}\hh(\n/2,\bbH),\bbR,1)\hh,
\end{array}
\end{equation}
formed by a simple Lie algebra $\,\,\mathfrak{g}\,\,$ over $\,\bbF=\bbR\,$ or 
$\,\bbF=\bbC\,$ (see (\ref{rfr})), the scalar field $\,\bbF\,$ itself, and a 
fixed scalar $\,\ve\in\bbF$, equal to $\,1\,$ or $\,\mathrm{i}$. Here 
$\,\k,\d\,$ are fixed integers with $\,\d\ge\k\ge0\,$ and $\,\d+\k=\n$. As 
before, $\,\by=\langle\,,\rangle\,$ denotes the $\,\bbF$-bi\-lin\-e\-ar 
Kil\-ling form of $\,\mathfrak{g}$, defined by (\ref{bvw}), so that 
$\,\by\in\et\,$ for the space $\,\et\,$ in (\ref{sbs}).

With any $\,a,b\in\mathfrak{g}\,$ we associate elements 
$\,\tya,\,\cyab,\,\myab$ of $\,\et\nnh$, related as in (\ref{sgv}) to the 
$\,\bbF$-lin\-e\-ar endo\-morphisms of $\,\mathfrak{g}\,$ sending each 
$\,v\in\mathfrak{g}$, respectively, to
\begin{equation}\label{avp}
\ve\hh(av+va)_0\w\hs,\hskip9pt
\ve^2[(a,v)\hh b+(b,v)\hh a]/(2\n)\hskip7pt\mathrm{and}\hskip7pt
\ve^2(avb+bva)_0\w/2\hh,
\end{equation}
where $\,(\hskip1.3pt,\hskip-.1pt)\,$ and the trace\-less part 
$\,(\hskip2.3pt)_0\w$ are as in (\ref{ntt}).

Note that the values (\ref{avp}) all lie in $\,\mathfrak{g}\,$ due to our 
choice of $\,\ve$. In addition, the operators assigning these values to 
$\,v\,$ are easily seen to be self-ad\-joint for the inner product 
(\ref{ntt}.i). Their self-ad\-joint\-ness relative to the Kil\-ling form 
$\,\by\,$ of $\,\mathfrak{g}\,$ is now immediate from Lemma~\ref{kfsln}. Next, 
whenever $\,a\in\mathfrak{g}$, we set
\begin{equation}\label{cea}
c=\ve\hh(a\nh^2)_0\w\hs,\hskip8ptd=\ve^2(a\nh^3)_0\w\hs,\hskip8pt
\cya=\cy\nh_{a,a}\w\hs,\hskip8pt\mya=\my_{a,a}\w\hs,\hskip8pt
\xy=\ve^2(a,a)/\n\hh,
\end{equation}
with $\,(a,a)=\mathrm{tr}\hskip2pta\nh^2\nnh$. Thus, in the sense of 
(\ref{avp}) and the lines preceding it,
\begin{equation}\label{ava}
\begin{array}{l}
\tya\mathrm{\ corresponds\ to\ }\,v\mapsto\ve\hh(av+va)_0\w\hs,
\hskip17pt\by\,\hs\mathrm{\ to\ }\,\mathrm{Id}\hh,\\
\cya\mathrm{\ to\ }\,v\mapsto\ve^2(a,v)\hh a/\n,\mathrm{\ and\ }
\,\mya\mathrm{\ to\ }\,v\mapsto\ve^2(ava)_0\hh.
\end{array}
\end{equation}
In proofs of some equalities involving $\,\cyab$ and $\,\myab$ we may assume 
that $\,b=a$, as
\begin{equation}\label{bil}
\mathrm{the\ dependence\ of\ }\,\cyab\mathrm{\ and\ }\,\myab\mathrm{\ on\ 
}\,a,b\,\mathrm{\ is\ bi\-lin\-e\-ar\ and\ symmetric.}
\end{equation}
Also, if $\,\xy\,$ (in (\ref{cea})) and $\,\ve\hh(a,c)/\n\,$ stand for the 
corresponding multiples of $\,\mathrm{Id}$,
\begin{equation}\label{eaq}
\mathrm{i)}\hskip7pt\ve\hh a\nh^2\nnh=c+\ve^{-1}\xy\hh,\hskip14pt\mathrm{ii)}
\hskip7pt\ve^2a^3\nnh=d+\ve\hh(a,c)/\n\hh,\hskip7pt\mathrm{with}\hskip6pt
(a,c)=\mathrm{tr}\hskip2ptac\hh.
\end{equation}
Using (\ref{oms}), where we may set $\,v=w$, (\ref{osj}.i) and (\ref{ava}), 
the three steps of Remark~\ref{trsln}, formula (\ref{eaq}.i), and then 
Lemma~\ref{kfsln}, we obtain, for $\,a\in\mathfrak{g}$,
\begin{equation}\label{nom}
\n^2\cu\hh\cya=\tyc-2\hh\mya+2\hh\xy\by\hh,\hskip9pt
\cu\hh\mya=-2\hh\cya\hh,\hskip9pt\cu\hh\tya=\tya\hh,\hskip9pt
\cu\hh\by=2\hh\by\hh.
\end{equation}\checked
More precisely, we just proved (\ref{nom}) for the first three triples in 
(\ref{lie}). It follows, however, that (\ref{nom}) is satisfied by the other 
two triples as well. In fact, according to (\ref{rfr}), relations (\ref{kfc}), 
(\ref{tce}) and assertions (ii) -- (iii) in Lemma~\ref{omgtc} hold for 
$\,\mathfrak{g}=\mathfrak{su}\hh(\d,\k)\,$ or 
$\,\mathfrak{g}=\mathfrak{sl}\hh(\n/2,\bbH)$, and 
$\,\mathfrak{g}\nh^\bbC\nh=\mathfrak{sl}\hh(\n,\bbC)$. Lemma~\ref{omgtc}(ii), 
combined with Remark~\ref{bilxt} below, now shows that the $\,\cu$-im\-ages of 
$\,\,\tya,\cyab,\,\myab$ in $\,\mathfrak{su}\hh(\d,\k)\,$ or 
$\,\mathfrak{sl}\hh(\n/2,\bbH)\,$ coincide with their $\,\cu$-im\-ages in 
$\,\mathfrak{sl}\hh(\n,\bbC)$.

Note that, if $\,\x,y,\z\in\bbF\,$ and 
$\,\ly=\x\hh\tya+\n^2\y\hh\cya+\z\hh\mya$, (\ref{nom}) yields
\begin{equation}\label{imo}
(\hh\mathrm{Id}\,-\,\cu\hh)\hh\ly
=(\n^2\y+2\z)\hh\cya+(2\y+\z)\hh\mya
-2\hs\xy\y\hh\by-\y\hh\tyc\hh.
\end{equation}\checked
\begin{rem}\label{bilxt}By (\ref{kfc}), the elements 
$\,\,\tya,\cyab,\,\myab$ of $\,\et\hs^\bbC$ associated with the triple 
$\,(\mathfrak{sl}\hh(\n,\bbC),\bbC,\mathrm{i})$, or 
$\,(\mathfrak{sl}\hh(\n,\bbC),\bbC,1)$, are the unique 
$\,\bbC$-bi\-lin\-e\-ar extensions of $\,\tya,\,\cyab,\,\myab$ defined for 
$\,(\mathfrak{su}\hh(\d,\k),\bbR,\mathrm{i})\,$ or, respectively, for 
$\,(\mathfrak{sl}\hh(\n/2,\bbH),\bbR,1)$. Thus, $\,\,\tya,\cyab,\,\myab$ lie 
in the real form $\,\et\subset\et\hs^\bbC$ appearing in (\ref{tce}).
\end{rem}

\section{The curvature spectra of the Lie algebras in (\ref{lie})}\label{cs}
\setcounter{equation}{0} 
Whenever $\,(\mathfrak{g},\bbF\nnh,\ve)\,$ is one of the triples (\ref{lie}), 
$\,\n=\d+\k\ge2$, and $\,a,b\in\mathfrak{g}$,
\begin{enumerate}
  \def\theenumi{{\rm\roman{enumi}}}
\item[{\rm(i)}] $\n\hh\langle\tya,\tya\rangle\,
=\,2\hh\ve^2(\n^2\nnh-4)\hh(a,a)$, 
with $\,\langle\,,\rangle,\hs\tya,\hs(\hskip2pt,\hskip.4pt)\,$ as in 
(\ref{inp}.i), (\ref{ava}), (\ref{ntt}.i),
\item[{\rm(ii)}] for $\,\n\ge3\,$ the operator 
$\,\mathfrak{g}\ni a\mapsto\tya\in\et\,$ is injective,
\item[{\rm(iii)}] for $\,\n=2\,$ one has $\,\tya=0\,$ and 
$\,2\hh\cyab=\myab+\ve^2(a,b)\by/2$, cf.\ (\ref{avp}) and (\ref{bvw}).
\end{enumerate}
In fact, if $\,\mathfrak{g}=\mathfrak{sl}\hh(\n,\bbF)$, (i) follows from the 
three steps of Remark~\ref{trsln}, as $\hs\langle\tya,\tya\rangle\hh$ is 
the trace of the endo\-morphism of $\,\mathfrak{g}\,$ sending $\,v\,$ to 
$\,\ve^2(a(av+va)_0\w+(av+va)_0\w a)_0\w$, that is, to $\,\ve^2$ times 
$\,a\nh^2v+va\nh^2\nh+2ava-4\hh[(a,v)\hh a+(a\nh^2\nnh,v)\hs\mathrm{Id}]/\n$. 
For the other choices of $\,\mathfrak{g}$, (i) is obvious from (\ref{rfr}) 
and Remark~\ref{bilxt}, combined with (\ref{trc}) since, again, 
$\,\langle\tya,\tya\rangle\,$ equals the trace of the square of the 
endo\-morphism corresponding to $\,\tya$ in (\ref{ava}). Next, the operator in 
(ii) is injective: by (i), it pulls some symmetric bi\-lin\-e\-ar form in 
$\,\et\,$ back to the nondegenerate form $\,(\hskip2pt,\hskip.4pt)\,$ in 
$\,\mathfrak{g}$, cf.\ Lemma~\ref{kfsln}. Finally, proving (iii) amounts, for 
reasons of symmetry as in (\ref{bil}), to showing that $\,a\nh^2$ is a 
multiple of $\,\mathrm{Id}\,$ and $\,2\hh\cya=\mya+\ve^2(a,a)\by/2$. This is 
easily verified if $\,a\,$ is the diagonal matrix $\,\mathrm{diag}\hs(1,-1)$. 
The case of arbitrary $\,a\,$ is now immediate as matrices conjugate to 
multiples of $\,\mathrm{diag}\hs(1,-1)\,$ form a dense subset of 
$\,\mathfrak{sl}\hh(2,\bbF)$, consisting of all trace\-less matrices with two 
distinct eigenvalues.
\begin{lem}\label{noeig}{\smallit 
If\/ $\,(\mathfrak{g},\bbF\nnh,\ve)\,$ is one of the triples\/ 
{\rm(\ref{lie})} and\/ $\,\n=\d+\k\ge2$, then the en\-do\-mor\-phism\/ 
$\,\cu\,$ of\/ $\,\et=[\mathfrak{g}\nh^*]^{\odot2}\nnh$ defined by\/ 
{\rm(\ref{oms})} is di\-ag\-o\-nal\-izable and has the following ordered 
system\/ $\,\mathrm{spec}\hs[\cu]\,$ of eigenvalues and\/ 
$\,\mathrm{mult}\hs[\cu]\,$ of the corresponding multiplicities.
\begin{enumerate}
  \def\theenumi{{\rm\alph{enumi}}}
\item[{\rm(a)}] $\mathrm{spec}\hs[\cu]=(-1,2)\,$ and\/ 
$\,\mathrm{mult}\hs[\cu]=(5,1)$, when\/ $\,\n=2$,
\item[{\rm(b)}] $\mathrm{spec}\hs[\cu]=(-2/3,1,2)\,$ and\/ 
$\,\mathrm{mult}\hs[\cu]=(27,8,1)\,$ when\/ $\,\n=3$,
\item[{\rm(c)}] $\mathrm{spec}\hs[\cu]=(-2/\n,2/\n,1,2)\,$ and\/ 
$\,\mathrm{mult}\hs[\cu]=(d^-\nnh,d^+\nnh,\n^2\nh-1,1)\,$ 
for\/ $\,\n\ge4$, where\/ $\,d^\pm\nh=\n^2(\n\mp3)(\n\pm1)/4$.
\end{enumerate}
Furthermore, with\/ $\,\tya$ and the Kil\-ling form\/ $\,\by\,$ as in\/ 
{\rm(\ref{ava})},
\begin{enumerate}
  \def\theenumi{{\rm\alph{enumi}}}
\item[{\rm(d)}] for\/ $\,\n\ge3\hh$, the assignment\/ $\,a\mapsto\tya\,$ is a linear 
iso\-mor\-phism\/ $\,\mathfrak{g}\to\mathrm{Ker}\,(\cu-\mathrm{Id})$,
\item[{\rm(e)}] $\mathrm{Ker}\,(\cu-2\hskip1.4pt\mathrm{Id})\hh=\hh\bbF\nh\by$,
\item[{\rm(f)}] $\mathrm{Ker}\,(\cu-\mathrm{Id})\,$ is a nondegenerate 
sub\-space of\/ $\,\et\nnh$, that is, {\rm(\ref{knd})} holds.
\end{enumerate}
}
\end{lem}
\begin{proof}We define $\,\ipm\in\et\,$ by 
$\,\ipm=\n\hh\cyab\mp\myab-(\n\mp2)^{-1}\tye-[\n(\n\mp1)]^{-1}\ve^2(a,b)\by$, 
where $\,e=\ve\hh(ab+ba)_0\w/2$. (Notation of (\ref{avp}) -- (\ref{ava}); if 
$\,\n=2$, only $\,\iym$ is defined and, from (iii) above, 
$\,\iym=4\hh\cyab-\ve^2(a,b)\by$.) Now (\ref{bil}), (\ref{cea}) and 
(\ref{nom}) give $\,\n^2\cu\hh\cyab=\tye-2\hh\myab+2\hh\ve^2(a,b)\by/\n\,$ and 
$\,\cu\hh\myab=-2\hh\cyab$, so that
\begin{equation}\label{oie}
\begin{array}{l}
\cu\hh\ipm=\pm2\hh\ipm/\n\hh,\hskip16pt\cu\hh\tye=\tye\hh,\hskip16pt
\cu\hh\by=2\hh\by\hskip9pt\mathrm{for}\hskip6pt\n\ge3\hh,\\
\cu\hh\iym=-\hh\iym\hskip7pt\mathrm{and}\hskip7pt\cu\hh\by=2\hh\by
\hskip9pt\mathrm{for}\hskip6pt\n=2\hh.
\end{array}
\end{equation}
The elements $\,\cyab$ span $\,\et=[\mathfrak{g}\nh^*]^{\odot2}\nnh$, since 
$\,\cyab$ is a multiple of the symmetric product 
$\,(a,\,\cdot\,)\odot(b,\,\cdot\,)$. As 
$\,2\n\hh\cyab=\iyp+\iym+2\n(\n^2\nnh-4)^{-1}\tye
+2(\n^2\nnh-1)^{-1}\ve^2(a,b)\by$ (if $\,\n\ge3$) or 
$\,4\hh\cyab=\iym+\ve^2(a,b)\by\,$ (if $\,\n=2$), (\ref{oie}) yields
\[
\begin{array}{l}
\et\hs=\,\ev\nh_+\w\nnh+\,\ev\nh_-\w\nnh+\,\ev\,+\hs\bbF\nh\by\mathrm{\ \ for\ 
}\,\,\n\ge3\hh,\hskip16.7pt\et\hs=\,\ev\nh_-\w\nnh+\,\bbF\nh\by\mathrm{\ \ 
for\ }\,\,\n=2\hh,\\
\ev\nh_\pm\w\subset\mathrm{Ker}\,(\cu\mp2\hskip1.4pt\mathrm{Id}/\n)\hh,
\hskip12pt
\ev\subset\mathrm{Ker}\,(\cu-\mathrm{Id})\hh,\hskip12pt
\bbF\nh\by\subset\mathrm{Ker}\,(\cu-2\hskip1.4pt\mathrm{Id})\hh,
\end{array}
\]
where $\,\ev\nh_\pm$ (or, $\,\ev$) is the subspace of $\,\et\,$ spanned by all 
$\,\ipm$ (or, respectively, by all $\,\tya$). Since the eigenvalues in 
question are mutually distinct, $\,\cu\,$ is di\-ag\-o\-nal\-izable and the 
above inclusions are equalities. Now (d) -- (f) are obvious from (ii) above 
and the second part of (\ref{dia}), while 
$\,\mathrm{tr}\hskip2pt\cu=1-\n^2\nnh$, cf.\ (\ref{osj}.iii) and (\ref{rfr}), 
which uniquely determines the correct multiplicities (with $\,d^-\nh=0\,$ for 
$\,\n=3$).
\end{proof}
\begin{rem}\label{eigvz}Lemmas~\ref{noeig} and~\ref{undrl}(b)\hh-\hh(c) 
easily imply that, for the underlying real Lie algebra 
$\,\mathfrak{g}_\bbR\s$ of $\,\mathfrak{g}=\mathfrak{sl}\hh(\n,\bbC)$, 
the operator $\,\cu_\bbR\s$ (notation of (\ref{tsy})) has the eigen\-values 
listed in (a) -- (c) above, with twice the multiplicities of (a) -- (c), 
plus the eigen\-value $\,0\,$ with the multiplicity $\,(\n^2\nnh-1)^2\nnh$. By 
Lemma~\ref{undrl}(a)\hh-\hh(c), the eigen\-spaces for the eigen\-values 
$\,2\,$ and $\,1\,$ are $\,\bbR$-iso\-mor\-phic images of those in (e) and (f) 
above under the operator $\,\sy\mapsto\hs\mathrm{Re}\,\sy$.
\end{rem}
\begin{rem}\label{lniso}Let $\,(\mathfrak{g},\bbF\nnh,\ve)\,$ be one of the 
triples (\ref{lie}), with $\,\n=\d+\k\ge3$. Then, by (\ref{ttd}) and 
Lemma~\ref{noeig}(d), the composite operator
\[
a\mapsto\tya=\hs\ty\mapsto\dj\hh\ty\mathrm{\ \ is\ a\ linear\ iso\-mor\-phism\ 
\ }\,\mathfrak{g}\to\mathrm{Ker}\,(\cu-\mathrm{Id})\to\mathrm{Ker}\,\dz\hh.
\]
This is also true for $\,\mathfrak{g}=\mathfrak{sl}\hh(\n,\bbC)\,$ treated as 
a real Lie algebra $\,\mathfrak{g}_\bbR\s$, with $\,\dj\hh\tya$ denoting the 
$\,\dj$-grad\-i\-ent, in the complex Lie algebra $\,\mathfrak{g}$, of the 
(com\-plex-bi\-lin\-e\-ar) symmetric tensor $\,\tya$. In fact, 
$\,\mathrm{Ker}\,\dz_\bbR\s=\hs\mathrm{Ker}\,\dz\,$ by Remark~\ref{rehol}(vi).
\end{rem}
\begin{rem}\label{natrl}The surjective operator $\,\es\to\mathfrak{g}$, 
mentioned immediately after (\ref{azz}), is obtained as the composite of 
the projection $\,\pr:\es\to\mathrm{Ker}\,\dz\,$ in (\ref{prj}.b) followed by 
the inverse $\,\mathrm{Ker}\,\hs\dz\to\,\mathfrak{g}\,$ of the iso\-mor\-phism 
in Remark~\ref{lniso}.
\end{rem}

\section{Multiplication tables}\label{mt}\checked
\setcounter{equation}{0}
In this section $\,(\mathfrak{g},\bbF\nnh,\ve)\,$ is always one of 
the triples (\ref{lie}), with $\,\n=\d+\k\ge3$, and $\,a\,$ denotes an 
arbitrary element of $\,\mathfrak{g}$. We use the notation of (\ref{avp}) -- 
(\ref{ava}).

For $\,\ca\,$ and $\,\{\hskip2pt\cdot\hskip2pt\}\,$ defined in the lines 
following Lemma~\ref{psidl} and in (\ref{stw}),
\begin{equation}\label{ast}
\begin{array}{ll}
\tya\ca\tya=\tyc+2\hh(\xy\by+\mya-2\hh\cya)\hh,
\hskip19pt&\cya\ca\hs\cya=\xy\hh\cya\hh,\\
\mya\ca\mya=\xy\hh\tyc+\xy^2\nnh\by+\myc-\cyc\hh,
&\tya\ca\hs\cya=2\hh\cyac\hh,\\
\tya\ca\hh\mya=\xy\tya+2\hh\myac-2\hh\cyac\hh,
&\cya\ca\hh\mya=\cy_{a,d}\w\hh,\\
\by\hs\ca\tya=\tya\hh,\hskip18pt\by\hs\ca\hh\cya=\cya\hh,
&\by\hs\ca\hh\mya=\mya\hh,
\end{array}
\end{equation}\checked
as one easily verifies with the aid of (\ref{avp}) -- (\ref{ava}) and 
(\ref{eaq}). Also,
\begin{equation}\label{bdb}
\begin{array}{l}
\{\tya\hskip-3pt\cdot\nnh\tya\}=\xy\by+\mya-2\hh\cya\hh,\hskip29pt
\{\mya\hskip-3pt\cdot\nnh\mya\}=\xy\hh\mya
-\cyc\hh,\phantom{{j_{j_j}}}\\
2\hh\{\tya\hskip-3pt\cdot\nnh\mya\}=\xy\tya-4\hh\cyac\hh,
\hskip36pt\{\cya\hskip-3pt\cdot\nnh\cya\}=0\hh,\\
2n^2\{\tya\hskip-3pt\cdot\nnh\cya\}=\tyd
-\xy\tya-2\hh\myac+2\hs\ve^2(a,c)\by/\n\hh,\phantom{{j_{j_j}}}\\
n^2\{\cya\hskip-3pt\cdot\nnh\mya\}=\myad-\xy\tyc
-\xy^2\nnh\by-\myc+\ve^2(a,c)\hh\tya/(2\n)\hh.
\end{array}
\end{equation}\checked
If $\,\mathfrak{g}=\mathfrak{sl}\hh(\n,\bbF)$, (\ref{bdb}) is immediate from 
(\ref{avp}) -- (\ref{ava}) and (\ref{eaq}). Specifically, we use the three 
steps of Remark~\ref{trsln} and Lemma~\ref{kfsln}. Remark~\ref{bilxt} and 
Lemma~\ref{omgtc}(iii) now show that (\ref{bdb}) holds for the remaining 
triples $\,(\mathfrak{su}\hh(\d,\k),\bbR,\mathrm{i})$ and 
$\,(\mathfrak{sl}\hh(\n/2,\bbH),\bbR,1)\,$ in (\ref{lie}) as well.

Furthermore, (\ref{ava}) and (\ref{dsv}.a) trivially imply that, for all 
$\,v\in\mathfrak{g}$,
\begin{equation}\label{dta}
[\hh\dj\hh\tya]_v\w v=\ve\hs[\hh a,v^2\hh]\hh,\hskip6pt
\n\hh[\hh\dj\hh\cya]_v\w v=\ve^2[(a,v)\hh a,v\hh]\hh,\hskip6pt
[\hh\dj\hh\mya]_v\w v=\ve^2[\hh a,vav\hh]\hh.
\end{equation}\checked
From (\ref{dta}), (\ref{ava}), (\ref{dsv}.b) and 
(\ref{eaq}.i) one in turn obtains
\begin{equation}\label{tdt}
\begin{array}{l}
\tya(\dj\hh\tya)=\dj\hh\tyc\hh,\hskip8pt
\cya(\dj\hh\tya)=\cya(\dj\hh\cya)=\cya(\dj\hh\mya)
=0\hh,\\
\n\hh[\hh\tya(\dj\hh\cya)]_v\w v=\ve^2[(a,v)\hh c,v\hh]\hh,\hskip8pt
[\hh\tya(\dj\hh\mya)]_v\w v=\ve^2[\hh c,vav\hh]\hh,\\
{}[\hs\mya(\dj\hh\tya)+\dj\hh(\xy\tya)]_v\w v=\ve^2(cv^2a-av^2c)\hh,\\
\n\hh[\hs\mya(\dj\hh\cya)+\dj\hh(\xy\hh\cya)]_v\w v
=\ve^3(a,v)\hh(cva-avc)\hh,\\
{}[\hs\mya(\dj\hh\mya)+\dj\hh(\xy\hh\mya)]_v\w v=\ve^3(cvava-avavc)\hh.
\end{array}
\end{equation}\checked
\begin{lem}\label{omgpr}{\smallit 
For\/ $\,(\mathfrak{g},\bbF\nnh,\ve)\,$ as in\/ {\rm(\ref{lie})}, 
$\,\n=\d+\k\ge3\hh$, setting\/ $\,\chi_a\w=2\hh\mya-\n^2\cya+2\hs\xy\by
+(\n^2\nnh+4)(\n^2\nnh-4)^{-1}\tyc$ and\/ 
$\,\phi_a\w=2\n^2\cya-\n^2\mya-4\hh\xy\by-4\hh\n^2(\n^2\nnh-4)^{-1}\tyc\hs$, 
in the notation of\/ {\rm(\ref{oms})} and\/ {\rm(\ref{avp})} -- 
{\rm(\ref{ava})}, we have, if\/ $\,\x,y,\z\in\bbF\,$ and\/ 
$\,\ly=\x\hh\tya+\n^2\y\hh\cya+\z\hh\mya$,
\begin{enumerate}
  \def\theenumi{{\rm\roman{enumi}}}
\item[{\rm(i)}] $(\n^2\nnh-4)\hs\cya=\tyc+(\cu-\mathrm{Id})\hh\chi_a\w\,$ 
and\/ $\,(\n^2\nnh-4)\hs\mya=-2\hh\tyc+(\cu-\mathrm{Id})\hh\phi_a\w$,
\item[{\rm(ii)}] $\chi_a\w,\phi_a\w\in(\cu-\mathrm{Id})(\et)$,
\item[{\rm(iii)}] $\pr\hh\tya=\tya\hh,\hskip5pt(\n^2\nnh-4)\hs\pr\cya=\tyc\hh,
\hskip6pt(\n^2\nnh-4)\hs\pr\mya=-2\hh\tyc\hh,\hskip6pt\pr\by=0$,
\item[{\rm(iv)}] $(n^2\nnh-4)\hh(\pr\hh\ly-\tya)=(n^2\nnh-4)\hh(\x-1)\hh\tya
+(\n^2\y-2\z)\hh\tyc\hh$,
\end{enumerate}
$\pr\,$ being the di\-rect-sum projection in\/ {\rm(\ref{prj}.a)}, 
well-defined in view of Lemma~\/{\rm\ref{noeig}(f)}.
}
\end{lem}
\begin{proof}Assertion (i) is a trivial consequence of (\ref{nom}), while 
(iii) is immediate from (i) and (\ref{nom}), as $\,\cu:\et\to\et\,$ is 
self-ad\-joint (Lemma~\ref{sfadj}(ii)). Now (iv) is obvious, and (iii) gives 
$\,\pr\hh\chi_a\w=\pr\hh\phi_a\w=0$, proving (ii).
\end{proof}

\section{Some relevant algebraic sets in the nine-dimensional 
space}\label{sr}\checked
\setcounter{equation}{0}
Equations (\ref{sys}) -- (\ref{uvw}) discussed below, depending on an integer 
parameter $\,\n$, are quite important in our argument. Namely, as we will 
show in the proof of Lemma~\ref{ractz}, for the Lie algebras 
$\,\mathfrak{g}\,$ appearing in (\ref{lie}) with $\,\n=\d+\k\ge3$, any 
real-an\-a\-lyt\-ic curve of weak\-\hbox{ly\hh-}\hskip0ptEin\-stein 
connections emanating from the standard connection $\,\dj\,$ is naturally 
mapped into the solution set of (\ref{sys}) -- (\ref{uvw}) in $\,\bbF^9\nnh$, 
in such a way that the image of $\,\dj\,$ is the point (\ref{fhz}). The 
rationality conclusion of Lemma~\ref{ratnl} then implies that the latter 
mapping must be constant. This means that the real-an\-a\-lyt\-ic curves in 
question all lie in a specific family $\,\ec\,$ of Ein\-stein connections, 
described by formula (\ref{cdl}) in Section~\ref{mr}.
\begin{lem}\label{nnupl}{\smallit 
For\/ $\,\bbF=\bbR\,$ or $\,\bbF=\bbC\,$ and any integer\/ $\,n\ge3\hh$, 
the conditions
\begin{equation}\label{sys}
\begin{array}{rl}
\mathrm{i)}\hskip4pt&\f-(2\y+\x^2\nh+\xy\z^2)\hs\xy=0\hh,\\
\mathrm{ii)}\hskip4pt&\U+(\y+\x\p)\hh\h=0\hh,\\
\mathrm{iii)}\hskip4pt&(n^2\nnh-4)\hh(\x-1)+(\n^2\y-2\z)\hh\h=0\hh,\\
\mathrm{iv)}\hskip4pt&\V\nnh-\n^2\y-2\z=0\hh,\\
\mathrm{v)}\hskip4pt&\W\nnh-2\y-\z=0\hh,\\
\mathrm{vi)}\hskip4pt&\p+\x-2\hs\xy\x\z-\h\hs\xy\z^2\nh-\h\y=0\hh,\\
\mathrm{vii)}\hskip4pt&\q+\n^2(\K+2)\hh\y-\V\nnh+2\z=0\hh,\\
\mathrm{viii)}\hskip4pt&\r+(\K+2)\hh\z-\W\nnh+2\y=0\hh,
\end{array}
\end{equation}\checked
imposed on the nonuple\/ $\,(\xy,\f\nh,\h,\x,\y,\z,\p,\q,\r)\in\bbF^9\nnh$, 
where we have set
\begin{equation}\label{uvw}
\begin{array}{l}
\K=\f+\h\p-\xy\r\hh,\\
\U=2\hs\xy\x\z+\h\hs\xy\z^2\nh-\xy\x\r+\x\f\hh,\\
\V\nnh=\n^2\K\y+2\n^2(2\hh\xy\z+\h\x+\h^2\z)\hh\y\\
\phantom{\V}+\,\n^4\xy\y^2\nh-2\x^2\nh+2(\xy\z-\h\x)\hh\z\hh,\\
\W\nnh=\K\z+\x^2\nh+\xy\z^2\nh+\h^2\z^2\nh+2\hh\h\x\z\hh,
\end{array}
\end{equation}
constitute a system of eight polynomial equations in nine unknowns.
\begin{enumerate}
  \def\theenumi{{\rm\roman{enumi}}}
\item[{\rm(a)}] A solution of\/ {\rm(\ref{sys})}, \hskip1ptfor which\/ 
$\,(\K,\U,\V,\W)=(0,0,-2,1)$, is given by
\begin{equation}\label{fhz}
\xy=\h=\f=0\hh,\hskip9pt\x=-\hh\p=1\hh,\hskip9pt\q=-\hh\n^2\y
=\frac{4\n^2}{\n^2\nnh-4}\hs,
\hskip9pt\z=-\hh\r=\frac{\n^2\nnh+4}{\n^2\nnh-4}\hs.
\end{equation}
\item[{\rm(b)}] Solutions of\/ {\rm(\ref{sys})} lying near\/ {\rm(\ref{fhz})} 
\hskip1ptform the graph of an\/ $\,\bbF$-an\-a\-lyt\-ic curve
\begin{equation}\label{xto}
\xy\,\mapsto\,(\f\nh,\h,\x,\y,\z,\p,\q,\r)\in\bbF^8,\hskip8pt\mathrm{with}
\hskip5pt\mathrm{d}\h/\nh\mathrm{d}\hh\xy=4\hskip5pt\mathrm{at}\hskip5pt\xy=0\hh,
\end{equation}
where\/ $\,\xy\,$ ranges over a neighborhood of\/ $\,0\,$ in\/ $\,\bbF$.
\end{enumerate}
}
\end{lem}
\begin{proof}First, (a) is obvious. Secondly, one easily verifies that
\begin{equation}\label{jac}
\left[\begin{array}{cccccccc}
1&0&0&0&0&0&0&\hskip15pt0\cr
*&\hskip-5pt\n^2\nnh/\nh(4\nnh-\nnh\n^2)\hskip-5pt&0&0&\hskip-9pt0\hskip-9pt&0&\hskip.7pt0\hskip.7pt&\hskip15pt0\cr
*&*&\hskip-10pt\n^2\nnh\nh-\nh4\hskip-10pt&0&\hskip-9pt0\hskip-9pt&0&\hskip.7pt0\hskip.7pt&\hskip15pt0\cr
*&*&*&\hskip8.6pt-\hh\n^2\hskip8.6pt&\hskip-9pt-2\hskip-9pt&0&\hskip.7pt0\hskip.7pt&\hskip15pt0\cr
*&*&*&-2&\hskip-9pt-1\hskip-9pt&0&\hskip.7pt0\hskip.7pt&\hskip15pt0\cr
*&*&*&*&\hskip-9pt*\hskip-9pt&\hskip14.6pt1\hskip14.6pt&\hskip.7pt0\hskip.7pt&\hskip15pt0\cr
*&*&*&*&\hskip-9pt*\hskip-9pt&*&\hskip.7pt1\hskip.7pt&\hskip15pt0\cr
*&*&*&*&\hskip-9pt*\hskip-9pt&*&\hskip.7pt*\hskip.7pt&\hskip15pt1
\end{array}
\right]
\end{equation}
is the Ja\-cob\-i\-an matrix at the point (\ref{fhz}) of the mapping 
$\,\bbF^8\nh\to\bbF^8$ sending $\,(\f\nh,\h,\x,\y,\z,\p,\q,\r)\,$ to the 
octuple formed by the left-hand sides in (\ref{sys}), in which $\,\xy\,$ has 
been replaced by $\,0$. (Each row of (\ref{jac}) thus represents the 
differential of the corresponding left-hand side written as a combination of 
$\,\mathrm{d}\hskip-.8pt\f\nh,\hs\mathrm{d}\h,\hs\mathrm{d}\x
,\hs\mathrm{d}\y,\hs\mathrm{d}\z,\hs\mathrm{d}\p,\hs \mathrm{d}\q$ and 
$\,\mathrm{d}\r$, while asterisks stand for various irrelevant entries.) 
Replacing the fourth row in (\ref{jac}) by the fourth row minus twice the 
fifth row, we obtain a triangular matrix with nonzero diagonal entries, so 
that (\ref{jac}) is nonsigular and the existence of the curve required in (b), 
except for the specific value of $\,\mathrm{d}\h/\nh\mathrm{d}\hh\xy$, is 
immediate from the implicit mapping theorem.

To evaluate $\,\mathrm{d}\h/\mathrm{d}\hh\xy\,$ at $\,\xy=0$, we may treat 
$\,\f\nh,\h,\x,\y,\z,\p,\q\,$ and $\,\r$, restricted to the graph of the curve 
in (\ref{xto}), as functions of $\,\xy$. Applying 
$\,(\hskip1.7pt,\hskip1pt)'\nh=\hs\mathrm{d}/\nh\mathrm{d}\hh\xy$ to 
(\ref{sys}.i), (\ref{sys}.ii) and the second line in (\ref{uvw}), one sees 
that, at the point $\,(\xy,\f\nh,\h,\x,\y,\z,\p,\q,\r)\,$ given by 
(\ref{fhz}), $\,\f'\nnh=2\y+1$, $\,\U'\nnh=(1-\y)\hs\h'\nnh$, and 
$\,\U'\nnh=2\z-\r+\f'\nnh$. (We have used the fact that $\,\xy=\h=\f=0\,$ and 
$\,\x=-\hh\p=1$.) Thus, 
$\,(1-\y)\hs\h'\nnh=\U'\nnh=2\z-\r+\f'\nnh=2\z-\r+2\y+1$, and so 
$\,\h'\nnh=4$.%, which completes the proof.
\end{proof}
\begin{rem}\label{vweyz}For later reference, note that, by (\ref{sys}.iv) -- 
(\ref{sys}.v),
\[
\V\nnh-2\W\nnh=(\n^2\nnh-4)\hh\y\hh,\hskip5pt
\n^2\W\nnh-2\V\nnh=(\n^2\nnh-4)\hh\z\hh,\hskip5pt
(\n^2\nnh+4)\V\nnh-4\n^2\W\nnh=(\n^2\nnh-4)\hh(\n^2\y-2\z)\hh.
\]
\end{rem}

\section{Negligible polynomials}\label{np}\checked
\setcounter{equation}{0} 
Given one of the triples (\ref{lie}), for 
$\,\n=\d+\k\ge3$, and a vector space $\,\ev\,$ with $\,\dimf\ev<\infty$, we 
define a $\,\ev$-val\-ued {\smallit negligible polynomial function\/} of 
the variables
\begin{equation}\label{var}
a,b\in\mathfrak{g}\hskip7pt\mathrm{and}\hskip7pt
(\xy,\f\nh,\h,\x,\y,\z,\p,\q,\r)\in\bbF^9, 
\end{equation}
to be any mapping $\,\mathfrak{g}\times\mathfrak{g}\times\bbF^9\to\ev\,$ 
expressible as a sum of terms, each of which is mul\-ti\-lin\-e\-ar in 
$\,(a,b,\,\ldots\,)$, or $\,(b,b,\,\ldots\,)$, or $\,(\xy,b,\,\ldots\,)$, or 
$\,(\h,b,\,\ldots\,)$, with the dots standing for any finite number of 
arguments from the list $\,a,b,\xy,\f\nh,\h,\x,\y,\z,\p,\q,\r\nnh$, possibly 
with repetitions. In the case of two mappings 
$\,E,\td E:\mathfrak{g}\times\mathfrak{g}\times\bbF^9\to\ev$,
\begin{equation}\label{sim}
\mathrm{we\ write\ }\,\,E\approx\td E\,\,\mathrm{\ when\ }\,\,E-\td E\,\,
\mathrm{\ is\ negligible.}
\end{equation}
\begin{lem}\label{neglg}{\smallit 
For\/ $\,E,\td E:\mathfrak{g}\times\mathfrak{g}\times\bbF^9\to\ev\,$ as above, 
suppose that\/ $\,E\approx\td E$.
\begin{enumerate}
  \def\theenumi{{\rm\roman{enumi}}}
\item[{\rm(i)}] We have\/ $\,E=\td E\,\,$ at any\/ 
$\,(a,b,\xy,\f\nh,\h,\x,\y,\z,\p,\q,\r)\hs\,$ with\/ $\,b=0$,
\item[{\rm(ii)}] Let\/ $\,(a,b,\xy,\f\nh,\h,\x,\y,\z,\p,\q,\r)\hs\,$ be\/ 
$\,C^\infty$ functions of a variable\/ $\,t\in[\hs0,\vd)$, where\/ 
$\,\vd\in(0,\infty)$. If\/ $\,a(0)=0\,$ as well as\/ $\,\xy(0)=\h(0)=0\,$ 
and\/ $\,\jm[b\hh]=0\,$ for an integer\/ $\,k\ge1$, in the notation of\/ 
{\rm(\ref{jet})}, then $\,\jk[E\hs]=\jk[\td E\hs]$.
\end{enumerate}
}
\end{lem}
\begin{proof}This is obvious from the definition of negligibility and the 
Leib\-niz rule.
\end{proof}
In our discussion, negligible polynomials arise as follows. Choosing 
$\,(\mathfrak{g},\bbF\nnh,\ve)$ to be one of the triples (\ref{lie}), with 
$\,\n=\d+\k\ge3$, and using the notation of \hbox{(\ref{avp}) --} (\ref{ava}), 
we introduce the variables (\ref{var}) consisting of an arbitrary element 
$\,a\,$ of $\,\mathfrak{g}$, the coefficients $\,\f\nh,\x,\y,\z,\p,\q,\r\,$ of 
two arbitrary linear combinations
\begin{equation}\label{lfi}
\mathrm{i)}\hskip7pt\ly\hs=\hs\x\hh\tya+\n^2\y\hh\cya+\z\hh\mya\hh,
\hskip14pt\mathrm{ii)}\hskip7pt\fy\hs=\hs\p\hh\tya+\q\hh\cya+\r\hh\mya
+\f\nh\by\hs,
\end{equation}
the scalar $\,\xy=\ve^2(a,a)/\n\,$ as in (\ref{cea}), where 
$\,(a,a)=\mathrm{tr}\hskip2pta\nh^2\nnh$, an additional scalar variable 
$\,\h\in\bbF\nnh$, and $\,b\in\mathfrak{g}\,$ given by
\begin{equation}\label{bec}
b\,=\,c\,-\,\h a\hh,\hskip22pt\mathrm{so\ that}\hskip7ptc\,=\,\h a\,+\,b\hh,
\end{equation}
where $\,c=\ve\hh(a\nh^2)_0\w$ stands for the trace\-less part of 
$\,\ve\hh a\nh^2\nnh$, as in (\ref{ntt}.ii) and (\ref{cea}).

Expressions such as $\,\ly\,$ and $\,\fy\,$ in (\ref{lfi}), along with their 
$\,\cu$-im\-ages, $\,\ca\hskip1.7pt$-prod\-ucts and 
$\,\{\hskip2pt\cdot\hskip2pt\}$-prod\-ucts (which we may evaluate using 
(\ref{nom}), (\ref{ast}) and (\ref{bdb})), then become functions of the 
variables (\ref{var}), and so do $\,c\,$ and $\,d\,$ in (\ref{cea}). 
Due to functional-dependence relations among our variables (\ref{var}), there 
is a built-in ambiguity of such a representation, which we do not attempt to 
remove. For instance, one may treat $\,(a,a)\,$ as a function of $\,a\,$ 
alone, or write it as $\,\n\hs\ve^{-2}\xy$. Similarly, 
$\,c=\ve\hh(a\nh^2)_0\w$ equals both $\,\ve\hh a\nh^2\nh-\ve^{-1}\xy\,$ and 
$\,\h a+b$, cf.\ (\ref{eaq}.i) and (\ref{bec}).

With $\,\,\approx\,$ and $\,\ly\,$ as in (\ref{sim}) and (\ref{lfi}.i), we now 
have
\begin{equation}\label{tcl}
\begin{array}{rl}
\mathrm{i)}&\tyc\ca\hs\ly\approx\h\tya\ca\hs\ly\hh,\hskip9pt
\h\tyc\approx\h^2\tya\hh,\hskip9pt
\xy\hh\tyc\approx\xy\hh\h\tya\hh,\hskip9pt
\ve^2(a,c)/\n\approx\xy\h\hh,\\
\mathrm{ii)}&\myc\approx\h^2\mya\hh,\hskip9pt
\cyc\approx\h^2\cya\hh,\hskip9pt
\cyac\approx\h\hh\cya\hh,\hskip9pt
\myac\approx\h\hh\mya\hh.
\end{array}
\end{equation}
Namely, (\ref{tcl}.i) is obvious as $\,\tyc-\h\tya=\tyb\,$ by (\ref{bec}), 
and $\,(a,c)-\h\hh(a,a)=(a,b)$, while (\ref{tcl}.ii) follows from (\ref{bec}) 
and (\ref{cea}). Furthermore,
\begin{equation}\label{dsh}
d\approx(\h^2\nh+\xy)\hh a\hh.
\end{equation}
In fact, (\ref{eaq}.i) implies that 
$\,\ve^2a^3\nh=\ve\hh(\ve\hh a\nh^2)a=\ve\hh c\hh a+\xy a\,$ and, as a 
consequence of (\ref{eaq}.ii), $\,d=\ve\hh c\hh a+\xy a-\ve\hh(a,c)/\n$. 
Replacing $\,c\,$ here by $\,\h a+b$, and the resulting occurrence of 
$\,\ve\hh a\nh^2$ by $\,\h a+b+\ve^{-1}\xy$, which is allowed due to 
(\ref{eaq}.i) and (\ref{bec}), we see that 
$\,d-(\h^2\nh+\xy)\hh a=\ve\hs ba+\h\hh b-\ve\hh(a,b)/\n\,$ due to the 
equality $\,\xy=\ve^2(a,a)/\n$, which yields (\ref{dsh}). Now (\ref{dsh}) 
gives
\begin{equation}\label{tds}
\tyd\approx(\h^2\nh+\xy)\hh\tya\hh,\hskip9pt
\cy_{a,d}\w\approx(\h^2\nh+\xy)\hh\cya\hh,\hskip9pt
\myad\approx(\h^2\nh+\xy)\hh\mya\hh.
\end{equation}
\begin{lem}\label{stsim}{\smallit 
If\/ $\,(\mathfrak{g},\bbF\nnh,\ve)\,$ is one of the triples\/ 
{\rm(\ref{lie})}, 
and\/ $\,\n=\d+\k\ge3\hh$, then
\[
\begin{array}{ll}
\tya\ca\tya=\ty_{\!c}\w+2\hh(\xy\by+\mya-2\hh\cya)\hh,
\hskip19pt&\cya\ca\hs\cya=\xy\hh\cya\hh,\\
\mya\ca\mya\approx\xy\hh\h\tya+\xy^2\nnh\by+\h^2(\mya-\cya)\hh,
&\tya\ca\hs\cya\approx2\hh\h\hh\cya\hh,\\
\tya\ca\hh\mya\approx\xy\tya+2\hh\h\hh(\mya-\cya)\hh,
&\cya\ca\hh\mya\approx(\h^2\nh+\xy)\hh\cya\hh,\\
\by\hs\ca\tya=\tya\hh,\hskip18pt\by\hs\ca\hh\cya=\cya\hh,
&\by\hs\ca\hh\mya=\mya\hh,
\end{array}
\]\checked
with the dependence on the variables\/ {\rm(\ref{var})} described 
above. Similarly,
\[
\begin{array}{l}
\{\tya\hskip-3pt\cdot\nnh\tya\}=\xy\by+\mya-2\hh\cya\hh,\hskip29pt
\{\mya\hskip-3pt\cdot\nnh\mya\}\approx\xy\hh\mya
-\h^2\cya\hh,\phantom{{j_{j_j}}}\\
2\hh\{\tya\hskip-3pt\cdot\nnh\mya\}\approx\xy\tya-4\hh\h\hh\cya\hh,
\hskip36pt\{\cya\hskip-3pt\cdot\nnh\cya\}=0\hh,\\
2n^2\{\tya\hskip-3pt\cdot\nnh\cya\}\approx\h^2\nh\tya-2\hh\h\hh\mya
+2\hs\xy\h\hh\by\hh,\phantom{{j_{j_j}}}\\
n^2\{\cya\hskip-3pt\cdot\nnh\mya\}\approx\xy\hh\mya
-\xy^2\nnh\by-\xy\h\tya/2\hh.
\end{array}
\]\checked
}
\end{lem}
\begin{proof}This is immediate from (\ref{tcl}) -- (\ref{tds}) combined with 
(\ref{ast}) -- (\ref{bdb}).
\end{proof}
\begin{lem}\label{fdldl}{\smallit 
Under the hypotheses of Lemma\/~{\rm\ref{stsim}}, \hskip1.2ptfor\/ 
$\,\K,\U,\V,\W\,$ given by\/ {\rm(\ref{uvw})}, 
\[
4\hh\{(\dj\hh\ly)\nnh\cdot\nnh(\dj\hh\ly)\}
\approx(\U-\K\x+\h\x\p)\hh\tya+(\V\nnh-\n^2\K\y)\hh\cya+(\W\nnh-\K\z)\hh\my_a
+(\x^2\nh+\xy\z^2)\hs\xy\by\hh,
\]\checked
where\/ $\,\ly=\x\hh\tya+\n^2\y\hh\cya+\z\hh\mya$, as in\/ {\rm(\ref{lfi})}.
}
\end{lem}
\begin{proof}The second part of Lemma~\ref{stsim} yields
\[
\begin{array}{l}
\{\ly\nnh\cdot\nnh\ly\}\hs
\approx\hs(\h^2\x\y+\xy\x\z-\h\hs\xy\y\z)\hh\tya
-[2\x^2+(\h\z+4\x)\hh\h\z]\hh\cya\\
\phantom{\{\ly\nnh\cdot\nnh\ly\}\hs}
+\,(\x^2\nh+\xy\z^2\nh-2\h\x\y+2\hs\xy\y\z)\hh\mya
+(\x^2\nh+2\h\x\y-2\hs\xy\y\z)\hs\xy\by\hh.
\end{array}
\]\checked
By (\ref{imo}), $\,(\hh\mathrm{Id}\,-\,\cu\hh)\hh\ly
=(\n^2\y+2\z)\hh\cya+(2\y+\z)\hh\mya-2\hs\xy\y\hh\by-\y\hh\tyc$. Replacing 
$\,\y\hh\tyc$ with $\,\h\y\hh\tya$, cf.\ (\ref{tcl}.i), we get, from 
(\ref{tcl}.i) and the first part of Lemma~\ref{stsim},
\[
\begin{array}{l}
[(\hh\mathrm{Id}-\cu\hh)\hs\ly\hh]\ca\hs\ly\approx
(\xy\x\z+\h\hs\xy\y\z-\h^2\x\y+\h\hs\xy\z^2)\hh\tya\\
\phantom{[(\hh\mathrm{Id}-\cu\hh)\hs\ly\hh]\ca\hs\ly\,}
+\,[\V\nh-\n^2\K\y+2\x^2\nh+(\h\z+4\x)\hh\h\z]\hh\cya\\
\phantom{[(\hh\mathrm{Id}-\cu\hh)\hs\ly\hh]\ca\hs\ly\,}
+\,[\h^2\z^2\nh+2(\h\x\y+\h\x\z-\xy\y\z)]\hh\mya
+(\xy\z^2\nh+2\hs\xy\y\z-2\hh\h\x\y)\hs\xy\by\hs.
\end{array}
\]\checked
In view of (\ref{edd}.c), $\,4\hh\{(\dj\hh\ly)\nnh\cdot\nnh(\dj\hh\ly)\}
=\{\ly\nnh\cdot\nnh\ly\}+[(\hh\mathrm{Id}-\cu\hh)\hs\ly\hh]\ca\hs\ly$, as 
required.
\end{proof}

\section{The first crucial step in the argument}\label{sa}\checked
\setcounter{equation}{0}
The following result will be used both to verify that the family $\,\ec\,$ 
defined by formula (\ref{cdl}) of Section~\ref{mr} actually consists of 
Ein\-stein connections, and to prove, later in Section~\ref{ra}, that 
$\,\ec\,$ contains every real-an\-a\-lyt\-ic curve of weak\-\hbox{ly\hh-}\hskip0ptEin\-stein 
connections emanating from $\,\dj$.
\begin{thm}\label{dsapp}{\smallit 
With the assumptions and notations as in Lemmas\/~{\rm\ref{stsim}} 
and\/~{\rm\ref{fdldl}},
\begin{equation}\label{ddl}
4\hh\dj\hh\{(\dj\hh\ly)\nnh\cdot\nnh(\dj\hh\ly)\}
+(\hh\dj\hh\ly)\hh\fy\,\approx\hs\dj(\U\nh\tya+\h\x\p\hh\tya+\V\nh\cya
+\W\nh\mya+\x\p\hh\tyb)\hh,
\end{equation}\checked
while for\/ 
$\,\mj=\dz(\dj\hh\ly)+4\hh\dj\hh\{(\dj\hh\ly)\nnh\cdot\nnh(\dj\hh\ly)\}
+(\hh\dj\hh\ly)\hh\fy\,$ we have
\begin{equation}\label{med}
\begin{array}{l}
\mj\hs\approx\,\dj\hs[\hs\U\tya\,+\,(\y+\x\p)\hh\h\hh\tya\,
+\,(\y+\x\p)\hh\tyb]\\
\phantom{\mj\hs}+\,\,\dj\hs[(\V\nnh-\n^2\y-2\z)\hh\cya
+(\W\nnh-2\y-\z)\hh\my_a]\hh,
\end{array}
\end{equation}\checked
and, setting\/ 
$\,\zy=\fy-8\hh\{\dj\nnh\cdot\nnh(\dj\hh\ly)\}
-4\hh\{(\dj\hh\ly)\nnh\cdot\nnh(\dj\hh\ly)\}$, we obtain
\begin{equation}\label{zap}
\begin{array}{l}
\zy\approx(\p+\x-2\hs\xy\x\z-\h\hs\xy\z^2\nh-\h\y)\hh\tya
+[\hs\q+\n^2(\K+2)\hh\y-\V\nnh+2\z]\hh\cya\\
\phantom{\zy\,}+\,[\hs\r+(\K+2)\hh\z-\W\nnh+2\y]\hh\my_a
+[\hh\f-(2\y+\x^2\nh+\xy\z^2)\hs\xy]\hh\by-\y\hh\tyb\hh.
\end{array}
\end{equation}\checked
}
\end{thm}
\begin{proof}Replacing $\,c\,$ in (\ref{tdt}) with $\,\h a+b\,$ (see 
(\ref{bec})) we have, from (\ref{dta}),
\begin{equation}\label{dtc}
\begin{array}{l}
\tya(\dj\hh\tya)=\dj\hh(\h\hh\tya)+\dj\hh\tyb\hh,\hskip8pt
\tya(\dj\hh\cya)\approx\dj\hh(\h\hh\cya)\hh,\hskip8pt
\tya(\dj\hh\mya)\approx\dj\hh(\h\hh\mya)\hh,\\
\cya(\dj\hh\tya)=\cya(\dj\hh\cya)=\cya(\dj\hh\mya)=0\hh,\\
\mya(\dj\hh\tya)\approx-\hh\dj\hh(\xy\hh\tya)\hh,\hskip8pt
\mya(\dj\hh\cya)\approx-\hh\dj\hh(\xy\hh\cya)\hh,\hskip8pt
\mya(\dj\hh\mya)\approx-\hh\dj\hh(\xy\hh\mya)\hh,\\
\by(\dj\hh\tya)=\dj\hh\tya\hh,\hskip8pt
\by(\dj\hh\cya)=\dj\hh\cya\hh,\hskip8pt
\by(\dj\hh\mya)=\dj\hh\mya\hh,
\end{array}
\end{equation}\checked
the last line being obvious from (\ref{nes}) since $\,\sy=\by\,$ corresponds 
via (\ref{sgv}) to $\,\Sigma=\mathrm{Id}$. If one were to treat $\,\approx\,$ 
as equality ignore the ``correction term'' $\,\dj\hh\tyb$, the 
$\,\dj$-im\-ages of $\,\tya,\cya,\mya$, and of their combination $\,\ly$, 
would lie in the eigen\-space, for the eigen\-value $\,\h,0,-\hs\xy\,$ or 
$\,1$, of the operator $\,\es\to\es\,$ sending $\,\sj$, respectively, to 
$\,\tya\sj,\cya\sj,\mya\sj$ or $\,\by\sj$. Thus, up to such an equivalence, 
$\,\dj\hh\ly\,$ lies in the eigen\-space of the operator 
$\,\sj\mapsto\fy\sj=(\p\hh\tya+\q\hh\cya+\r\hh\mya+\f\nh\by)\sj\,$ for the 
combined eigen\-value $\,\K\,$ given by (\ref{uvw}). Restoring the 
correction term, we now have 
$\,\fy\hh(\hh\dj\hh\ly)\,\approx\,\dj\hh(\K\ly+\x\p\hh\tyb)$, and so
\[
(\hh\dj\hh\ly)\hh\fy\,\approx\,\dj\hh(\K\ly+\x\p\hh\tyb)\hh,
\]
since Lemma~\ref{psidl} and its proof remain valid for $\,\approx\,$ 
equivalences instead of equalities. Hence (\ref{ddl}) is immediate from 
Lemma~\ref{fdldl}, with $\,\dj\by=0\,$ by (\ref{fdd}.c). Relation (\ref{med}) 
is in turn a trivial consequence of (\ref{ddl}) and formula (\ref{dio}) for 
$\,\sy=\ly$, combined with (\ref{imo}) and (\ref{bec}). Finally, as 
(\ref{edd}.b) yields 
$\,-8\hh\{\dj\nnh\cdot\nnh(\dj\hh\ly)\}=\ly+(\hh\mathrm{Id}\,
-\,\cu\hh)\hs\ly$, (\ref{zap}) easily follows from (\ref{lfi}), (\ref{bec}), 
(\ref{imo}) and Lemma~\ref{fdldl}.
\end{proof}
\begin{rem}\label{polym}The polynomial mappings $\,\jz,\mz\,$ of 
Section~\ref{oa} are defined as follows. For 
$\,\mathbf{v}=(\xy,\f\nh,\h,\x,\y,\z,\p,\q,\r)\in\bbF^9\nnh$, 
the first component of $\,\jz\hh(\nh\mathbf{v}\nh)\,$ is $\,0$, and the 
other eight components are the left-hand sides of (i), (iii), (ii) and 
(iv) -- (viii) in (\ref{sys}), with (\ref{uvw}), while 
$\,\mz\hh(\nh\mathbf{v}\nh)=(0,0,0,\y+\x\p,0,0,-\hh\y,0,0)$.
\end{rem}

\section{The main results}\label{mr}
\setcounter{equation}{0}
Given a real (or, complex) Lie group $\,\gj\,$ with the Lie algebra 
$\,\mathfrak{g}$, we again refer to the Le\-\hbox{vi\hh-}\hskip0ptCi\-vi\-ta connections of 
left-in\-var\-i\-ant Ein\-stein metrics on $\,\gj\,$ as Ein\-stein 
connections in $\,\mathfrak{g}$. The metrics themselves are 
pseu\-\hbox{do\hskip1pt-}\hskip0ptRiem\-ann\-i\-an or, respectively, 
hol\-o\-mor\-phic; $\,\by\,$ and $\,\dj=[\hskip2pt,\hskip.6pt]/2\,$ stand for 
the Kil\-ling form, with (\ref{bvw}), and its Le\-\hbox{vi\hh-}\hskip0ptCi\-vi\-ta connection; 
and $\,\dj\hh\ly\,$ is defined as in  (\ref{sns}) or, equivalently, 
(\ref{dsv}.a). The more general classes of uni\-mod\-u\-lar tor\-sion-free 
connections with parallel Ric\-ci tensor, and weak\-\hbox{ly\hh-}\hskip0ptEin\-stein connections, 
were introduced in Section~\ref{ec}.

Whenever $\,(\mathfrak{g},\bbF\nnh,\ve)\,$ is one of the 
triples (\ref{lie}), with $\,\n=\d+\k$, while $\,\xy,\h\in\bbF\nnh$, the 
scalar $\,\ve^{-1}\xy\,$ denotes a multiple of $\,\mathrm{Id}$, and 
$\,a\in\mathfrak{g}$, one clearly has
\begin{equation}\label{ahx}
\ve\hh a\nh^2\nh=\h a+\ve^{-1}\xy
\end{equation}
if and only if $\,\xy=\ve^2(a,a)/\n\,$ and $\,c=\h a$, where 
$\,c=\ve\hh(a\nh^2)_0\w$. (Notation of (\ref{cea}).)
\begin{thm}\label{xmpls}{\smallit
Let\/ $\,(\mathfrak{g},\bbF\nnh,\ve)\,$ be one of\/ {\rm(\ref{lie})}, 
$\,\n=\d+\k\ge3\hh$. If\/ $\,\ly,\fy\,$ are given by\/ {\rm(\ref{lfi})} 
\hskip2ptfor\/ $\,\xy,\f\nh,\h,\x,\y,\z,\p,\q,\r\in\bbF\,$ satisfying\/ 
{\rm(\ref{sys})} with\/ {\rm(\ref{uvw})}, except\/ {\rm(\ref{sys}.iii)}, 
while\/ $\,a\in\mathfrak{g}\,$ and\/ {\rm(\ref{ahx})} holds, then\/ 
$\,\nabla\nh=\dj+\dj\hh\ly\,$ is a weak\-\hbox{ly\hh-}\hskip0ptEin\-stein, as well as 
uni\-mod\-u\-lar, tor\-sion-free connection in\/ $\,\mathfrak{g}\,$ with the\/ 
$\,\nabla\nnh$-par\-al\-lel Ric\-ci tensor\/ 
$\,\rho\nnh^\nabla\nnh=-\hh(\by+\fy)/4$.
}
\end{thm}
\begin{proof}By (\ref{bec}), the assumption (\ref{ahx}), that is, $\,c=\h a$, 
gives $\,b=0$, and so, in view of Lemma~\ref{neglg}(i), $\,\approx\,$ 
equivalences are actually equalities. Applied to (\ref{med}) and 
(\ref{zap}), this yields $\,\mj=0\,$ and $\,\zy=0$, since the right-hand sides 
of (\ref{med}) and (\ref{zap}) vanish as a consequence of (\ref{sys}) (while 
(\ref{sys}.iii) is not used). Combined with the definitions of $\,\mj\,$ and 
$\,\zy\,$ in Theorem~\ref{dsapp}, the relations $\,\mj=0\,$ and $\,\zy=0$ 
state that, by (\ref{cnd}), condition (a) of Lemma~\ref{equiv} holds for 
$\,\sj=\dj\hh\ly\,$ and $\,\sy=\fy$. Since (b) in Lemma~\ref{equiv} is 
equivalent to (a), our claim follows, cf. (\ref{inc}).
\end{proof}
Under the assumption $\,c=\h a\,$ (that is, (\ref{ahx})) made in 
Theorem~\ref{xmpls},
\begin{equation}\label{nrm}
\mathrm{(\ref{sys}.iii)\ holds\ if\ and\ only\ if\ }
\,\pr\hh\ly=\tya\hs,\mathrm{\ where\ }\ly=\x\hh\tya+\n^2\y\hh\cya+\z\hh\mya\hh,
\end{equation}
as one sees using Lemma~\ref{omgpr}(iv). The dependence of $\,\tya$ (or, $\,\cya$ 
and $\,\mya$) on $\,a\,$ is, by (\ref{ava}), homogeneous linear (or, respectively, 
quadratic). Consequently, $\,\ly\,$ and $\,\fy\,$ in (\ref{lfi}), along with 
$\,\nabla\,$ and $\,\rho\nnh^\nabla\nnh$, will remain unchanged if we replace 
$\,a\,$ with a nonzero multiple, and at the same time suitably rescale 
$\,\f\nh,\x,\y,\z,\p,\q\,$ and $\,\r$. Thus, (\ref{sys}.iii) is a {\smallit 
normalizing condition}, which can always be realized by rescaling, as long as 
$\,\pr\hh\ly\ne0\,$ (or, equivalently, 
$\,(n^2\nnh-4)\hh\x+(\n^2\y-2\z)\hh\h\ne0$), while the system (\ref{sys}) 
{\smallit without\/} equation (\ref{sys}.iii) is re\-scal\-ing-in\-var\-i\-ant.

In terms of the inclusions $\,\ee\subset\,\eu\subset\ew\nh$, cf.\ (\ref{inc}), 
all connections obtained in Theorem~\ref{xmpls} lie in $\,\eu$, and hence in 
$\,\ew\,$ but, as shown below in Section~\ref{ne}, they need not be Ein\-stein 
connections (elements of $\,\ee$).

However, Ein\-stein connections do arise in a special case of 
Theorem~\ref{xmpls}, which is the first part of the main result of this paper, 
stated as follows.
\begin{thm}\label{mnres}{\smallit 
Suppose that\/ $\,(\mathfrak{g},\bbF\nnh,\ve)\,$ is one of the triples\/ 
{\rm(\ref{lie}):}
\[
(\mathfrak{sl}\hh(\n,\bbR),\bbR,1),\hskip2pt
(\mathfrak{sl}\hh(\n,\bbC),\bbC,1),\hskip2pt
(\mathfrak{sl}\hh(\n,\bbC),\bbC,\mathrm{i}),\hskip2pt
(\mathfrak{su}\hh(\d,\k),\bbR,\mathrm{i}),\hskip2pt
(\mathfrak{sl}\hh(\n/2,\bbH),\bbR,1)\hh,
\]
with\/ $\,\n=\d+\k\ge3\hh$, and\/ $\,\tya,\cya,\mya$ are defined by\/ 
{\rm(\ref{ava})}. Then the set
\begin{equation}\label{cdl}
\begin{array}{l}
\ec\,=\,\dj+\el\,\,\mathrm{\ for\ }\,\,
\el\,=\,\{\dj\hh\eya:a\in\mathfrak{g}\hs\,\mathrm{\ and\ }\,a\nh^2\nh=0\}\hh,\\
\mathrm{where \ \ \ }\,\eya\,=\,\tya\hs-\,(\n^2\nnh-4)^{-1}[4\n^2\cya
-(\n^2\nnh+4)\hh\mya]\hh,
\end{array}
\end{equation}
consists of Ein\-stein connections in\/ $\,\mathfrak{g}\hh$. If\/ $\,\by\,$ 
denotes the Kil\-ling form, with\/ {\rm(\ref{bvw})},
\begin{enumerate}
  \def\theenumi{{\rm\roman{enumi}}}
\item[{\rm(i)}] every\/ $\,\nabla\in\ec\,$ has the nondegenerate\/ 
$\,\nabla\nnh$-par\-al\-lel Ric\-ci tensor\/ 
$\,\rho\nnh^\nabla\nnh=(\eya\nnh-\by)/4$,
\item[{\rm(ii)}] the assignment\/ $\,a\mapsto\dj+\dj\hh\eya$ maps\/ 
$\,\ep\nh=\{a\in\mathfrak{g}:a\nh^2\nh=0\}\,$ bijectively 
onto\/ $\,\ec$,
\item[{\rm(iii)}] the inverse of the bijection in\/ {\rm(ii)} is\/ 
$\,\nabla\mapsto a\,$ for\/ $\,a\in\mathfrak{g}\,$ uniquely characterized 
by\/ $\,\pr\hh(\nabla\nh-\dj)=\dj\hh\tya$, with\/ $\,\pr\,$ as in\/ 
{\rm(\ref{prj}.b)}.
\end{enumerate}
Finally, \hskip1.4ptfor the set\/ $\,\ec\,$ in\/ {\rm(\ref{cdl})}, and\/ 
$\,\ew\hs$ given by\/ {\rm(\ref{inc})},
\begin{enumerate}
  \def\theenumi{{\rm\roman{enumi}}}
\item[{\rm(iv)}] $\ec\,$ is an algebraic set in\/ 
$\,\ey=[\mathfrak{g}\nh^*]^{\otimes2}\nnh\otimes\mathfrak{g}\hh$, and hence a closed 
subset of\/ $\,\ey$,
\item[{\rm(v)}] $\ec\,$ is relatively open in the set\/ $\,\,\ee\,$ of all 
Ein\-stein connections, and in\/ $\,\ew$,
\item[{\rm(vi)}] $\ec\,$ forms a connected component of both\/ $\,\,\ee\,$ 
and\/ $\,\ew$.
\item[{\rm(vii)}] $\dimf\ec=[\n^2\nnh/2]\,$ for\/ 
$\,\mathfrak{g}=\mathfrak{sl}\hh(\n,\bbF)$, while\/ 
$\,\dimr\ec=2\hs\d\k\,$ when\/ $\,\mathfrak{g}=\mathfrak{su}\hh(\d,\k)$, and\/ 
$\,\dimr\ec=4[n^2\nnh/8]\,$ if\/ $\,\mathfrak{g}=\mathfrak{sl}\hh(\n/2,\bbH)$,
\item[{\rm(viii)}] for the three choices of\/ $\,\mathfrak{g}\,$ in\/ 
{\rm(vii)}, $\ec\,$ is a union of\/ $\,[\n/2]+1$, or of\/ $\,(\k+1)(\k+2)/2\hh$, 
or, respectively, of\/ $\,[\n/4]+1\,$ ad\-joint-ac\-tion orbits.
\end{enumerate}
Explicitly, for\/ $\,\nabla\nh=\dj+\dj\hh\eya$ and\/ $\,v,w\in\mathfrak{g}\hh$, we 
have
\begin{equation}\label{nvw}
\begin{array}{l}
2\hs\nsvw\hs=\hs[v,\hs w]\hs+\hs\ve\hs[a,\hs vw\hs+\hs wv\hs+\hs\ej_v^aw]\hh,
\hskip22pt\mathrm{where}\\
(\n^2\nnh-4)\hh\ej_v^aw=(\n^2\nnh+4)\hs\ve\hs(vaw+wav)\hs
-\hs4\n\hs\ve\hs(\mathrm{tr}\hskip2ptav)\hh w\hs
-\hs4\n\hs\ve\hs(\mathrm{tr}\hskip2ptaw)\hh v\hs.
\end{array}
\end{equation}
}
\end{thm}
\begin{thm}\label{isola}{\smallit
A fixed pos\-i\-tive-def\-i\-nite multiple of the Kil\-ling form is isolated in 
the set of suitably normalized left-in\-var\-i\-ant Riemannian Ein\-stein 
metrics on\/ $\,\mathrm{SU}\hh(\n)\hh$, $\,\n\ge3\hh$.
}
\end{thm}
Theorem~\ref{isola} settles a rather narrow special case of B\"ohm, Wang and 
Ziller's Finiteness Conjecture \cite[p.\ 683]{bohm-wang-ziller}.

It is not known whether the Kil\-ling form of $\,\mathrm{SU}\hh(\n)$, for 
$\,\n\ge3$, represents an isolated point of the moduli space of all Riemannian 
Ein\-stein metrics on $\,\mathrm{SU}\hh(\n)$ modulo diffeomorphisms and 
scaling. Cf.\ \cite[Theorem~1.2]{koiso}.
\begin{thm}\label{slcre}{\smallit 
The conclusions of Theorem\/~{\rm\ref{mnres}} remain valid for\/ 
$\,\mathfrak{sl}\hh(\n,\bbC)\,$ treated as a real Lie algebra, except that the 
the formula in\/ {\rm(i)} now reads 
$\,\rho\nnh^\nabla\nnh=2\,\mathrm{Re}\hskip1.7pt(\eya\nnh-\by)/4$.

In other words, $\,\ec\,$ is a connected component of the set of all real 
Ein\-stein connections in\/ $\,\mathfrak{sl}\hh(\n,\bbC)$, and similarly for 
weak\-\hbox{ly\hh-}\hskip0ptEin\-stein connections. In particular, all weak\-\hbox{ly\hh-}\hskip0ptEin\-stein 
connections in the underlying real Lie algebra of\/ 
$\,\mathfrak{sl}\hh(\n,\bbC)$, sufficiently close to\/ $\,\dj$, are 
necessarily hol\-o\-mor\-phic, that is, constitute\/ $\,\bbC$-bi\-lin\-e\-ar 
mappings\/ $\,\mathfrak{g}\times\mathfrak{g}\to\mathfrak{g}$.
}
\end{thm}
The conclusion in the final paragraph of the Introduction easily follows from 
Theorem~\ref{slcre}. See also Remark~\ref{rehol}.

The proof of Theorem~\ref{mnres} will be given later, in Sections~\ref{fp} and 
\ref{pa}. For proofs of Theorems~\ref{isola} and~\ref{slcre}, see 
Sections~\ref{is} and~\ref{pa}.

\section{Proof of Theorem~\ref{mnres}, first part}\label{fp}\checked
\setcounter{equation}{0}
In this section we verify all claims made in Theorem~\ref{mnres} except (v) -- 
(viii).

The hypotheses of Theorem~\ref{xmpls} are clearly satisfied by any 
$\,a\in\mathfrak{g}\,$ such that $\,a\nh^2\nh=0\,$ and the nonuple 
$\,(\xy,\f\nh,\h,\x,\y,\z,\p,\q,\r)\,$ with (\ref{fhz}). Since (\ref{lfi}) 
then yields $\,\ly=\hh-\fy=\eya$, Theorem~\ref{xmpls} implies that $\,\ec\,$ 
with (\ref{cdl}) consists of tor\-sion-free connections $\,\nabla\,$ having 
the $\,\nabla\nnh$-par\-al\-lel Ric\-ci tensor given by the formula in (i).

As $\,a\nh^2\nh=0\,$ and $\,\xy=\h=0$, (\ref{cea}) and (\ref{bec}) yield 
$\,c=b=0$. According to Lemma~\ref{neglg}(i), $\,\approx\,$ equivalences in 
the first part of Lemma~\ref{stsim} now become equalities: 
$\,\tya\ca\tya=2\hh(\mya-2\hh\cya)\,$ and 
$\,\cya\ca\hs\cya=\mya\ca\mya=\tya\ca\hs\cya=\tya\ca\hh\mya=\cya\ca\hh\mya=0$. 
\hbox{All three\hh-fac\-tor} $\,\ca\hskip1.7pt$-prod\-ucts of linear 
combinations of $\,\tya,\cya,\mya$ must therefore vanish. In particular, 
$\,\Sigma^3\nh=0\,$ for the en\-do\-mor\-phism $\,\Sigma\,$ of 
$\,\mathfrak{g}\,$ corresponding to $\,\sy=\eya$ via (\ref{sgv}), and so 
$\,0\,$ is the only eigenvalue of $\,\Sigma$. Consequently, $\,-1\,$ is the 
only eigenvalue of $\,\Sigma-\mathrm{Id}$, and the Ric\-ci tensor 
$\,\rho\nnh^\nabla\nnh=(\eya\nnh-\by)/4\,$ is nondegenerate, as required in 
(i). Since $\,\rho\nnh^\nabla$ is $\,\nabla\nnh$-par\-al\-lel, $\,\nabla\,$ is 
its Le\-\hbox{vi\hh-}\hskip0ptCi\-vi\-ta connection (Remark~\ref{lccon}). 
Thus, $\,\rho\nnh^\nabla$ is an Ein\-stein metric, which shows that $\,\ec\,$ 
consists of Ein\-stein connections.

Let $\,\nabla\in\ec$. Since (\ref{sys}.iii) holds, (\ref{nrm}) yields 
$\,\pr\hh\eya=\tya\,$ and so, from Remark~\ref{dpepd}, 
$\,\pr\hh(\dj\hh\eya)=\dj\hh\tya$. (Again, as 
$\,\xy,\f\nh,\h,\x,\y,\z,\p,\q\,$ and $\,\r\,$ are given by (\ref{fhz}), 
$\,\ly=\x\hh\tya+\n^2\y\hh\cya+\z\hh\mya$ coincides with $\,\eya$.) This 
proves assertions (ii) -- (iii): 
$\,\nabla\in\ey=[\mathfrak{g}\nh^*]^{\otimes2}\nnh\otimes\mathfrak{g}\,$ is an 
element of $\,\ec\,$ if and only if
\begin{equation}\label{nqe}
\nabla\nh=\hs\dj+\dj\hh\eya\hskip14pt\mathrm{and}\hskip14pta\nh^2\nh=0\hh,
\end{equation}
where $\,a\in\mathfrak{g}\,$ is uniquely characterized by 
$\,\pr\hh(\nabla\nh-\dj)=\dj\hh\tya$, that is, by being the image of 
$\,\pr\hh(\nabla\nh-\dj)\,$ under the inverse of the iso\-mor\-phism in 
Remark~\ref{lniso}.

Now (iv) follows as well, since (\ref{nqe}) is a system of (nonhomogeneous) 
quadratic equations imposed on $\,\nabla\nnh$.

Finally, the explicit expression (\ref{nvw}) is immediate from (\ref{tdv}), 
(\ref{dta}) combined with symmetry of $\,\dj\hh\eya$ (due to (\ref{sns}) 
applied to $\,\nabla\nh=\dj$), and (\ref{acb}).

\section{The rationality condition}\label{rc}\checked
\setcounter{equation}{0}
The use of Theorem~\ref{xmpls} to construct examples of weak\-\hbox{ly\hh-}\hskip0ptEin\-stein 
connections other than $\,\dj\,$ requires not just finding 
$\,\xy,\f\nh,\h,\x,\y,\z,\p,\q,\r\in\bbF\,$ that satisfy (\ref{sys}), but also 
realizing the assumption (\ref{ahx}) with $\,a\ne0$. This leads to the 
following additional restriction.
\begin{lem}\label{ratnl}{\smallit 
If\/ $\,(\mathfrak{g},\bbF\nnh,\ve)\,$ are as in\/ {\rm(\ref{lie})} 
and\/ $\,\n=\d+\k\ge3\hh$, while\/ $\,\xy,\h\in\bbF\,$ and there exists\/ 
$\,a\in\mathfrak{g}\smallsetminus\{0\}\,$ satisfying\/ {\rm(\ref{ahx})}, 
then
\begin{equation}\label{rtl}
\n\nh^+\nh\n\nh^-\h^2\nh=(\n\nh^+\nnh-\hs\n\nh^-)^2\xy\,\,\mathrm{\ for\ some\ 
integers\ }\,\hs\n\nh^\pm\nnh\ge1\,\mathrm{\ with\ }\,\n\nh^+\nnh+\hs\n\nh^-\nh
=\n\hh.
\end{equation}
}
\end{lem}
\begin{proof}For $\,a\,$ treated as a complex matrix, (\ref{ahx}) gives  
$\,(\ve\hh a-\h/2)^2\nh=\omega^{\hh2}\nnh$, where $\,\omega$ is a complex 
square root of $\,\xy+\h^2\nnh/4$. We may assume that $\,\omega\ne0$, since 
otherwise $\,\ve\hh a-\h/2\,$ is nilpotent, and therefore trace\-less, so that 
$\,\h=0$, and $\,\xy=\ve^2(a,a)/\n=0$, which yields (\ref{rtl}). Now 
$\,\bbC^\n$ is the direct sum of $\,\mathrm{Ker}\,(\ve\hh a\mp\omega-\h/2)$, 
the eigen\-spaces of $\,a$. Their dimensions 
$\,\n\nh^\pm$ are positive, or else $\,a\ne0\,$ would be a trace\-less 
multiple of $\,\mathrm{Id}$. Also, 
$\,0=\mathrm{tr}\hskip2pt\ve\hh a=\n\nh^+(\omega+\h/2)-\n\nh^-(\omega-\h/2)
=(\n\nh^+\nnh-\hs\n\nh^-)\hs\omega+n\hh\h/2$, that is, 
$\,\omega=\n\hh\h/[2(\n\nh^-\nnh-\hs\n\nh^+)]$. As 
$\,\omega^{\hh2}\nnh=\xy+\h^2\nnh/4$, this gives 
$\,\xy+\h^2\nnh/4=\omega^{\hh2}\nnh
=\n^2\h^2\nnh/[2(\n\nh^+\nnh-\hs\n\nh^-)]^2\nnh$, completing the proof.
\end{proof}

\section{Real-an\-a\-lyt\-ic curves through $\,0\,$ in 
$\,\hz\hs^{-\nh1}(0)$}\label{ra}\checked
\setcounter{equation}{0}
As before, in this section $\,(\mathfrak{g},\bbF\nnh,\ve)\,$ is one of the 
triples (\ref{lie}), $\,\n=\d+\k\ge3$, and $\,\et,\es\,$ are the spaces 
(\ref{sbs}), while $\,\hz:\es\to\es\,$ is given by (\ref{ths}.i). Recall that 
tor\-sion-free connections $\,\nabla\,$ in $\,\mathfrak{g}\,$ form the 
af\-fine space $\,\dj+\es$, and so $\,\nabla\nh=\dj+\hs\sj$, where 
$\,\sj:\mathfrak{g}\times\mathfrak{g}\to\mathfrak{g}\,$ is 
$\,\bbF$-bi\-lin\-e\-ar and symmetric.

The following lemma provides the second crucial step in our argument: the 
conclusion that, for the sets $\,\ec\,$ and $\,\ew\,$ appearing in (\ref{cdl}) 
and (\ref{inc}), $\,\ec\,$ contains all real-an\-a\-lyt\-ic curves in $\,\ew$, 
emanating from $\,\dj$.
\begin{lem}\label{ractz}{\smallit 
Every real-an\-a\-lyt\-ic curve\/ 
$\,[\hs0,\vd)\ni t\mapsto\dj+\hs\sj(t)\in\dj+\es$, consisting of 
left-in\-var\-i\-ant tor\-sion-free connections in\/ $\,\mathfrak{g}\,$ with\/ 
$\,\hz\hs(\sj(t))=0\,$ and\/ $\,\sj(0)=0\hh$, lies entirely in the set\/ 
$\,\ec=\dj+\el\,$ of Ein\-stein connections, defined by\/ {\rm(\ref{cdl})}. 
This remains true for
\begin{equation}\label{slr}
\mathrm{the\ underlying\ real\ Lie\ algebra\ }\,\mathfrak{g}_\bbR\s\mathrm{\ 
of\ }\,\mathfrak{g}=\mathfrak{sl}\hh(\n,\bbC)\hh,
\end{equation}
with\/ $\,\bbC$-bi\-lin\-e\-ar\/ $\,\dj\hh\eya$ in\/ {\rm(\ref{cdl})}, even 
though\/ $\,\,\sj(t)\,$ are only assumed\/ $\,\bbR$-bi\-lin\-e\-ar.
}
\end{lem}
\begin{proof}For such a curve $\,t\mapsto\sj(t)$, let $\,\sy=\sy(t)\in\et\,$ 
and $\,a=a(t)\in\mathfrak{g}\,$ be defined by (\ref{fds}.b) and 
$\,\pr\hh\sj=\dj\hh\tya$, where $\,\sj=\sj(t)\,$ and $\,\pr\,$ is the 
projection (\ref{prj}.b). (See Remark~\ref{lniso}.) Besides $\,\sj,\sy\,$ and 
$\,a$, we introduce another curve parametrized by $\,t\in[\hs0,\vd)$, namely, 
$\,c=c\hh(t)\in\mathfrak{g}$, with (\ref{cea}), 
plus nine $\,\bbF$-val\-ued real-an\-a\-lyt\-ic functions of $\,t$, 
which are $\,\xy=\ve^2(a,a)/\n$, for $\,(a,a)=\mathrm{tr}\hskip2pta\nh^2\nnh$, 
and $\,\f\nh,\h,\x,\y,\z,\p,\q,\r$, depending on $\,\xy\,$ (and hence on 
$\,t$, with smaller $\,\vd$) via (\ref{xto}). We also use $\,\ly=\ly(t)$ 
and $\,\fy=\fy(t)\,$ given by (\ref{lfi}), with our 
$\,\f\nh,\h,\x,\y,\z,\p,\q,\r\,$ and $\,a=a(t)$. For all $\,t$,
\begin{equation}\label{eqs}
\begin{array}{rl}
\mathrm{i)}\hskip8pt&
-\dz\sj\,=\,4\hh\dj\hs\{\sj\nnh\cdot\nnh\sj\nh\}\,
+\,\sj\sy\hh,\phantom{{j_{j_j}}}\\
\mathrm{ii)}\hskip8pt&
\sy\,=\,8\hh\{\dj\nnh\cdot\hskip-1.3pt\sj\}\,
+\,4\hh\{\sj\nnh\cdot\nnh\sj\nh\}\hh,\\
\mathrm{iii)}\hskip8pt&
\sj\,-\,\dj\hh\ly\,\in\,\dz(\es)\hh.
\end{array}
\end{equation}
In fact, as $\,\hz\hs(\sj(t))=0$, (\ref{hse}) gives (\ref{eqs}.i\hs-\hs ii). 
Our $\,\xy,\f\nh,\h,\x,\y,\z,\p,\q,\r\,$ satisfy (\ref{sys}), including 
(\ref{sys}.iii), so that, from (\ref{nrm}), $\,\pr\hh\ly=\tya$, and hence 
$\,\pr\hh(\dj\hh\ly)=\dj\hh\tya$ (see Remark~\ref{dpepd}). However, our choice 
of $\,a\,$ is characterized by $\,\pr\hh\sj=\dj\hh\tya$. Thus, 
$\,\pr\hh(\sj-\dj\hh\ly)=0$, which yields (\ref{eqs}.iii).

We now proceed to show that, for $\,b=c-\h a$, at every $\,t\in[\hs0,\vd)$,
\begin{equation}\label{ceh}
\mathrm{i)}\hskip8ptb=0\hs,\hskip22pt\mathrm{ii)}\hskip8pt\sj
=\dj\hh\ly\hs,\hskip22pt\mathrm{iii)}\hskip8pt\sy=\fy\hs.
\end{equation}
[In case (\ref{slr}), $\,\sj=\sj(t)\,$ and $\,\sy\,$ are 
$\,\bbR$-bi\-lin\-e\-ar, and (\ref{eqs}) needs to be rewritten so as to use 
the notation of (\ref{tsy}). However, 
$\,a\in\mathfrak{g}=\mathfrak{sl}\hh(\n,\bbC)$, and $\,\xy\,$ (equal to 
$\,\ve^2(a,a)/\n$), as well as $\,\f\nh,\h,\x,\y,\z,\p,\q,\r\,$ are {\smallit 
complex\/} numbers. Consequently, $\,\ly,\fy$ are $\,\bbC$-bi\-lin\-e\-ar 
(since so are their ingredients $\,\tya,\cya,\mya,\by$), and (\ref{ceh}.iii) 
must be replaced with $\,\sy=2\,\mathrm{Re}\hskip2.7pt\fy$.]

To prove (\ref{ceh}), we establish, by induction on $\,k\ge0$, the equalities
\begin{equation}\label{jka}
\mathrm{i)}\hskip6pt\jk[b\hh]
=0\hs,\hskip15pt\mathrm{ii)}\hskip6pt\jk[\hs\sj\hh]
=\jk[\hh\dj\hh\ly\hh]\hs,\hskip15pt\mathrm{iii)}\hskip6pt
\jk[\hh\sy]=\jk[\fy\hh]
\end{equation}
of $\,k$-jets at $\,t=0\,$ (notation of (\ref{jet})). First, our ``initial 
condition'' $\,\sj(0)=0\,$ and the above definitions of 
$\,a,c,\sy,\ly,\fy,\xy\,$ and $\,\h\,$ give, by (\ref{fhz}),
\begin{equation}\label{jts}
\sj(0)\nh=\nh0\hh,\hskip5pta(0)\nh=\nh c\hh(0)\nh=\nh0\hh,\hskip5pt\sy(0)\nh
=\nh\ly(0)\nh=\nh\fy(0)\nh=\nh0\hh,\hskip5pt\xy(0)\nh=\nh\h(0)\nh=\nh0\hh.
\end{equation}
Thus, (\ref{jka}) holds for $\,k=0$. Assume now that $\,k\ge1\,$ and
\begin{equation}\label{jkm}
\mathrm{i)}\hskip6pt\jm[b\hh]=0\hs,\hskip15pt\mathrm{ii)}\hskip6pt
\jm[\hs\sj\hh]=\jm[\hh\dj\hh\ly\hh]\hs,\hskip15pt\mathrm{iii)}\hskip6pt
\jm[\hh\sy]=\jm[\fy\hh]\hh.
\end{equation}
As $\,\sj(0)=0\,$ and $\,\sy(0)=0\,$ (see (\ref{jts})), combining (\ref{jkz}), 
for $\,B(\sj,\td\sj)=\{\sj\nnh\cdot\nnh\td\sj\nh\}\,$ or 
$\,B(\sj,\sy)=\sj\sy$, with (\ref{jkm}.ii) -- (\ref{jkm}.iii), we obtain
\begin{equation}\label{jks}
\jk[\hh\{\sj\nnh\cdot\nnh\sj\nh\}\hh]
=\jk[\hh\{(\dj\hh\ly)\nnh\cdot\nnh(\dj\hh\ly)\nh\}\hh]\hh,
\hskip15pt\jk[\hh\sj\sy]=\jk[\hh(\dj\hh\ly)\fy\hh]\hh.
\end{equation}
However, by Lemma~\ref{neglg}(ii), (\ref{jts}) and (\ref{jkm}.i), the 
$\,\approx\,$ equivalences in Theorem~\ref{dsapp} imply equalities of 
$\,k$-jets at $\,t=0$. Since the left-hand sides in (\ref{sys}) all vanish due 
to our choice of $\,\f\nh,\h,\x,\y,\z,\p,\q\,$ and $\,\r$, (\ref{ddl}) and 
(\ref{zap}) thus yield
\begin{equation}\label{dkz}
\begin{array}{rl}
\mathrm{i)}&[\hh4\hh\dj\hh\{(\dj\hh\ly)\nnh\cdot\nnh(\dj\hh\ly)\}
+(\hh\dj\hh\ly)\hh\fy\hh]^{(k)}\\
&\hskip40pt=\,\dj\hh[\hs\U\nh\tya+\h\x\p\hh\tya+\V\nh\cya
+\W\nh\mya+\x\p\hh\tyb]^{(k)},\\
\mathrm{ii)}&[\hh\zy\hh]^{(k)}\nh=-\hh[\hh\y\hh\tyb]^{(k)}\hh,
\end{array}
\end{equation}
where $\,[\hs\ldots\hs]^{(k)}$ denotes the $\,k\hh$th derivative of 
$\,\ldots\,$ at $\,t=0$.

[In case (\ref{slr}), we leave formula (\ref{dkz}) unchanged, as its 
ingredients all refer to $\,\mathfrak{g}$, and modify (\ref{jka}) -- 
(\ref{jks}) as follows. In (\ref{jka}), (\ref{jkm}) and (\ref{jks}), 
$\,2\,\mathrm{Re}\hskip2.7pt\fy$ should appear, rather than $\,\fy$. In the 
first equality of (\ref{jks}), $\,\{\hskip2pt\cdot\hskip2pt\}_\bbR\s$ has to 
be used instead of $\,\{\hskip2pt\cdot\hskip2pt\}$, while the 
``multiplications'' should be replaced by their $\,\mathfrak{g}_\bbR\s$ 
counterparts (and $\,\fy\,$ by $\,2\,\mathrm{Re}\hskip2.7pt\fy$). Since 
$\,\dj\hh\ly\,$ is $\,\bbC$-bi\-lin\-e\-ar, Remarks~\ref{nrese} 
and~\ref{rehol}(v) allow us to rewrite the modified right-hand sides in 
(\ref{jks}) as 
$\,2\jk[\hh\mathrm{Re}\hs\{(\dj\hh\ly)\nnh\cdot\nnh(\dj\hh\ly)\nh\}]$ and 
$\,\jk[\hh(\dj\hh\ly)\fy\hh]$, with both operations now referring to the 
complex Lie algebra $\,\mathfrak{g}$.]

From (\ref{jks}) and (\ref{eqs}.i) we get $\,-\dz[\hh\sj\hh]^{(k)}\nh
=4\hh\dj\hs[\{(\dj\hh\ly)\nnh\cdot\nnh(\dj\hh\ly)\}]^{(k)}
+\hh[(\hh\dj\hh\ly)\hh\fy\hs]^{(k)}\nnh$, so that, by (\ref{dkz}.i), 
$\,\dz[\hh\sj\hh]^{(k)}\nh
=-\hh\dj\hh[\hs\hs\U\nh\tya+\h\x\p\hh\tya+\V\nh\cya
+\W\nh\mya+\x\p\hh\tyb]^{(k)}\nnh$. Next, (\ref{dio}) and (\ref{imo}) give 
$\,-\dz(\dj\hh\ly)=\dj\hh[\hh(\n^2\y+2\z)\hh\cya+(2\y+\z)\hh\mya-\h\y\hh\tya
-\y\hh\tyb]$, as $\,\dj\by=0$, cf.\ (\ref{fdd}.c), and $\,c\,$ may be 
replaced with $\,\h a+b$. Adding the former equality to 
$\,[\hs\ldots\hs]^{(k)}$ of the latter, and using equations (ii), (iv), (v) in 
(\ref{sys}), we obtain 
$\,\dz[\hh\sj-\dj\hh\ly\hh]^{(k)}\nh=-\hh[(\y+\x\p)\hh\tyb]^{(k)}\nnh$. Both 
sides here must vanish as a consequence of (\ref{opl}.b), since one lies in 
$\,\dz(\es)\,$ and the other, due to (\ref{krd}), in $\,\mathrm{Ker}\,\dz$. 
The resulting relation $\,[(\y+\x\p)\hh\tyb]^{(k)}\nh=0$, along with 
(\ref{jkm}.i) and the fact that, by (\ref{fhz}), $\,\y+\x\p\ne0\,$ at $\,t=0$, 
yields (\ref{jka}.i). Vanishing of $\,\dz[\hh\sj-\dj\hh\ly\hh]^{(k)}$ implies 
in turn that $\,[\hh\sj-\dj\hh\ly\hh]^{(k)}\nh=0$, as $\,\dz\,$ is injective 
on $\,\dz(\es)\,$ (see (\ref{opl}.c)), while 
$\,[\hh\sj-\dj\hh\ly\hh]^{(k)}\in\dz(\es)\,$ according to (\ref{eqs}.iii). Now 
(\ref{jka}.ii) follows from (\ref{jkm}.ii). Finally, (\ref{jka}.i) and 
(\ref{dkz}.ii) give $\,[\hh\zy\hh]^{(k)}\nh=0$. The definition of $\,\zy\,$ in 
Theorem~\ref{dsapp}, combined with the equality 
$\,[\hs\sy]^{(k)}\nh=[8\hh\{\dj\nnh\cdot\nnh(\dj\hh\ly)\}
+4\hh\{(\dj\hh\ly)\nnh\cdot\nnh(\dj\hh\ly)\}\hh]^{(k)}$ (obvious from 
(\ref{eqs}.ii) and (\ref{jka}.ii)), shows 
that $\,[\hs\sy]^{(k)}\nh=[\fy]^{(k)}\nnh$. Thus, (\ref{jka}.iii) is immediate 
from (\ref{jkm}.iii), completing the induction step, and proving (\ref{ceh}).

[In case (\ref{slr}), the last paragraph requires only minor changes. The 
right-hand side of the formula in the first line is $\,\bbC$-bi\-lin\-e\-ar, 
since so is $\,\dj\hh\ly$, even if $\,2\,\mathrm{Re}\hskip2.7pt\fy\,$ is used 
instead of $\,\fy$. (See Remarks~\ref{nrese} and~\ref{rehol}(v).) Thus, 
applying Remark~\ref{rehol}(viii) to the left-hand side, which now reads 
$\,-\dz_\bbR\s[\hh\sj\hh]^{(k)}\nnh$, we conclude that $\,[\hh\sj\hh]^{(k)}$ 
is $\,\bbC$-bi\-lin\-e\-ar, and Remark~\ref{rehol}(vii) allows us to replace 
$\,\dz_\bbR\s$ with $\,\dz$. Also, $\,2\,\mathrm{Re}\,$ must precede the 
right-hand sides of the last two equalities.]

Our fixed real-an\-a\-lyt\-ic curve $\,[\hs0,\vd)\ni t\mapsto\sj(t)\in\es\,$ 
with $\,\sj(0)=0$, lying entirely in $\,\hz\hs^{-\nh1}(0)$, thus satisfies 
relations (\ref{ceh}), in which all summands and factors, along with 
$\,\xy=(a,a)/\n$, are scalar or vec\-tor-val\-ued real-an\-a\-lyt\-ic 
functions of the variable $\,t$. It now follows that $\,\xy\,$ vanishes 
identically in $\,[\hs0,\vd)$. In fact, otherwise, for some 
$\,\vd\hh'\nh\in(0,\vd)$, one would have $\,\h\hs\xy\ne0\,$ at all 
$\,t\in(0,\vd\hh'\hh]\,$ and $\,\h^2\nnh/\xy\to0\,$ as $\,t\to0\,$ since, by 
(\ref{fhz}) -- (\ref{xto}), $\,\h\,$ is an analytic function of the variable 
$\,\xy\,$ with $\,\h=0\ne\hh d\h/d\hh\xy\,$ at $\,\xy=0$, while $\,\xy=0\,$ at 
$\,t=0\,$ according to (\ref{jts}). Consequently, $\,\h^2\nnh/\xy\,$ 
would be nonconstant. Yet it must be constant: by (\ref{rtl}), its values are 
all rational. The resulting contradiction proves that $\,\xy=0$ identically. 
(Note that (\ref{ceh}.i) implies (\ref{ahx}), 
and so Lemma~\ref{ratnl} may be applied here to $\,a=a(t)$, for any 
$\,t\in(0,\vd\hh'\hh]$.)

Since $\,\xy=0\,$ on $\,[\hs0,\vd)$, the functions $\,\h,\x,\y,\z,\p,\q,\r\,$ 
and $\,\f\nnh$, depending on $\,t$ via $\,\xy$, are all constant and their 
values are given by (\ref{fhz}). Thus, (\ref{ceh}.ii) amounts to 
$\,\sj=\dj\hh\eya$ for all $\,t$, as in (\ref{cdl}), with $\,a\nh^2\nh=0$, 
completing the proof.
\end{proof}
\begin{rem}\label{ssini}Here is why $\,(\sj,\sy)-\iz_a\w[\mathbf{v}]\,$ lies 
in the image of $\,\lz\,$ (as claimed in the lines following (\ref{jlk}), with 
$\,\sj,\sy,a,\nh\mathbf{v}\,$ depending on $\,t$). By (\ref{iav}), 
$\,(\sj,\sy)-\iz_a\w[\mathbf{v}]=(\sj-\dj\hh\ly,\hs\sy-\fy)$, 
while $\,\lz(\sj,\sy)=(\dz\sj,\sy)\,$ (see Remark~\ref{kssez}), and 
$\,\sj-\dj\hh\ly\in\dz(\es)$ in view of (\ref{eqs}.iii).
\end{rem}

\section{Proof of Theorem~\ref{isola}}\label{is}\checked
\setcounter{equation}{0}
We have the following consequence of Lemma~\ref{ractz}.
\begin{cor}\label{sffcl}{\smallit 
Under the assumptions listed at the beginning of Section~\/{\rm\ref{ra}}, the 
set\/ $\,\ec=\dj+\el\,$ of Ein\-stein connections in\/ $\,\mathfrak{g}\hh$, 
defined by\/ {\rm(\ref{cdl})}, contains all weak\-\hbox{ly\hh-}\hskip0ptEin\-stein connections 
sufficiently close to\/ $\,\dj$. This is the case for\/ 
$\,\,\mathfrak{g}_\bbR\s$ \hskip2.3ptin\/ {\rm(\ref{slr})} as well.
}
\end{cor}
\begin{proof}By (\ref{eud}.ii), weak\-\hbox{ly\hh-}\hskip0ptEin\-stein connections in 
$\,\mathfrak{g}\,$ form the set $\,\dj+\ez$, where $\,\ez=\hz\hs^{-\nh1}(0)$. 
Lemma~\ref{ractz} states that the hypotheses, and hence the conclusion, of 
Corollary~\ref{equal} are satisfied by $\,\el\,$ in (\ref{cdl}) and $\,\ez$, 
both of which are algebraic sets due to Theorem~\ref{mnres}(iv) and the 
definition (\ref{ths}.i) of $\,\hz$.
\end{proof}
Theorem~\ref{isola} is now immediate. First, Corollary~\ref{sffcl} applied to 
the triple $\,(\mathfrak{su}\hh(\d,\k),\bbR,\mathrm{i})\,$ in (\ref{lie}) 
implies that all Ein\-stein connections sufficiently close to $\,\dj\,$ lie in 
$\,\ec$. Secondly, for $\,\k=0\,$ and $\,\d=\n\ge3\,$ one has 
$\,\ec=\{\dj\}$, as $\,0\,$ is the only element $\,a\in\mathfrak{su}\hh(\n)\,$ 
with $\,a\nh^2\nh=0$. In other words, the Le\-\hbox{vi\hh-}\hskip0ptCi\-vi\-ta 
connection $\,\dj\,$ of the Kil\-ling form $\,\by\,$ is isolated among 
Le\-\hbox{vi\hh-}\hskip0ptCi\-vi\-ta connections of left-in\-var\-i\-ant 
Riemannian Ein\-stein metrics on $\,\mathrm{SU}\hh(\n)$. Rephrased in terms of 
metrics (see Remark~\ref{bijct}), this amounts to Theorem~\ref{isola}.

\section{Complex Witt bases}\label{cw}
\setcounter{equation}{0}
The following lemma uses Witt's theorem \cite[Chapter 13]{berger} to evaluate 
the dimensions of some Lie algebras and manifolds.
\begin{lem}\label{pshrm}{\smallit 
Let\/ $\,e_1\w,\dots,e_k\w$ be a basis of a totally null complex sub\-space\/ 
$\,\ev\,$ in the space\/ $\,\bbC^\n$ endowed with the standard 
ses\-qui\-lin\-e\-ar Her\-mit\-i\-an inner product\/ $\,\langle\,,\rangle\,$ 
of the sign pattern formed by $\,\d\,$ pluses and $\,\k\,$ minuses, where\/ 
$\,\d\ge\k\ge k$.
\begin{enumerate}
  \def\theenumi{{\rm\alph{enumi}}}
\item[{\rm(a)}] $\bbC^\n$ has a basis\/ 
$\,e_1\w,\dots,e_k\w,\hat e_1\w,\dots,\hat e_k\w,u_{2k+1}\w,\dots,u_n\w$ 
containing\/ $\,e_1\w,\dots,e_k\w$ such that, for all\/ $\,p,q,r,s\,$ in the 
appropriate ranges, $\,\langle e_r\w,e_s\w\rangle
=\langle\hat e_r\w,\hat e_s\w\rangle=\langle e_r\w,u_p\w\rangle
=\langle\hat e_r\w,u_p\w\rangle=0$, $\,\langle e_r\w,\hat e_s\w\rangle
=\delta_{rs}\w\hs$, $\,\langle u_p\w,u_p\w\rangle\in\{1,-1\}$, and\/ 
$\,\langle u_p\w,u_q\w\rangle=0\,$ if\/ $\,p\ne q$.
\item[{\rm(b)}] $\dim\mathfrak{h}=k^2$ for the Lie algebra\/ 
$\,\mathfrak{h}\,$ consisting of all\/ $\,a\in\mathfrak{su}\hh(\d,\k)\,$ 
with\/ $\,a(\bbC^\n)\subset\ev\nh$.
\item[{\rm(c)}] Every $\,a\in\mathfrak{su}\hh(\d,\k)\,$ with $\,a^2\nh=0\,$ 
is\/ $\,\mathrm{SU}\hh(\d,\k)$-con\-ju\-gate to\/ $\,ta$, \hskip1ptfor 
all\/ $\,t\in(0,\infty)$, so that\/ $\,0\,$ lies in the closure of the\/ 
$\,\mathrm{SU}\hh(\d,\k)$-orbit of\/ $\,a\hh$.
\item[{\rm(d)}] $\dim\mathfrak{k}=(\n-k)^2\nh+2k^2\nh-1\,$ for the Lie 
algebra\/ $\,\mathfrak{k}=\{a\in\mathfrak{su}\hh(\d,\k):a(\ev)\subset\ev\}$.
\item[{\rm(e)}] The set\/ $\,\nj_{\d,\k,k}$ of all\/ $\,k$-di\-men\-sion\-al 
totally null complex sub\-spaces of $\,\bbC^\n$ is a manifold, 
$\,\dim\nj_{\d,\k,k}=(2\n-3k)\hh k$, and the action of\/ 
$\,\mathrm{SU}\hh(\d,\k)\,$ on\/ $\,\nj_{\d,\k,k}$ is transitive.
\item[{\rm(f)}] Elements\/ $\,a\in\mathfrak{su}\hh(\d,\k)\,$ with\/ 
$\,a\nh^2\nh=0\,$ form exactly\/ $\,(\k+1)((\k+2)/2\,$ 
$\,\mathrm{SU}\hh(\d,\k)$-con\-ju\-gacy classes, classified by the values of\/ 
$\,\mathrm{rank}\hskip2.7pta\in\{0,1,\dots,\k\}\,$ and the positive index\/ 
$\,\mathrm{ind}_+\w a\,$ of the Her\-mit\-i\-an form\/ 
$\,\mathrm{i}\langle a(\,\cdot\,),\,\cdot\,\rangle$, with\/ 
$\,\mathrm{ind}_+\w a\in\{0,1,\dots,\mathrm{rank}\hskip2.7pta\}$.
\end{enumerate}
}
\end{lem}
\begin{proof}Since $\,\langle\,,\rangle\,$ descends to a (nondegenerate) 
Her\-mit\-i\-an inner product in $\,\ev^\perp\nnh/\hh\ev\nh$, we may choose 
$\,u_{2k+1}\w,\dots,u_n\w$ representing an orthonormal basis of 
$\,\ev^\perp\nnh/\hh\ev\nh$, and then define 
$\,\td e_1\w,\dots,\td e_k\w\in\bbC^\n$ by requiring 
$\,\langle\,\cdot\,,\td e_1\w\rangle\dots,
\langle\,\cdot\,,\td e_k\w\rangle\,$ to be the first $\,k\,$ elements of 
basis of $\,[\bbC^\n\hh]^*$ dual to a basis for which  
$\,e_1\w,\dots,e_k\w,u_{2k+1}\w,\dots,u_n\w$ are the initial $\,\n-k\,$ 
vectors. Setting $\,\hat e_r\w
=\td e_r\w-\sum_{s=1}^k\langle\td e_r\w,\td e_s\w\rangle e_s\w/2$, we 
obtain (a).

Skew-ad\-joint\-ness of $\,a\in\mathfrak{h}$, combined with the relation 
$\,a(\bbC^\n)\subset\ev\nh$, gives $\,\ev^\perp\nh\subset\mathrm{Ker}\,a$. 
Thus, elements $\,a\,$ of $\,\mathfrak{h}\,$ are characterized by
\begin{equation}\label{aer}
ae_r\w=au_p\w=0\hh,\hskip10pta\hat e_r
=\textstyle{\sum_{s=1}^k}X_{rs}\w e_s\w\hh,\hskip9pt\mathrm{where}\hskip6pt
X=[X_{rs}\w]\in\mathfrak{u}\hh(k)\hh,
\end{equation}
the requirement that the $\,k\times k\,$ matrix with the entries $\,X_{rs}\w$ 
be skew-Her\-mit\-i\-an expressing here skew-ad\-joint\-ness of 
$\,a\in\mathfrak{h}$. (As 
$\,a(\bbC^\n)\subset\ev\subset\ev^\perp\nh\subset\mathrm{Ker}\,a$, 
Remark~\ref{trtpr}(i) implies that $\,a\,$ is trace\-less.) Since 
$\,\dim\mathfrak{u}\hh(k)=k^2\nnh$, (b) follows.

Given $\,a\in\mathfrak{su}\hh(\d,\k)\,$ with $\,a\nh^2\nh=0$, 
the image $\,a(\bbC^\n)\,$ is totally null. Thus,
\begin{equation}\label{esp}
e_1\w,\dots,e_k\w\hs\mathrm{\ \ span\ \ }\,a(\bbC^\n)\,\mathrm{\ \ for\ \ }\,k
=\mathrm{rank}\hskip2.7pta\,\mathrm{\ \ and\ some\ basis\ as\ in\ (i),}
\end{equation}
which yields (\ref{aer}) with some $\,X$. Replacing each $\,e_r\w$ by 
$\,t^{-1/2}e_r\w$ and $\,\hat e_r\w$ by $\,t^{1/2}\hat e_r\w$ we obtain 
a new basis with (\ref{esp}), in which $\,X\,$ now represents $\,ta$, proving 
(c).

Next, $\,a\in\mathfrak{k}\,$ if and only if $\,a\,$ has, in the basis with 
(i), the block matrix form
\[
\left[\begin{array}{ccc}
X\,\,&\,\,Y&T\cr
0\,\,&\,\,H&-\nh Y^*\cr
0\,\,&\,\,0&-\nh X^*
\end{array}
\right].
\]
Here $\,T=-\nh T^*\nh\in\mathfrak{u}\hh(k)\,$ and 
$\,H=-\nh H^*\nh\in\mathfrak{u}\hh(\d-k,\k-k)$, while $\,(\hskip2.4pt)^*$ 
denotes three versions of the Her\-mit\-i\-an transpose. The zero 
sub\-matrices reflect the inclusions $\,a(\ev)\subset\ev\,$ and 
$\,a(\ev^\perp)\subset\ev^\perp$ (the latter due to skew-ad\-joint\-ness 
of $\,a$). This gives the formula in (d), with $\,X,Y,T,H\,$ and 
trace\-less\-ness of $\,a\,$ (which is a separate condition) contributing 
$\,2k^2,\,2(\n-2k)\hh k,\,k^2\nnh,(\n-2k)^2$ and $\,-1\,$ to the total.

Transitivity in (e) is obvious from (a). Thus, $\,\nj_{\d,\k,k}$ is a 
manifold. By (d), $\,\dim\nj_{\d,\k,k}=\dim\mathrm{SU}\hh(\d,\k)
-\dim\mathfrak{k}=\n^2\nnh-1-[(\n-k)^2\nh+2k^2\nh-1]$, as claimed in (e).

Finally, given $\,a\in\mathfrak{su}\hh(\d,\k)\,$ with $\,a\nh^2\nh=0$, 
the image $\,a(\bbC^\n)\,$ is totally null. We thus have (\ref{esp}), and 
hence (\ref{aer}) for some $\,X$. Another basis as in (\ref{esp}) arises when 
all $\,e_r\w$ and $\,\hat e_r\w$ are replaced by $\,e_r'$ and 
$\,\hat e_r'$, where $\,e_r'=\sum_{s=1}^kW_{\nh rs}\w e_s\w$ and 
$\,\hat e_r'=\sum_{s=1}^kZ_{rs}\w\hat e_s\w$ for any invertible 
complex $\,k\times k\,$ matrix $\,Z\,$ and its inverse conjugate transpose 
$\,W\nh$. The matrix playing the role of $\,X\,$ in (\ref{aer}) for $\,a\,$ 
and the new basis is $\,ZXZ^*\nnh$, where $\,Z^*\nh=W^{-1}$ is the conjugate 
transpose of $\,Z$. Since $\,X\in\mathfrak{u}\hh(k)\,$ is orthonormally 
di\-ag\-o\-nal\-izable, a suitable choice of $\,Z\in\mathfrak{u}\hh(k)\,$ will 
render $\,ZXZ^*$ diagonal, with imaginary entries. Furthermore, rescaling each 
$\,e_r\w$ and $\,\hat e_r\w$ as in the line following (\ref{esp}), with 
$\,t>0\,$ possibly depending on $\,r$, allows us to assume that the imaginary 
diagonal entries are all $\,\pm\hh\mathrm{i}\,$ or $\,0$. Combined with 
transitivity in (e), this gives (f).
\end{proof}
\begin{rem}\label{canon}The analog of (f) in Lemma~\ref{pshrm} for 
$\,\mathrm{SL}\hh(\n,\bbF)\,$ (where $\,\bbF\,$ is $\,\bbR,\bbC$ or $\,\bbH$), 
rather than $\,\mathrm{SU}\hh(\d,\k)$, is simpler: elements $\,a\,$ of 
$\,\mathfrak{sl}\hh(\n,\bbF)\,$ with $\,a\nh^2\nh=0\,$ form $\,[\n/2]+1\,$ 
$\,\mathrm{SL}\hh(\n,\bbF)$-con\-ju\-gacy classes, classified by the value of 
$\,\mathrm{rank}\hskip2.7pta$, ranging over $\,\{0,1,\dots,\k\}$. In fact, if  
$\,k=\mathrm{rank}\hskip2.7pta$, a basis
\begin{equation}\label{aeo}
ae_1\w,\dots,ae_k\w,e_1\w,\dots,e_k\w,u_{2k+1}\w,\dots,u_n\w\hh,
\end{equation}
where $\,ae_1\w,\dots,ae_k\w,u_{2k+1}\w,\dots,u_n\w$ span $\,\mathrm{Ker}\,a$, 
gives $\,a\,$ a canonical matrix form.
\end{rem}
\begin{rem}\label{scale}Assertion (c) of Lemma~\ref{pshrm} remains valid if 
one replaces $\,\mathrm{SU}\hh(\d,\k)$ with any of the other groups in 
(\ref{gsu}). To see this, one may use, instead of (\ref{aeo}), the basis 
$\,tae_1\w,\dots,tae_k\w,e_1\w,\dots,e_k\w,u_{2k+1}\w,\dots,u_n\w$.

Condition (c) in Lemma~\ref{pshrm} is also an immediate consequence of (f) 
(and similarly for $\,\mathrm{SU}\hh(\d,\k)$). We state it separately for easy 
reference.
\end{rem}

\section{Proofs of Theorem~\ref{mnres}, second part, and 
Theorem~\ref{slcre}}\label{pa}
\setcounter{equation}{0}
By (\ref{inc}), $\,\ec\subset\ee\subset\ew\subset\dj+\es$, which reduces 
proving Theorem~\ref{mnres}(v) to establishing relative openness of $\,\ec\,$ 
in $\,\ew\,$ or, equivalently, showing that every point $\,\dj+\hs\sj\,$ of 
$\,\ec\,$ has a neighborhood in the af\-fine space $\,\dj+\es\,$ which does 
not intersect $\,\ew\smallsetminus\ec$. To this end let us suppose that, on 
the contrary, some $\,\dj+\hs\sj\in\ec$ is the limit as $\,k\to\infty\,$ of 
a sequence $\,\dj+\hs\sj_k$ in $\,\ew\smallsetminus\ec$. By 
Remark~\ref{scale}, any neighborhood of $\,\dj$ contains an ad\-joint-ac\-tion 
image of $\,\dj+\hs\sj\,$ and, with it, the images of $\,\dj+\hs\sj_k$ for 
all large $\,k$. Thus, some sequence in $\,\ew\smallsetminus\ec\,$ converges to 
$\,\dj$, which contradicts Corollary~\ref{sffcl}.

The set $\,\ec$, being connected (since so is the cone 
$\,\ep\nh=\{a\in\mathfrak{g}:a\nh^2\nh=0\}$), as well as closed and relatively 
open in both $\,\ee\,$ and $\,\ew\,$ (by (iv) -- (v)), must be a connected 
component of both, which proves (vi). Assertion (viii) is in turn immediate 
from Lemma~\ref{pshrm}(f) and Remark~\ref{canon}.

All of the above remains valid for the underlying real Lie algebra of 
$\,\mathfrak{sl}\hh(\n,\bbC)$.

The next result trivially implies (vii) in Theorem~\ref{mnres}, as well as 
Theorem~\ref{slcre}; for the latter, we also use 
Remark~\ref{rehol}(ii)\hh-\hh(iv), the final clause of Corollary~\ref{sffcl}, 
and the preceding one\hh-line paragraph.

In the remainder of this section $\,(\mathfrak{g},\bbF\nnh,\ve)\,$ denotes one 
of the triples (\ref{lie}), for $\,\n=\d+\k\ge3$. All manifolds and mappings 
are assumed $\,\bbF$-an\-a\-lyt\-ic, with the term `real/com\-plex' 
indicating the appropriate choice between $\,\bbF=\bbR\,$ or $\,\bbF=\bbC$. 
The symbol $\,[\hskip5pt]$ stands for the integer part, and 
$\,\mathrm{Gr}_k^\n$ for the real/com\-plex Grass\-mann manifold of all 
\hbox{$\,k$-}\hskip0ptdi\-men\-sion\-al real/com\-plex vector sub\-spaces of 
$\,\bbF^\n\nnh$.

The polynomial bijective correspondence in Theorem~\ref{mnres}(ii) reduces the 
description of $\,\ec\,$ to determining the structure of the cone 
$\,\ep\nh=\{a\in\mathfrak{g}:a\nh^2\nh=0\}$. We achieve the latter by 
providing the following nonsingular model for $\,\ep\nh$.
\begin{thm}\label{model}{\smallit
There exist a connected\/ $\,\bbF$-an\-a\-lyt\-ic manifold\/ $\,\em\,$ with an 
open dense subset\/ $\,\em'$, and an\/ $\,\bbF$-an\-a\-lyt\-ic mapping\/ 
$\,\qz:\em\to\mathfrak{g}$, such that\/ $\,\qz\hh(\em)=\ep\,$ and\/ $\,\qz\,$ 
sends\/ $\,\em'$ dif\-feo\-mor\-phi\-cal\-ly onto a real$/\hn$com\-plex 
sub\-man\-i\-fold of\/ $\,\mathfrak{g}\,$ contained in\/ $\,\ep\nh$, while\/ 
$\,\dimf\em=[\n^2\nnh/2]\,$ for\/ $\,\mathfrak{g}=\mathrm{sl}\hh(\n,\bbF)\,$ 
and\/ $\,\dimr\em=2\hs\d\k$, if\/ $\,\mathfrak{g}=\mathrm{su}\hh(\d,\k)$, or\/ 
$\,\dimr\em=4[\n^2\nnh/8]$, if\/ $\,\mathfrak{g}=\mathrm{sl}\hh(\n/2,\bbH)$.
}
\end{thm}
\begin{proof}For triples other than $\,\mathfrak{sl}\hh(\n/2,\bbH),\bbR,1)$, we 
define the integer $\,k\,$ by $\,k=[\n/2]\,$ if 
$\,\mathfrak{g}=\mathrm{sl}\hh(\n,\bbF)\,$ and $\,k=\k\,$ if 
$\,\mathfrak{g}=\mathrm{su}\hh(\d,\k)$. As $\,\k\le\d$, in the latter case 
$\,k\,$ is the maximum dimension of a totally null complex sub\-space in 
$\,\bbC^\n$ endowed with the ses\-qui\-lin\-e\-ar inner product mentioned in 
Lemma~\ref{pshrm}.

Let $\,\eb\,$ denote the compact connected $\,\bbF$-an\-a\-lyt\-ic manifold 
formed by all 
$\,(\ev,\ev\hh'\hh)\in\mathrm{Gr}_k^\n\times\mathrm{Gr}_{\n-k}^\n$ such that 
$\,\ev\subset\ev\hh'$ and, for $\,\mathfrak{g}=\mathrm{su}\hh(\d,\k)\,$ only, 
$\,\ev\hh'$ is the orthogonal complement of the complex totally null 
sub\-space $\,\ev$. Thus, when $\,\mathfrak{g}=\mathrm{sl}\hh(\n,\bbF)\,$ and 
$\,\n\,$ is odd, $\,\eb\,$ is an $\,\bbF\mathrm{P}^{\n-k-1}$ bundle 
over $\,\mathrm{Gr}_k^\n$, with the bundle projection 
$\,(\ev,\ev\hh'\hh)\mapsto\ev\,$ and the fibre 
$\,\mathrm{P}(\ev\hh'\nnh/\ev)\,$ over $\,\ev\in\mathrm{Gr}_k^\n$. For  
$\,\mathfrak{g}=\mathrm{sl}\hh(\n,\bbF)$ and even $\,\n$, we may identify 
$\,\eb\,$ with $\,\mathrm{Gr}_k^\n$, as $\,\ev=\ev\hh'\nnh$. Similarly, 
$\,\ev\,$ uniquely determines $\,\ev\hh'$ if 
$\,\mathfrak{g}=\mathrm{su}\hh(\d,\k)$, which leads to the identification 
$\,\eb=\nj_{\d,\k,\k}$, with $\,\dim\nj_{\d,\k,k}=(2\d\nnh-\k)\k\,$ (as 
$\,\n=\d+\k$, cf.\ Lemma~\ref{pshrm}(e)). Consequently, depending on whether 
$\,\mathfrak{g}=\mathrm{sl}\hh(\n,\bbF)\,$ and $\,\n\,$ is even, or 
$\,\mathfrak{g}=\mathrm{sl}\hh(\n,\bbF)\,$ and $\,\n\,$ is odd, or 
$\,\mathfrak{g}=\mathrm{su}\hh(\d,\k)$,
\begin{equation}\label{dim}
\dimf\eb\nh=k^2\nnh,\mathrm{\ or\ }\,\dimf\eb\nh=(k+2)k\hh,
\mathrm{\ or\ }\,\dimr\eb\nh=(2\d\nnh-\k)\k\hh.
\end{equation}
We denote by $\,\em\,$ the total space of the real/com\-plex vector bundle 
over $\,\eb\,$ with the fibre $\,\ef\,$ over any $\,(\ev,\ev\hh'\hh)\in\eb\,$ 
consisting of all $\,a\in\mathfrak{g}\,$ such that the image of $\,a$ as a 
linear en\-do\-mor\-phism of $\,\bbR^\n$ or $\,\bbC^\n$ is contained in 
$\,\ev$, and its kernel contains $\,\ev\hh'\nnh$. As the inclusion 
$\,\ev\subset\ev\hh'$ then gives $\,a\nh^2\nh=0$, an obvious surjective 
mapping $\,\qz:\em\to\ep\,$ is defined by requiring the 
$\,\qz\hskip1pt$-im\-age of a fibre element $\,a\,$ at any 
$\,(\ev,\ev\hh'\hh)\in\eb\,$ to be $\,a\,$ itself. In all cases, 
$\,\dimf\ef\nnh=k^2\nnh$. In fact, if 
$\,\mathfrak{g}=\mathrm{sl}\hh(\n,\bbF)$, one may view $\,\ef\,$ as the space 
of all linear operators $\,\bbF^\n\hskip-2pt/\ev\hh'\nh\to\ev\nh$, while for 
$\,\mathfrak{g}=\mathrm{su}\hh(\d,\k)$ we can use Lemma~\ref{pshrm}(b).

The vector bundle just described is an $\,\bbF$-an\-a\-lyt\-ic sub\-bundle of 
the trivial bundle with the fibre $\,\mathfrak{g}\,$ over $\,\eb$, since it 
is the kernel of the obvious $\,\bbF$-an\-a\-lyt\-ic bundle morphism, the rank 
of which is constant as $\,\dimf\ef\nnh=k^2\nnh$. Thus, the manifold $\,\em\,$ 
and the mapping $\,\qz:\em\to\mathfrak{g}\,$ are both $\,\bbF$-an\-a\-lyt\-ic, 
while the dimension clause follows if one adds $\,\dimf\ef\nnh=k^2$ to the 
dimensions in (\ref{dim}).

Let us now choose the open subset $\,\em'$ of $\,\em\,$ to be the total space 
of a (non-vec\-tor) sub\-bundle of the bundle $\,\em$, the fibre of which over 
any $\,(\ev,\ev\hh'\hh)\in\eb$ is the subset of the fibre of $\,\em\,$ 
formed by those $\,a\,$ that, in addition, have the maximum rank $\,k$. The 
dif\-feo\-mor\-phic property of $\,\qz\,$ on $\,\em'$ is then obvious, with 
the inverse mapping $\,\qz\hh(\em')\to\em'$ sending any max\-i\-mum\hh-rank 
$\,a\,$ to the fibre element $\,a\,$ at the point $\,(\ev,\ev\hh'\hh)\in\eb\,$ 
which is the im\-age\hh-ker\-nel pair of $\,a$.

For the remaining triple $\,\mathfrak{sl}\hh(\n/2,\bbH),\bbR,1)$, one sets 
$\,k=[\n/4]\,$ and repeats the above definitions $\,\em\,$ and $\,\ef\,$ 
verbatim, while that of $\,\eb\,$ is modified: the constituents 
$\,\ev,\ev\hh'$ of any $\,(\ev,\ev\hh'\hh)\in\eb\,$ are, in addition, required 
to be qua\-ter\-ni\-on\-ic sub\-spaces of 
$\,\bbH^{\n/2}\nh\approx\hs\bbC^\n\nnh$. The argument used to prove 
(\ref{dim}), with $\,\bbF\,$ and $\,\n\,$ replaced by $\,\bbH\,$ and $\,\n/2$, 
now shows that the ``qua\-ter\-ni\-on\-ic'' dimension of $\,\eb\,$ equals 
$\,k^2$ or $\,(k+2)k$, depending on whether $\,\n/2\,$ is even or odd. Adding 
to either value $\,\dimh\ef\nnh=k^2\nnh$, we obtain $\,[\n^2\nnh/8]$, so that 
$\,\dimr\em=4[\n^2\nnh/8]$. Finally, the discussion of the last paragraph 
applies to this case without any changes.
\end{proof}

\section{Non\hh-\hn Ein\-stein examples}\label{ne}
\setcounter{equation}{0}
Not all weak\-\hbox{ly\hh-}\hskip0ptEin\-stein connections arising in 
Theorem~\ref{xmpls} are Ein\-stein connections. For instance, the nonuple 
$\,(\xy,\f\nh,\h,\x,\y,\z,\p,\q,\r)$, where
\begin{equation}\label{nsl}
\begin{array}{l}
\xy=(\z+1)/\nh\z^2,\hskip9pt\h=0\hh,\hskip9pt\x=1\hh,\hskip9pt\y=-z\hh,\hskip9pt\p
=(\z+2)/\nh\z\hh,\\
\q=-\hh\n^2(\z+2)\hh,\hskip9pt\r=\z+2\hh,\hskip9pt\f=(\z+1)(2-\z)/\nh\z^2,
\end{array}
\end{equation}
will satisfy (\ref{sys}) with 
$\,(\K,\U,\V,\W)=(-2(\z+1)/\nh\z,0,(2-\n^2)\z,-\z)\,$ in (\ref{uvw}), if one 
chooses $\,\z=-\hh(\n^2\nnh-2)/(\n^2\nnh-1)$. In fact, all equalities in 
(\ref{sys}) except (\ref{sys}.iv), and the values of $\,\K,\U,\W\nnh$, are 
easily verified by treating $\,\z\,$ as arbitrary and substituting 
$\,\n^2\y+2\z\,$ for $\,\V\,$ in (\ref{sys}.vii). We then also have 
$\,\K+2=-2/\nh\z$ and $\,\z+1=1/(\n^2\nnh-1)$. Since $\,\h=0<\xy\,$ in 
(\ref{nsl}), the assumption (\ref{ahx}) in Theorem~\ref{xmpls} amounts to 
requiring that $\,(\ve\hh a)^2$ be a specific positive multiple of 
$\,\mathrm{Id}$. Such $\,a\in\mathfrak{g}\,$ exists for any triple (\ref{lie}) 
as long as $\,\n=\d+\k\,$ is even.

Consequently, (\ref{nsl}) with $\,\z=-\hh(\n^2\nnh-2)/(\n^2\nnh-1)\,$ and 
any even $\,\n\ge3$ gives rise, via Theorem~\ref{xmpls}, to examples of 
uni\-mod\-u\-lar connections $\,\nabla\,$ such that 
$\,\nabla\nnh\rho\nnh^\nabla\nnh=0$. They are not, however, Ein\-stein 
connections, since $\,-\hh4\hh\rho\nnh^\nabla\nnh=\by+\fy\,$ is degenerate and 
nonzero. Namely, as explained below, $\,a\,$ lies in the null\-space of 
$\,\by+\fy\,$ (and $\,a\ne0\,$ by (\ref{ahx}) with $\,\h=0\ne\xy$), while 
$\,\langle\by,\by+\fy\rangle\ne0$.

Specifically, since $\,\tya(a,\,\cdot\,)=\langle2c,\,\cdot\,\rangle$, 
$\,\cya(a,\,\cdot\,)=\langle\xy a,\,\cdot\,\rangle$, and 
$\,\mya(a,\,\cdot\,)=\langle d,\,\cdot\,\rangle$ by (\ref{cea}), 
using the three steps of Remark~\ref{trsln} to evaluate traces, we get
\begin{equation}\label{taa}
\begin{array}{l}
\mathrm{the\ images\ of\ }\,a\,\mathrm{\ under\ the\ en\-do\-mor\-phisms\ 
(\ref{ava})\ are\ }\,\,2c,\hs\,a,\hs\,\xy a\hs,\,d\hh,\\
\mathrm{and\ the\ traces\ of\ the\ en\-do\-mor\-phisms\ (\ref{ava})\ equal\ 
}\,0,\,n^2\nnh-1,\,\xy,\,-\hs\xy\hh.
\end{array}
\end{equation}
Therefore $\,a$, if nonzero, is an eigen\-vector, for the eigen\-value 
$\,2\h\p+\xy\q+(\h^2\nh+\xy)\hh\r+\f\nh+1$, of the linear en\-do\-mor\-phism 
of $\,\mathfrak{g}\,$ corresponding to  $\,-\hh4\hh\rho\nnh^\nabla$ via 
(\ref{sgv}), while the trace of that en\-do\-mor\-phism is 
$\,\langle\by,\by+\fy\rangle\,=\,(n^2\nnh-1)\hh\q\,+\,(\r-\f\hh)\hs\xy$, 
with $\,\langle\,,\rangle\,$ as in (\ref{inp}.i). The claim concerning the 
eigen\-value is obvious here from (\ref{taa}), \hbox{along with} the 
equalities $\hs\,b=0\,\hs$ and $\hs\,d\hs=(\h^2\nh+\xy)\hh a$, immediate from 
the assumption $\,c=\h a$ (that is, (\ref{ahx})) and (\ref{bec}) or, 
respectively, from (\ref{dsh}) and Lemma~\ref{neglg}(i).

The above conclusions about the eigen\-value and trace apply to 
$\,-\hh4\hh\rho\nnh^\nabla\nnh=\by+\fy$ for all connections 
$\,\nabla\,$ arising from Theorem~\ref{xmpls}. In the special case 
(\ref{nsl}), with $\,\z\,$ treated as arbitrary, the eigen\-value 
equals $\,(\z+2)[\hh1-(n^2\nnh-1)\hh(\z+1)]/\nh\z^2\nnh$, and so it is 
$\,0\,$ due to our choice of $\,\z$. Similarly, the trace is 
$\,(n^2\nnh-1)-(n^2\nnh+1)\hh(\z+1)\ne0$.

\subsubsection*{Acknowledgments}
The authors thank Wolfgang Ziller for helpful comments.

A part of A.\ Derdzinski's work on the paper was done at the Erwin 
Schr\"o\-ding\-er International Institute for Mathematical Physics (ESI), 
Vienna, in July 2011, during the first week of the Workshop on Cartan 
Connections, Geometry of Homogeneous Spaces, and Dynamics, organized by 
Andreas \v Cap, Charles Frances and Karin Melnick.

\'S.\ R.\ Gal was partially supported by Polish MNiSW grant N N201 541738.

\end{document}